\documentclass[a4paper,10pt]{scrartcl}

\usepackage[english] {babel}
\usepackage[utf8]{inputenc}
\usepackage{amsfonts} 
\usepackage{amssymb}
\usepackage{amsthm}
\usepackage{amsmath}
\usepackage{amsxtra}
\usepackage{a4wide}
\usepackage{mathrsfs}
\usepackage{exscale}

\usepackage{diagbox}

\usepackage[numbers,square]{natbib}

\usepackage{bm}
\usepackage{relsize}

\usepackage{subcaption}

\usepackage{appendix}

\usepackage{hyperref}
\usepackage{cleveref}
\usepackage{graphicx}
\usepackage{caption}
\usepackage{amssymb}
\usepackage{dsfont}
\usepackage[normalem]{ulem}
\usepackage{tikz}
\usepackage{mathtools}

\usepackage[obeyDraft]{todonotes}
\usepackage{todonotes}

\usepackage{multirow}

\newcommand{\Prob}[2]{\mathbb P_{#1} \left(#2\right)}

\newcommand{\E  }[2]{\mathbb E_{#1} #2}
\newcommand{\Exp}[1]{\mathbb E\left[#1\right]}

\newcommand{\statPoi}{T_{t,n,d}}
\newcommand{\reals}{\mathbb{R}}

\newcommand{\psf}{h}

\newcommand*\diff{\mathop{}\!\mathrm{d}}

\newcommand{\veps}{\varepsilon}
\newcommand{\nat}{\mathbb{N}}
\newcommand{\nato}{\mathbb{N}_0}

\newcommand{\ratioSum}{\rho}

\newcommand{\indicator}{\mathbf{1}}

\newcommand{\domain}{\mathcal{D}_{\veps,i}}

\newcommand{\Var}[1]{\mathbb V\left[#1\right]}
\newcommand{\V  }[2]{\mathbb V_{#1} #2}

\newcommand{\new}[1]{\textcolor{black}{#1}}
\newcommand{\newSupp}[1]{\new{#1}}

\DeclareMathOperator*{\erf}{erf}
\newcommand{\Gauss}{\mathcal{N}}
\DeclareMathOperator{\Poi}{Poi}
\DeclareMathOperator{\Bin}{Bin}

\DeclareMathOperator{\fwhm}{FWHM}
\DeclareMathOperator{\na}{NA}

\newtheorem{theorem}{Theorem}[section] 
\newtheorem{lemma}[theorem]{Lemma}

\theoremstyle{definition} 
\newtheorem{definition}[theorem]{Definition}

\newtheorem{corollary}[theorem]{Corollary}

\newtheorem{proposition}[theorem]{Proposition}
\newtheorem{remark}[theorem]{Remark}
\newtheorem{assumption}[theorem]{Assumption}

\begin{document}
	
	\begin{center}
		\begin{minipage}{.8\textwidth}
			\centering 
			\LARGE What is resolution? A statistical minimax testing perspective on super-resolution microscopy\\[0.5cm]
			
			\normalsize
			\textsc{Gytis Kulaitis}\\[0.1cm]
			\verb+gytis.kulaitis@mathematik.uni-goettingen.de+\\
			Institute for Mathematical Stochastics, University of G\"ottingen\\[0.3cm]
			
			\textsc{Axel Munk}\\[0.1cm]
			\verb+munk@math.uni-goettingen.de+\\
			Institute for Mathematical Stochastics, University of G\"ottingen\\
			and\\
			Felix Bernstein Institute for Mathematical Statistics in the Bioscience, University of G\"ottingen\\
			and\\
			Max Planck Institute for Biophysical Chemistry, G\"ottingen, Germany\\[0.3cm]
			
			\textsc{Frank Werner}\footnotemark[1]\\[0.1cm]
			\verb+frank.werner@mathematik.uni-wuerzburg.de+\\
			Institute of Mathematics, University of W\"urzburg
		\end{minipage}
	\end{center}
	
	\footnotetext[1]{Corresponding author}
	
	\begin{abstract}
		As a general rule of thumb the resolution of a light microscope (i.e. the ability to discern objects) is predominantly described by the full width at half maximum (FWHM) of its point spread function (psf)---the diameter of the blurring density at half of its maximum. Classical wave optics suggests a linear relationship between FWHM and resolution also manifested in the well known Abbe and Rayleigh criteria, dating back to the end of 19th century. However, during the last two decades conventional light microscopy has undergone a shift from microscopic scales to nanoscales. This increase in resolution comes with the need to incorporate the random nature of observations (light photons) and challenges the classical view of discernability, as we argue in this paper. Instead, we suggest a statistical description of resolution obtained from such random data. Our notion of discernability is based on statistical testing whether one or two objects with the same total intensity are present. For Poisson measurements we get linear dependence of the (minimax) detection boundary on the FWHM, whereas for a homogeneous Gaussian model the dependence of resolution is nonlinear. Hence, at small physical scales modeling by homogeneous gaussians is inadequate, although often implicitly assumed in many reconstruction algorithms. In contrast, the Poisson model and its variance stabilized Gaussian approximation seem to provide a statistically sound description of resolution at the nanoscale. Our theory is also applicable to other imaging setups, such as telescopes.
	\end{abstract}

	
	\textit{AMS classification numbers:} 60F05, 62B10, 62C20, 91B06, 94A12, 94A13, 94A15, 94A17.  \\[0.3cm]

	\date{\today}

	\section{Introduction}
	
	\subsection{Lens optics and diffraction}\label{sec:diffraction}
	
	According to geometrical optics, an ideal light microscope would be able to distinguish two points in space being arbitrary close. However, in 1873 Abbe \citep{Abbe1873} formulated what later became known as the \textit{Abbe diffraction limit} (\Cref{fig:airyFwhmAbbeRayleigh}C): Two points can be resolved only if their distance $d$ in space is at least
	\begin{equation}\label{eq:abbeLimit}
	d=\frac{\lambda}{2 \na},
	\end{equation}
	where $\lambda$ is the wavelength of incoming light and $\na$ is the numerical aperture of the microscope. The numerical aperture is equal to the product of the refractive index of the medium ($1$ for vacuum, $\approx 1$ for air) and the sine of one-half of the angle of the cone of light that can enter the microscope. Abbe \citep{Abbe1873} argued that diffraction and interference of light have to be taken into account when distances in the order of the wavelength of the illumination light are considered (see \citep{Cremer2013} and references therein for a comprehensive account). This paradigm has limited light microscopy for more than a century until the ground-breaking advent of super-resolution microscopy \citep{Hell1994}, see \Cref{sec:fromMicroToNano}. For the following, it is beneficial to recall the basic physics tailored to our needs, see also \citep{Aspelmeier2015}.
	
	
	
	
	Given a specimen under the microscope $f$, due to diffraction (and the resulting interference, see \Cref{fig:airyFwhmAbbeRayleigh}A and B) the imaging system causes a blur so that we do not simply observe an $M$ times magnified image of $f$. This blur is usually obtained by calculating analytically or estimating from data the blur pattern of a single point---the point spread function (psf) $h$. For an incoherent imaging system, e.g. a fluorescence microscope, using Huygens's principle\new{, see e.g. \citep[Section 8.2]{Born1999},} the image of the specimen then can be obtained by summing up the blurred images of the points constituting the sample. This results in a convolution
	\begin{equation}\label{eq:convolution}
	g\left(x\right) = \int_O h\left(x - Mx'\right) f(x')\diff x',
	\end{equation}
	where $O$ is the space containing the specimen---the object space---and $f:O\to\reals$. The space consisting of magnified points $Mx'$ is called the image space $I$ and $g:I\to\reals$ is the image of the specimen.
	
	If the microscope was perfect and there was no blur, then the psf $h$ would simply correspond to a delta function $\delta_{x-Mx'}$, so that $g(x) = f(x/M)$. In general, the psf $h$ can be computed explicitly by scalar diffraction theory. Under the assumption of circular aperture and using the paraxial approximation \citep{Born1999,Orfanidis2016}, $h$ becomes proportional to the Airy pattern \citep{Airy1835} (\Cref{fig:airyFwhmAbbeRayleigh}A)
	\begin{equation}\label{eq:airy}
	h\left(x\right) \propto \left| 2 A \left(\frac{2\pi}{\lambda} \frac{\na}{M} \left\Vert x\right\Vert_2\right)\right|^2,
	\end{equation}
	where $\lambda$ is the illumination wavelength and $||\cdot||_2$ is the Euclidean norm. The function $A$ in \eqref{eq:airy} is given by $A(u) = J_1(u)/u$, where $J_1$ is the Bessel function of the first kind. 
	
	Independently of Abbe, Lord Rayleigh formulated in 1879  a resolution criterion for spectroscopes \citep{Rayleigh1879}. Applied to microscopes Rayleigh's criterion reads that two point sources at $x_1$ and $x_2$ having equal intensity can just be resolved if the central maximum of the first psf centered at $x_1$ coincides with the first minimum of the second psf. The first zero of the Bessel function $J_1$ is at $x\approx 3.8317$ and hence $x/2\pi\approx 0.6098$. Thus, in the case of circular aperture the Rayleigh criterion is given by
	\begin{equation}\label{eq:rayleighLimit}
	d= 0.61\, \frac{\lambda}{\na}.
	\end{equation}
	Note that this is slightly more conservative than Abbe's result \eqref{eq:abbeLimit}. See \Cref{fig:airyFwhmAbbeRayleigh}C and D for a comparison.
	
	The resolution criteria \eqref{eq:abbeLimit} and \eqref{eq:rayleighLimit} can be understood in terms of the full width at half maximum (FWHM) of the (effective) psf (see \Cref{fig:airyFwhmAbbeRayleigh}B, where $\fwhm = \left|x_2 - x_1\right|$). More precisely, the FWHM is defined as the width of the psf when its intensity is half of its maximal intensity. The ability to state both Abbe and Rayleigh criteria in terms of the FWHM has lead to the common understanding that two point sources in space can be resolved by a light microscope as soon as their distance is larger than roughly the FWHM of the psf $h$. Usage of the FWHM as a resolution criterion dates back to at least 1927 \citep{Houston1927} and is still popular today \citep{Egner2020}. The FWHM criterion is particularly well-suited if the psf can be approximated by a Gaussian kernel as shown in \Cref{fig:airyFwhmAbbeRayleigh}B, since this function does not have any local minima. In fact, the approximation of the psf by a Gaussian is very common and sufficient for many practical purposes, see e.g. \citep{Diezmann2017}. 
	For an Airy pattern \eqref{eq:airy}, the FWHM can be computed by first computing the FWHM of $A(u)^2 = (J_1(u)/u)^2$, which---due to $\max_u A \left(u\right)^2 = A(0)^2 = 1$---is determined by the solution of $J_1(u) = \pm u / \sqrt{2}$. This yields an FWHM of $3.232$ for $A(u)^2$, and hence taking the additional scaling factors in \eqref{eq:airy} into account together with $Mx' = x$,
	we get the FWHM resolution criterion in its
	most common form
	\begin{equation}\label{eq:fwhmLimit}
	d = \fwhm = 0.51\, \frac{\lambda}{\na}.
	\end{equation}
	Thus, the FWHM limit is almost equal to the Abbe resolution limit \eqref{eq:abbeLimit} and somewhat below the Rayleigh resolution limit \eqref{eq:rayleighLimit}.
	
	\new{Note that all three resolution criteria postulate a linear dependency of the resolution on the FWHM, which is in good agreement with experimental results, see e.g. \cite{Egner2020}.}
	
	We mention that due to their generality, the above resolution criteria are not confined to microscopes and can also be applied to telescopes \citep{Acuna1997, Bertero2009}, or other imaging devices, in general. We stress that there are many other resolution criteria such as the recently popularized Fourier ring correlation \citep{Banterle2013}, which can be expressed in terms of the FWHM as well. 
	\new{Hence, in summary, the FWHM can be viewed as a simple but very informative number to quantify optical resolution.}
	
	\begin{figure}
		\centering
		\includegraphics[width=\textwidth]{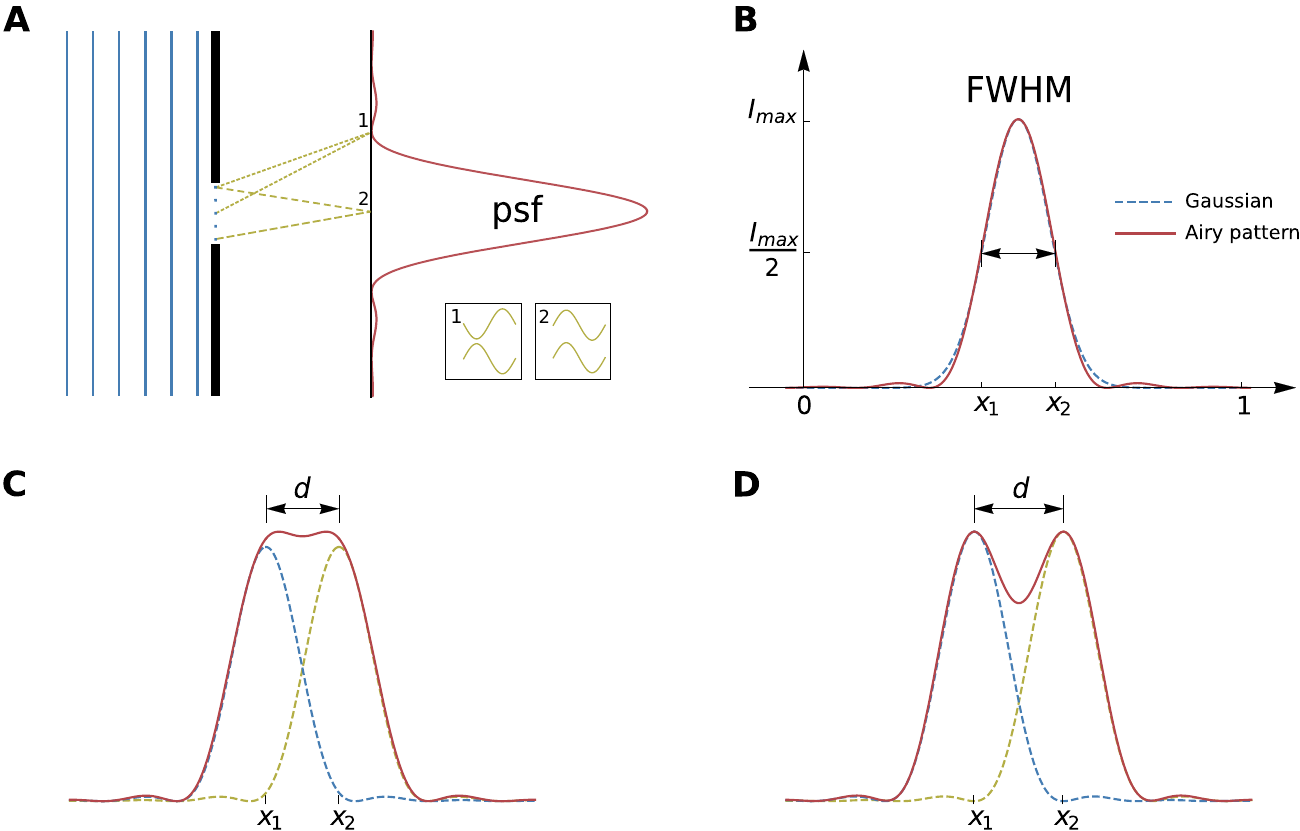}
		\caption{(A) 1D view of a 2D wave traveling through a circular aperture of width on the same order as the wavelength. By Huygen's principle each point on a wavefront acts as a point source (5 points shown). Due to diffraction and interference an Airy pattern is formed---where the light interferes constructively/destructively we get (local) maxima/minima in the intensity pattern. If the distance between the aperture and the screen is much larger than the wavelength, the slit acts as a point light source.
			(B) Approximation of an Airy pattern centered at $\frac{1}{2}(x_1 + x_2)$ by a Gaussian profile matching the maxima with the FWHM indicated.
			(C)/(D) Two Airy patterns centered at $x_1$ and $x_2$, distance \eqref{eq:abbeLimit}/\eqref{eq:rayleighLimit} apart, and their superposition (solid red).
		}
		\label{fig:airyFwhmAbbeRayleigh}
	\end{figure}
	
	\smallskip
	
	Concerning microscopes, from \Cref{eq:abbeLimit,eq:rayleighLimit,eq:fwhmLimit} it seems that there are only two possible ways to improve the resolution: either the wavelength has to be decreased, or the numerical aperture increased. Since the wavelength $\lambda$ is inversely proportional to the energy of the incoming light,
	decreasing the wavelength might damage the sample, a major issue in living cell microscopy. Hence, visible light (380--760 nm) is preferred for such applications. Concerning the second option, the numerical aperture of a modern lens is around 1.3--1.5 \citep{Diezmann2017}, and this value has not improved substantially during the last decades. In fact, Abbe's resolution limit has been standing as a paradigm for more than hundred years \citep{Hell1994}, limiting conventional light microscopes to about 250 nm lateral and 500 nm axial resolution\footnote{
		Axial resolution is the resolution in the longitudinal direction of the measurement trajectory ($z$-axis), whereas lateral resolution is the resolution in the image plane $(x,y)$. Note that the Abbe and Rayleigh criteria in \eqref{eq:abbeLimit} and \eqref{eq:rayleighLimit} hold for lateral resolution.
	}
	\citep{Betzig2006, Hell2007, Cremer2013, Heintzmann2013}. 
	
	\subsection{From microscopy to nanoscopy}\label{sec:fromMicroToNano}
	
	One important idea to improve on Abbe's resolution limit is confocal microscopy suggested by Minsky \citep{Minsky1961,Pawley2006} in 1961. After a laser excitation fluorescent dyes of fluorophores emit light of higher wavelength (less energy) than absorbed due to rotational and vibrational losses which then can be recorded in a detector device. Here only a small spot of the object is illuminated at any given time, and non-focused light is blocked by a pinhole. Moving the pinhole over the sample (scanning) creates multiple images which are then combined to produce the full image. Clearly, the smaller the pinhole, the more the resolution is increased. On the other hand, a smaller pinhole decreases the overall image intensity. Theoretically, confocal microscopy increases the resolution by $\sqrt{2}$, see e.g. \citep{Egner2020} or \citep{Hell2007}, but due to these competing effects practical increase is lower.
	Consequently, although providing some improvement, confocal microscopy on its own cannot break the resolution barrier \citep{Aspelmeier2015}.
	
	An early approach to overcome Abbe's resolution limit relies on the fact that both limits in \Cref{eq:abbeLimit,eq:rayleighLimit} are only valid in the \textit{far-field}, i.e. when sample and microscope are sufficiently far apart. Similarly, the regime when the sample and the microscope are less than a wavelength apart is called \textit{near-field}. In this case, the size of the aperture and not the wavelength determines the resolution \citep{Courjon2003}.
	In 1972 Ash and Nicholls \citep{Ash1972} went below Abbe's diffraction limit in the near-field.  Using 3 cm wavelength they achieved a resolution of $\lambda/60$. Current experiments are able to achieve a lateral resolution of 20 nm and a vertical resolution of 2--5 nm \citep{Durig1986, Oshikane2007}. Although impressive, near-field microscopes have certain disadvantages, the most obvious being that the specimen must be very close to the microscope and one is hence mostly limited to surface measurements. Moreover, they are unsuitable for transparent objects which excludes many biological samples. 
	
	\new{The major breakthrough to overcome Abbe's diffraction limit using far-field microscopy is intimately related to the development of photoswitchable fluorophores \citep{Hell2007, Huang2009}. These can be switched on and off in a statistically controlled manner, which finally allows to narrow the region of photon emission down to the nanoscale---resulting in super resolution microscopy. 
		The fundamental importance of this principle and its impact on modern science is reflected in the 2014 Nobel prize in Chemistry shared by E. Betzig, S. Hell and W. Moerner ``for the development of super-resolved fluorescence microscopy'' \citep{Ehrenberg2014}, where the term \textit{super-resolution} refers to any technique, which is able to break Abbe's diffraction limit in the far field. Since super-resolution microscopy is able to achieve resolutions in the nanoscale, it is also called \textit{nanoscopy}.}
	
	\new{The present paper analyzes not only ``classical'' microscopy, but also an important type of nanoscopy---the so-called \textbf{scanning mode} super-resolution microscopy.
		Before we dive into the scanning mode in the next paragraphs, we would like to stress that other important non-scanning mode super-resolution techniques exist, which we do not address in this paper. This includes in particular Single Marker Switching (SMS) nanoscopy in its various forms \citep{Betzig2006, Rust2006, Hess2006, Heilemann2008, Egner2007}. For a survey from a statistical perspective on nanoscale imaging in general see, e.g., \citep{Staudt2020} and for a survey on statistical single-molecule techniques see, e.g., \citep{Du2020}.}

	\new{In the scanning mode super-resolution microscopy, non-linearity of the response to excitation is exploited and dyes or fluorophores in a pre-defined region are shut off to enhance resolution. The sample itself---just like in confocal microscopy---is scanned along a grid by illuminating it with a (pulsed) excitation beam focused at the current grid point.
		We do not aim to describe all possible approaches here in detail (see, e.g., \citep{Wouterlood2012} or \citep{Aspelmeier2015} for a survey accessible to a statistical audience) and focus on the most prominent state-of-the-art scanning mode super-resolution technique---Stimulated Emission Depletion (STED) \citep{Hell1994, Hell2007, Klar2000, Egner2020}. In STED the fluorescent dyes are only excited in the center of a torus shaped region and are actively depleted inside the torus, see \Cref{fig:sted}. Super-resolution is achieved by selectively switching off the surrounding molecules by a second laser beam (depletion). Using a dichroic beamsplitter, it is ensured that only the fluoresced light is detected at the detector. On each grid point this procedure is repeated for a fixed time (the pixel dwell time) $t$ or equivalently for a fixed number of pulses (also denoted by $t$ for simplicity). Therefore, one is able to image specific predefined structures, instead of observing a superposition of the whole sample.}
	
	\begin{figure}
		\hspace{-0.2mm}\begin{subfigure}{.33\textwidth}
			\centering
			\includegraphics[width=.95\linewidth]{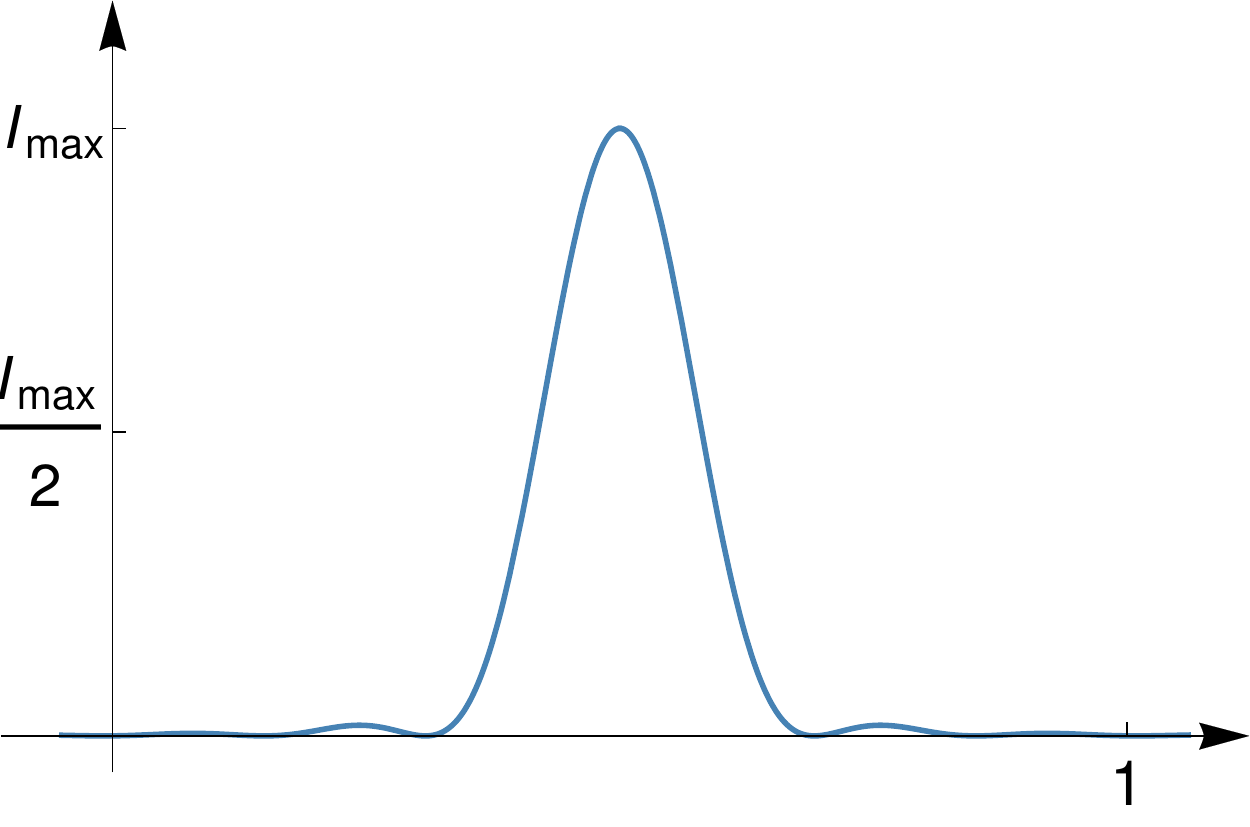}
		\end{subfigure}
		\begin{subfigure}{.33\textwidth}
			\centering
			\includegraphics[width=.95\linewidth]{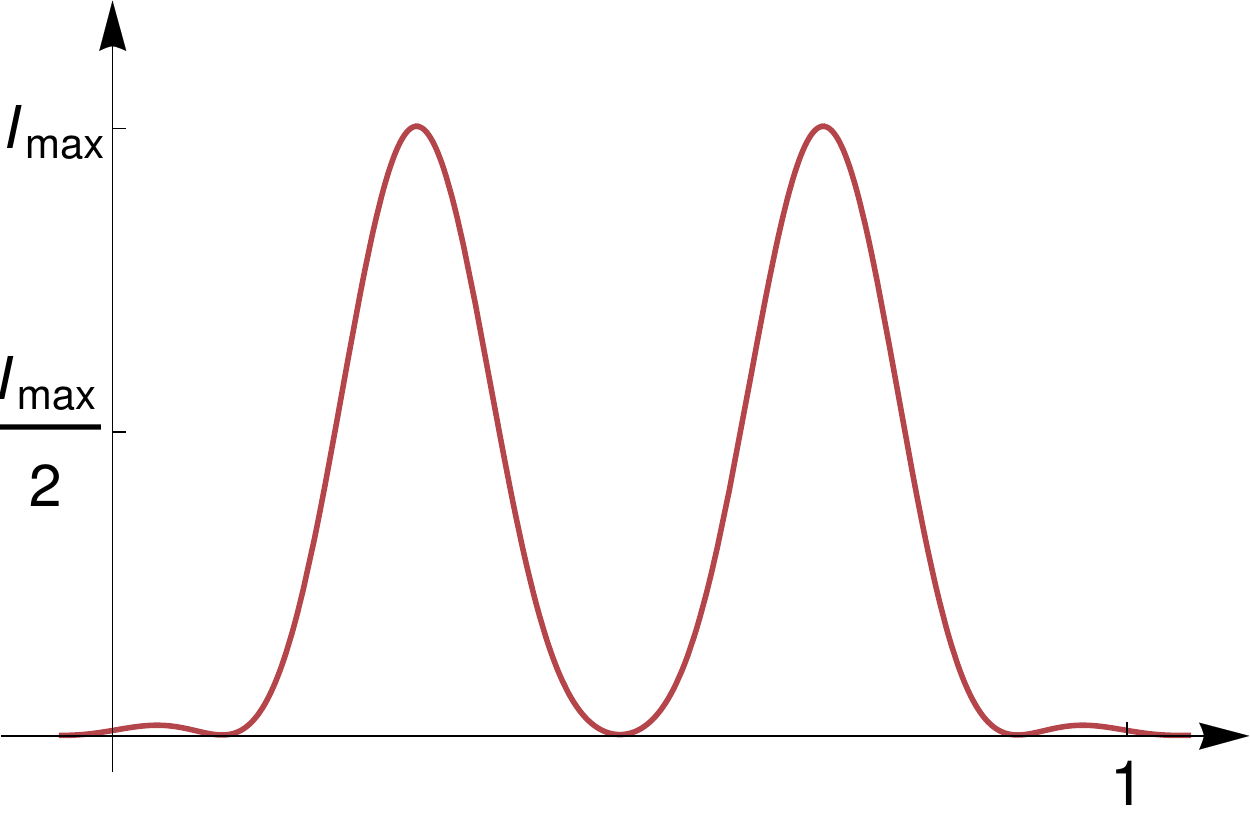}
		\end{subfigure}
		\begin{subfigure}{.33\textwidth}
			\centering
			\includegraphics[width=.95\linewidth]{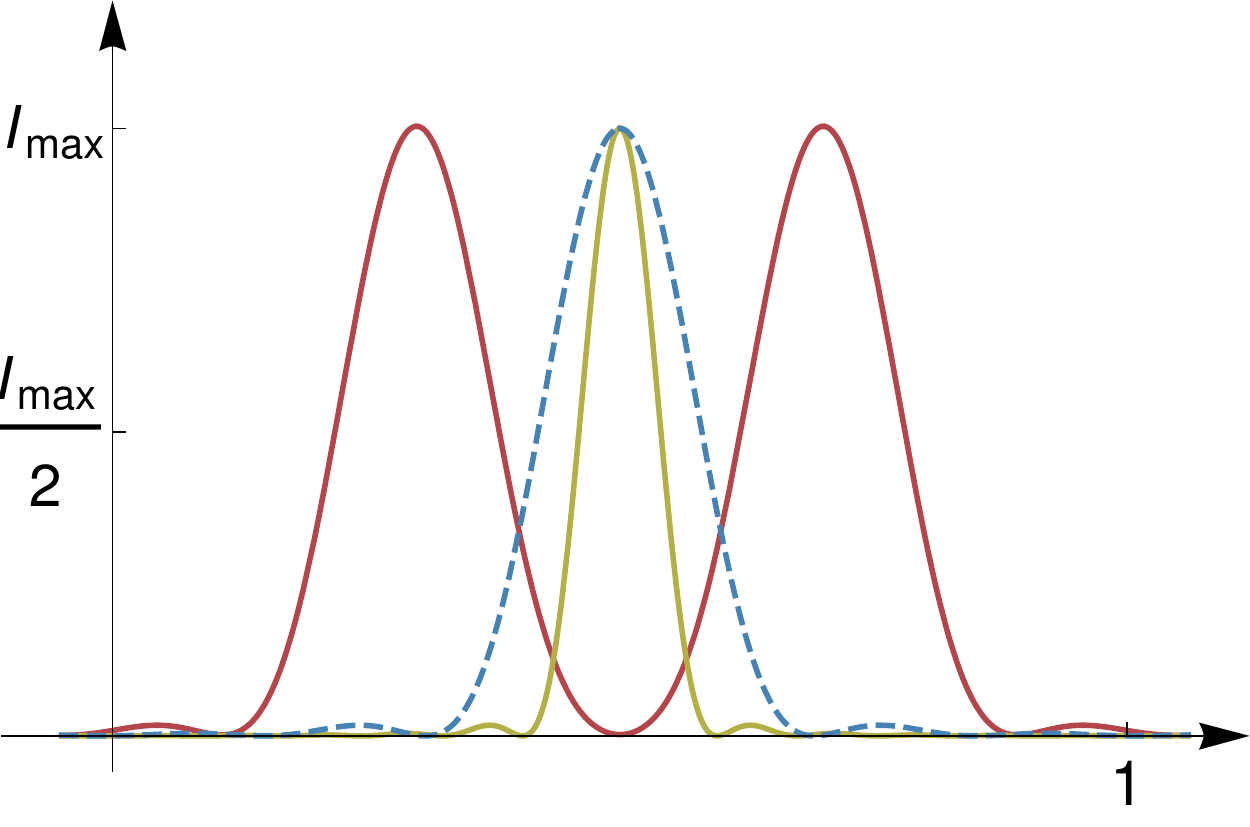}
		\end{subfigure}\vspace{0.5em}
		\begin{subfigure}{.33\textwidth}
			\centering
			\includegraphics[width=.92\linewidth]{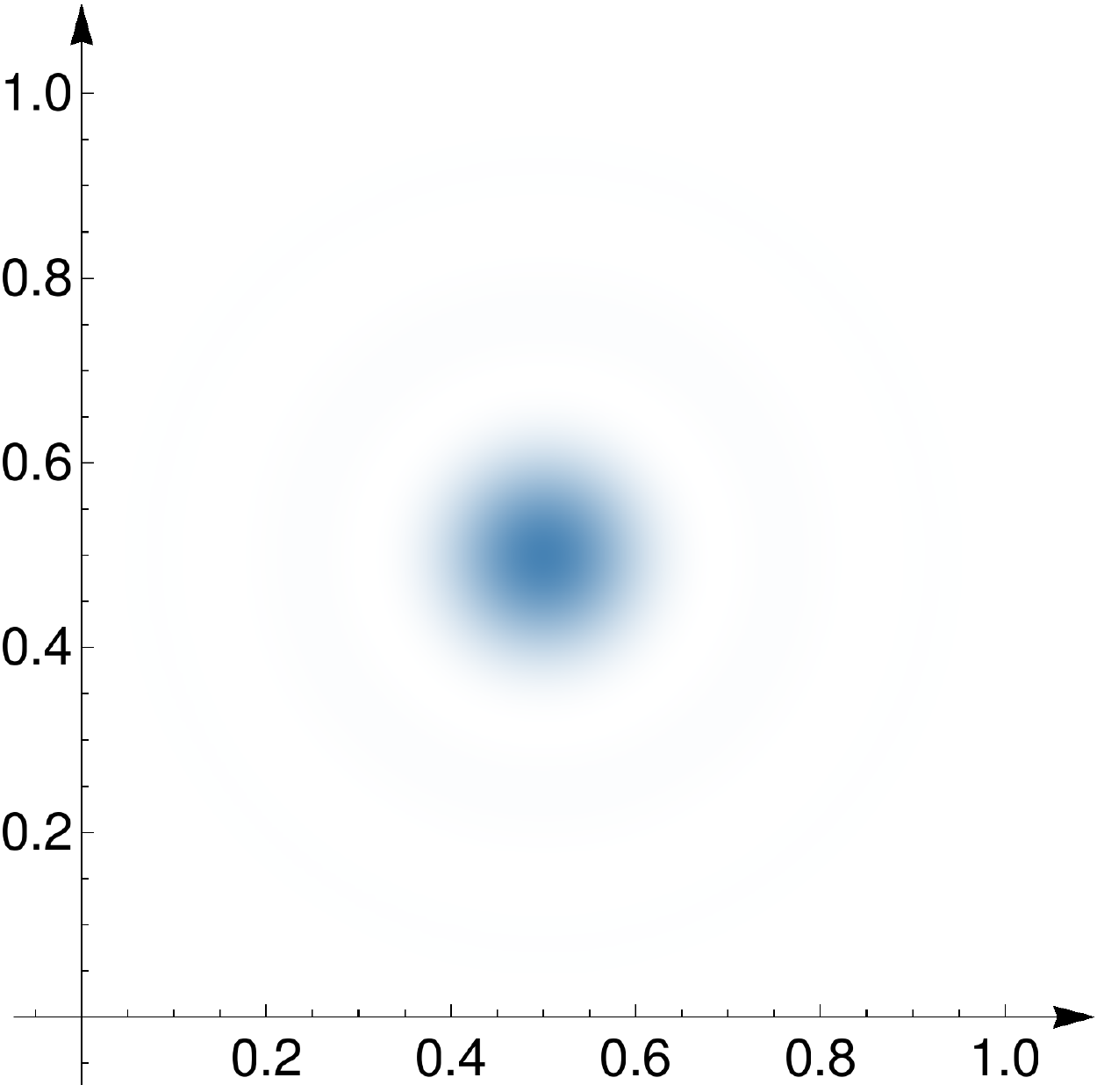}
		\end{subfigure}%
		\begin{subfigure}{.33\textwidth}
			\centering
			\includegraphics[width=.92\linewidth]{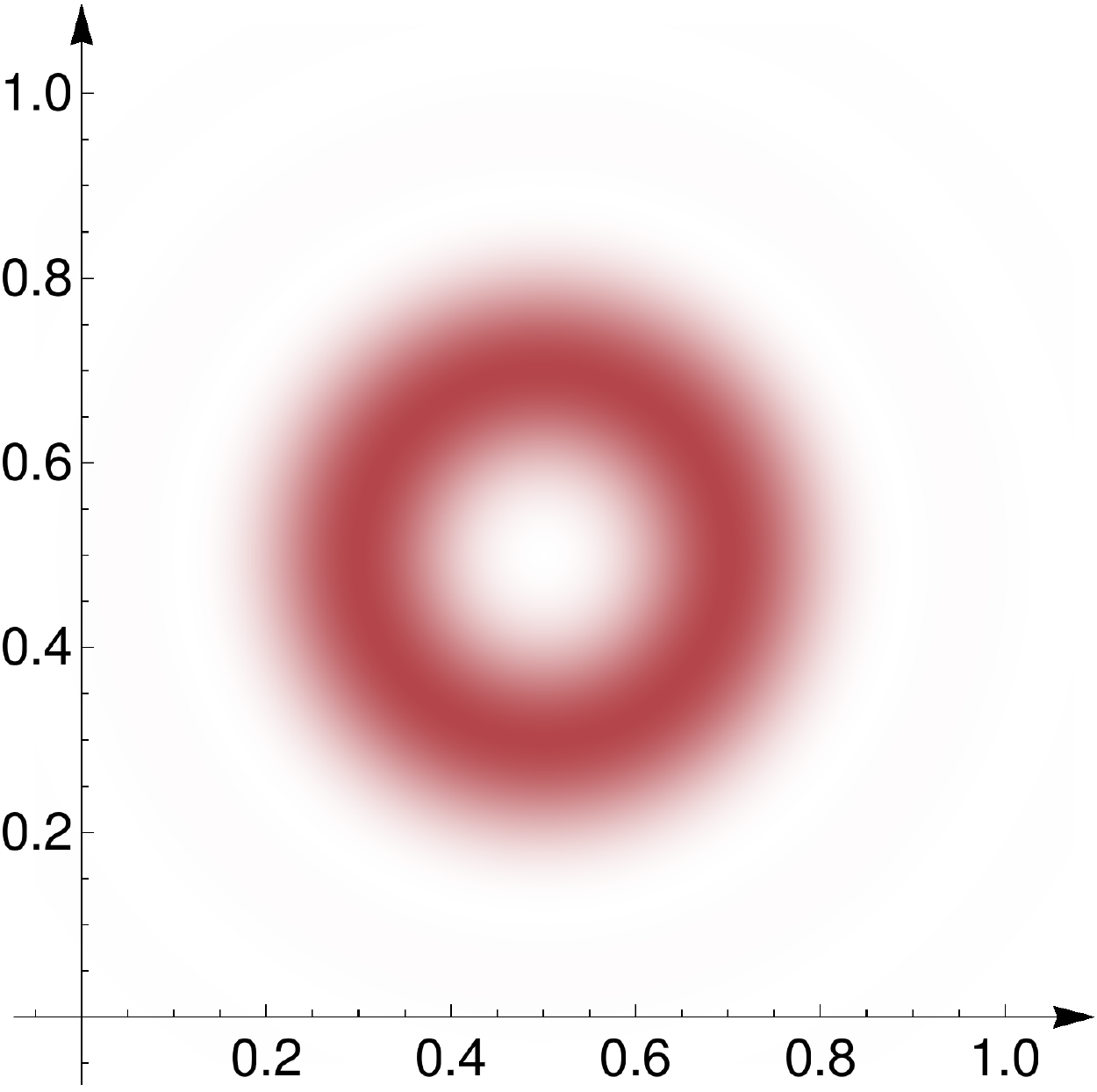}
		\end{subfigure}
		\begin{subfigure}{.33\textwidth}
			\centering
			\includegraphics[width=.92\linewidth]{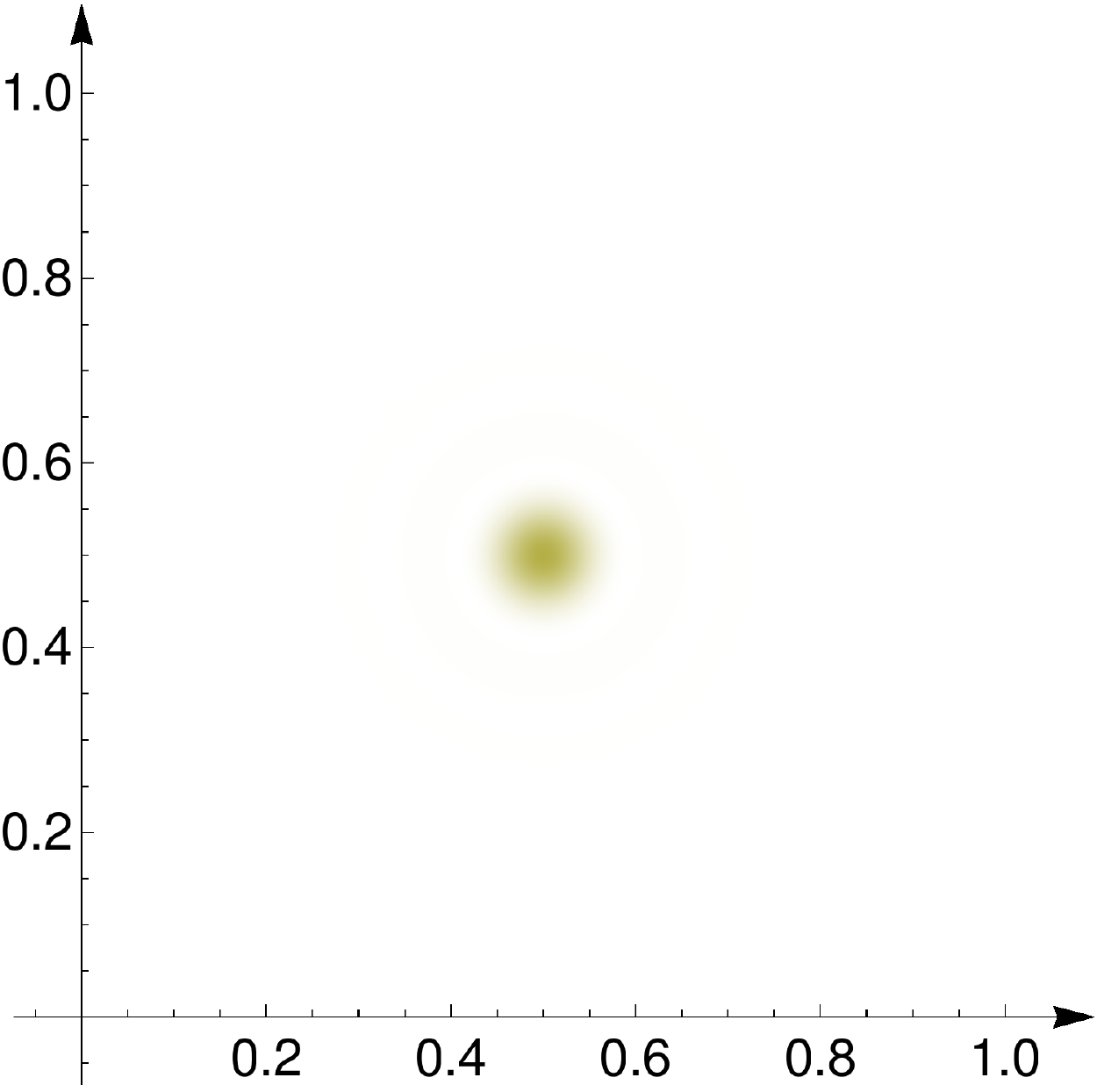}
		\end{subfigure}\hspace{10mm}
		\caption{STED microscopy. Column I: Original psf (blue), Column II: Depletion psf (red), Column III: effective psf (solid beige). The top row shows psfs in 1D, the bottom row in 2D.
		}
		\label{fig:sted}
	\end{figure}
	
	\begin{figure}
		\begin{subfigure}{.49\textwidth}
			\centering
			\includegraphics[width=.95\linewidth]{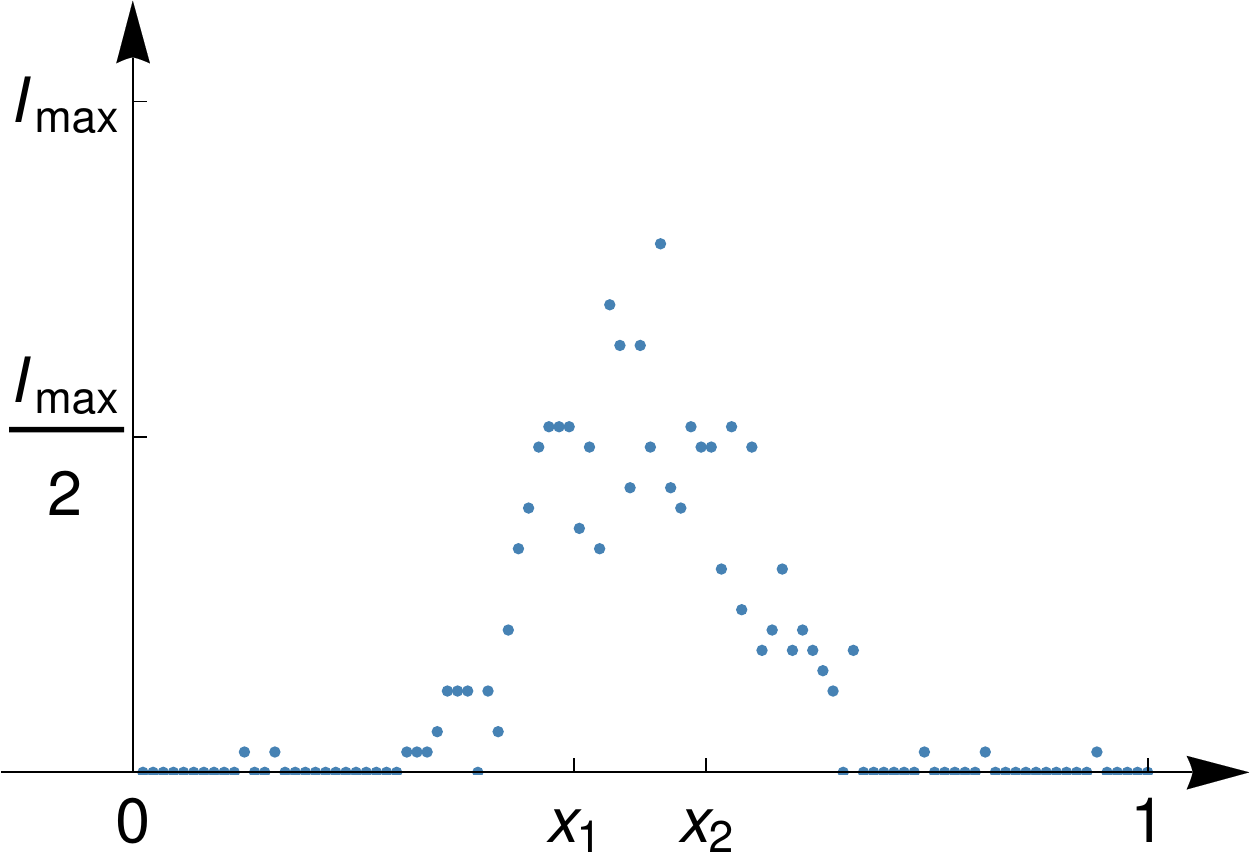}
		\end{subfigure}
		\begin{subfigure}{.49\textwidth}
			\centering
			\includegraphics[width=.95\linewidth]{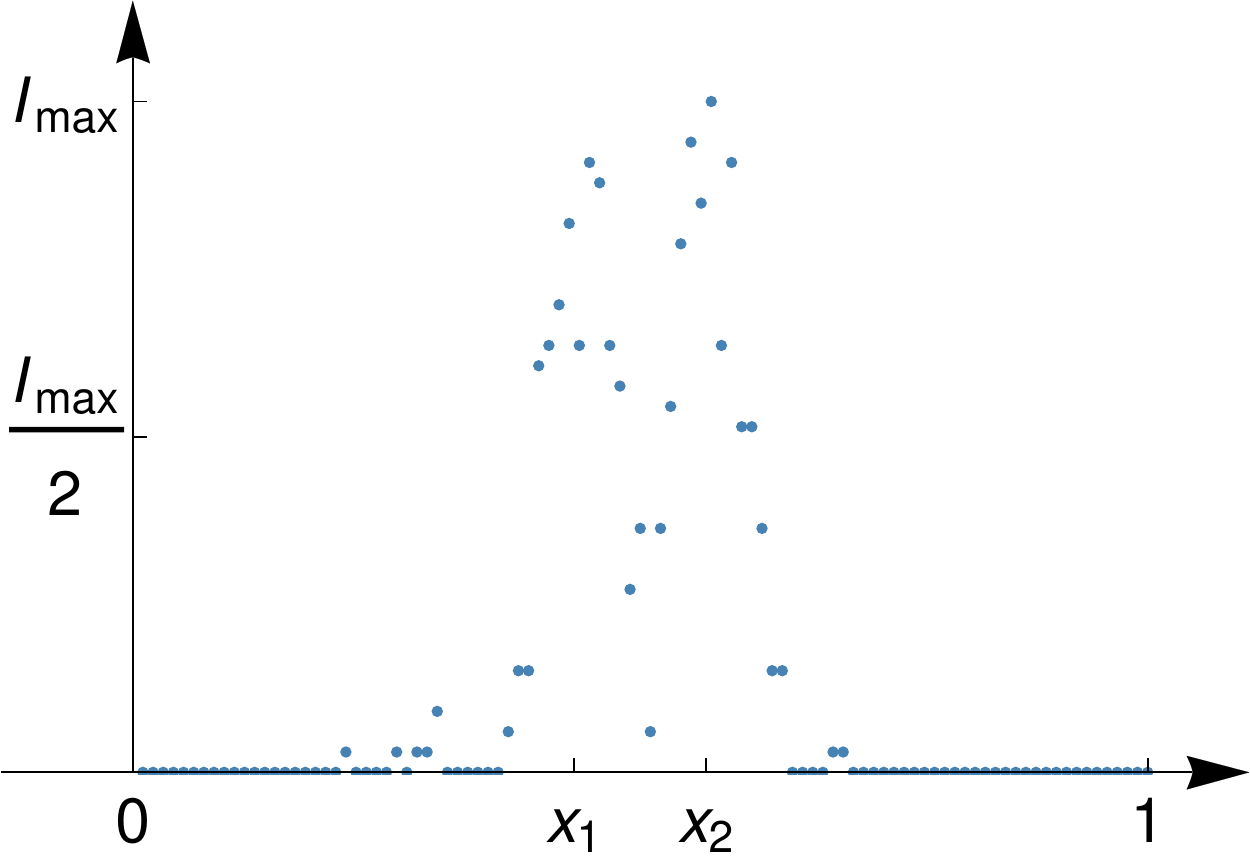}
		\end{subfigure}\\
		\caption{
			Photons generated by two point sources at $x_1$ and $x_2$ which are hard to distinguish for a conventional light microscope having Airy psf \eqref{eq:airy}, but are easily distinguishable with narrower effective psf after STED.	
		}\label{fig:stedDiscrete}
	\end{figure}

	To analyze the resolution of a STED super-resolution microscope, we can still employ an analog to Rayleigh's criterion \eqref{eq:rayleighLimit} by computing the effective psf, see \new{\Cref{fig:sted,fig:stedDiscrete}}. This leads to a resolution criterion of the form
	\begin{equation}\label{eq:STEDresolution}
	d = \frac{\lambda}{2 NA \sqrt{1 + \xi}},
	\end{equation}
	see \citep{Hell2007, Westphal2005}, where $\xi > 0$ is the shrinkage factor increasing in the direction of maximal intensity within the depletion spot. Note that, in principle, the resolution can be increased arbitrarily by increasing $\xi$. However, in practice, this leads to a decreased number of measured photons in view of the thinner psf and hence to a decreased signal-to-noise ratio. We will discuss this trade-off in Section 1.7. In experiments, resolutions of around $2.4$ nm have been achieved this way, see \citep{Rittweger2009}.

	\new{Even though the resolution formula \eqref{eq:STEDresolution} is reasonable in view of \eqref{eq:abbeLimit} and \eqref{eq:rayleighLimit} if both are understood in terms of the FWHM, \eqref{eq:STEDresolution} lacks dependency on another important contribution in super-resolution microscopy---the statistical error. In experiments, one is clearly aware that} both the experimental setup \textbf{and} the statistical error should play a role in the actual resolution of a (super-resolution) microscope. In fact, in any real world experiment, the noise plays a central role for the actual ability to distinguish two point sources, leading to the conclusion that the noise level (e.g. the observed number of photons) should also play a role in \Cref{eq:abbeLimit,eq:rayleighLimit,eq:STEDresolution}. This plays a minor role on the microscopic scale but becomes more severe as resolution increases, especially at the nanoscale. Given the vast applications of microscopy and rapid progress of super-resolution, a refined understanding of fundamental principles governing resolution is of immense importance. However, as far as we know, such mathematically rigorous model of statistical resolution is still lacking. To overcome this gap, in this paper we aim to provide unifying modeling (see \Cref{sec:statistics}) and statistical analysis (\Crefrange{sec:statTestingProblem}{sec:physicalImplications}), which allow to understand both the effect of the experimental setup (in terms of the convolution in \eqref{eq:convolution}) and the random nature of photon counts on the resulting resolution.

	\section{Statistics}\label{sec:statistics}
	
	\subsection{Statistical model}\label{sec:statisticalModel}
	
	To derive a mathematically rigorous formulation for the resolution of a (fluorescence) microscope with psf $h$, we start with modeling the actual observations. Throughout this paper we confine ourselves to the one-dimensional problem which is a prototype for higher spatial dimensions (see \Cref{rmk:higherDimensions} below). 
	
	In practice, the physical space $O$ is scanned bin-wise or sampled at once by a CCD camera or another detection device. We will assume that the image space $I$, the space of magnified points, is the unit interval $[0,1]$, and each scanned bin in $O$ corresponds to a bin $B_i = \left[(i-1)/n,i/n\right] \subset I$. From a mathematical point of view, we can for most experimental setups also re-scale $O = \left[0,1\right]$, and in this case scanning at a bin $B_i$ means to center the psf at the center of $B_i$. Each bin is either illuminated $t\in \nat$ times by a short excitation pulse (pulsed illumination), or it is illuminated continuously for some time $t$ (continuous illumination) which we may also assume to be an integer due to time discretization in the measurement process (e.g. $t$ can denote time in pico- or nanoseconds). For each bin we observe the total number of detected photons denoted by $Y_i \in \nat$. Clearly, $Y_i$ is a random quantity, but according to the above reasoning, we may assume that
	\begin{equation}\label{eq:expectationY_i}
	\Exp{Y_i} = t\int_{B_i} g(x) \, \mathrm{d}x,
	\end{equation}
	where $g$ is the image of the specimen as defined in \eqref{eq:convolution}. We assume here and in the following that the statistical experiments when measuring at $B_i$ are independent for different values of $i$, which is physically evident in many measurement settings, see e.g. \citep{Aspelmeier2015,Hohage2016}. Consequently, we observe a vector $\left(Y_i\right)_{ i\in \{1,\ldots,n\}}$ of independent random variables
	\begin{align}\label{eq:model}
	Y_i \overset{\text{indep.}}{\sim} F_{t\int_{B_i} g(x)\,\mathrm{d}x}, \qquad i \in \left\{1,...,n\right\}
	\end{align}
	with a family of distributions $F_{t\theta}$ for parameters $\theta \in \left(0,\infty\right)$ in mean value parametrization. 
	
	The specific choice of $F_{t\theta}$ depends fundamentally on the imaging setup and on the number of photons collected. We consider the following scenarios here:
	\begin{description}
		
		\item[\textbf{Poisson model (P)}]\label{model:poisson} The finest model we will consider here is a Poisson model $F_{t\theta} = \text{Poi} \left(t \theta\right)$. This is well-known and widely used in the literature, see e.g. \citep{Bertero2009,Hohage2016}. It is often derived in the setting of continuous illumination, but the Poisson model can also be motivated by means of the law of small numbers, see e.g. \citep{Munk2020}.
		
		\item[\textbf{Variance stabilized Gaussian model (VSG)}]\label{model:VSG} Due to the central limit theorem, for sufficiently large $t$ also normal models appear a reasonable approximation. Following the previous reasoning, this then leads to $\mathcal{N}(t\theta, t\theta)$. Applying the variance stabilizing transform $f(x) = 2\sqrt{x}$, we thus analyze $F_{2\sqrt{t\theta}} = \mathcal{N}(2\sqrt{t\theta}, 1)$.
		
		\item[\textbf{Homogeneous Gaussian model (HG)}]\label{model:HG}
		The simplest model to assume in this situation is the homogeneous Gaussian model $\mathcal{N}(\mu, \sigma^2)$ for some general mean $\mu = t\theta$ and some constant variance $\sigma^2$. In particular, many recovery algorithms rely on this model assumption, see e.g. \citep{Bertero2009,Hohage2016} for further discussion. After re-normalizing the mean $\mu$ by $\sigma$, we can w.l.o.g. set $\sigma = 1$ and consider the model $\mathcal{N}(t\theta, 1)$.
	\end{description}
	For a comprehensive discussion and more details on the modeling see e.g. \citep{Aspelmeier2015,Munk2020}. We emphasize that the homogeneous Gaussian model is commonly used as a proxy for ``microscopy with noise'' and has been investigated in many studies. We will, however, show that it is misleading in the present context. In contrast, we will show that the other two models (asymptotically) lead to the same resolution which scales linearly with the FWHM in agreement with the physical understanding.
	\begin{remark}\label{rmk:semiclassical}
		We consider photons, but treat them as classical particles. In the case of Poisson model, our modeling as given in \eqref{eq:expectationY_i} and \eqref{eq:model} corresponds to the so-called \textit{semiclassical detection model}, see e.g. Chapter 9 of \citep{Goodman1985}. This model is an approximation and follows from the general theory of light and matter interactions---quantum electrodynamics (QED), see e.g. \citep{Leonhardt2010} and in particular its Appendix B.
	\end{remark}
	
	\new{\begin{remark}\label{rmk:generalModels}
			In the above models, the only source of randomness is photon counting. However, in practice other sources can also occur, which require a detailed model of their own. An important second source of randomness in actual measurements is background contributions which arise either from external light sources or from contributions in out-of-focus regions. This can be included in the above models by adding a constant factor $\gamma > 0$ to the psf $\psf$, at least as long as the background can be assumed to be homogeneous (see also \Cref{rmk:backNoise} below). Furthermore, it might happen that the sensing devices cannot register each single photon precisely, but only with probability $\eta \in \left(0,1\right)$. From a statistical point of view, this leads to thinning of the counting process, see \cite{Munk2020} for details. An appropriate generalization of the above models in this case is given by multiplying the expectations with $\eta$. In the following, for notational simplicity we will focus on the three models P, VSG and HG without background and thinning, but will discuss our results in the more general case in \Cref{rmk:mainThmWithBackNoise}.
	\end{remark}}
	
	\subsection{Statistical testing problem}\label{sec:statTestingProblem}
	Building on the models of \Cref{sec:statisticalModel}, in the following we will describe the resolution of a microscope with psf $\psf \geq 0$ as a detection problem. We consider general psfs and provide a mathematically rigorous (asymptotic) statistical testing theory for resolution. To this end, we test the hypothesis that there is one point source at $x_0'$ against the alternative that there are two equally bright point sources at $x_1'$ and $x_2'$, respectively. This reflects the ability to discern between one and two objects, and is in line with many common resolution criteria, see e.g. \citep{Dekker1997}. Taking into account the previous considerations on diffraction, in particular \eqref{eq:convolution}, and setting $x_i = Mx_i'$ for $i \in \{0,1,2\}$ we hence consider testing the hypothesis that
	\begin{subequations}\label{eq:H0vsH1}
		\begin{equation}\label{eq:H0}
		H_0 : g(x) = \psf(x-x_0)
		\end{equation}
		against the alternative
		\begin{equation}\label{eq:H1}
		H_1 : g(x) = \frac{1}{2}\,\psf(x-x_1) + \frac{1}{2}\,\psf(x-x_2),
		\end{equation}
	\end{subequations}
	see \Cref{fig:testProblem} for an illustration. The factors of $1/2$ in the alternative ensure that the image $g$ has the same intensity under $H_0$ and $H_1$ (for generalizations to $q h(x-x_1) + (1-q) h(x-x_2)$ with $q\in(0,1)$ see \newSupp{\Cref{sec:generalCase}}). 
	We always assume that $x_0$ is fixed. For each particular alternative, we also assume that $x_1$ and $x_2$ are fixed as well. However, in the asymptotic analysis we will let $d = |x_1 - x_2|\to0$\footnote{In our analysis we will couple all parameters to the illumination time $t$. However, for ease of readability we omit the subscripts $t$, i.e. we write $n=n_t$ and $d=d_t$ throughout.} and later on we will consider the worst case scenario (\Cref{thm:symmetricallyPlacedHardest}). Without loss of generality, we scale the image function $g$ in \eqref{eq:H0vsH1} to be defined on the unit interval $[0,1]$ and normalize it to have volume $1$. 
	Setting the domain of $g$ to be the unit interval $[0,1]$ allows us to interpret $h(\cdot - x_i)$ as functions with domain $[0,1]$ for $i\in\{0,1,2\}$.
	
	
	\begin{figure}
		\centering
		\begin{subfigure}{.49\textwidth}
			\centering
			\includegraphics[width=.95\linewidth]{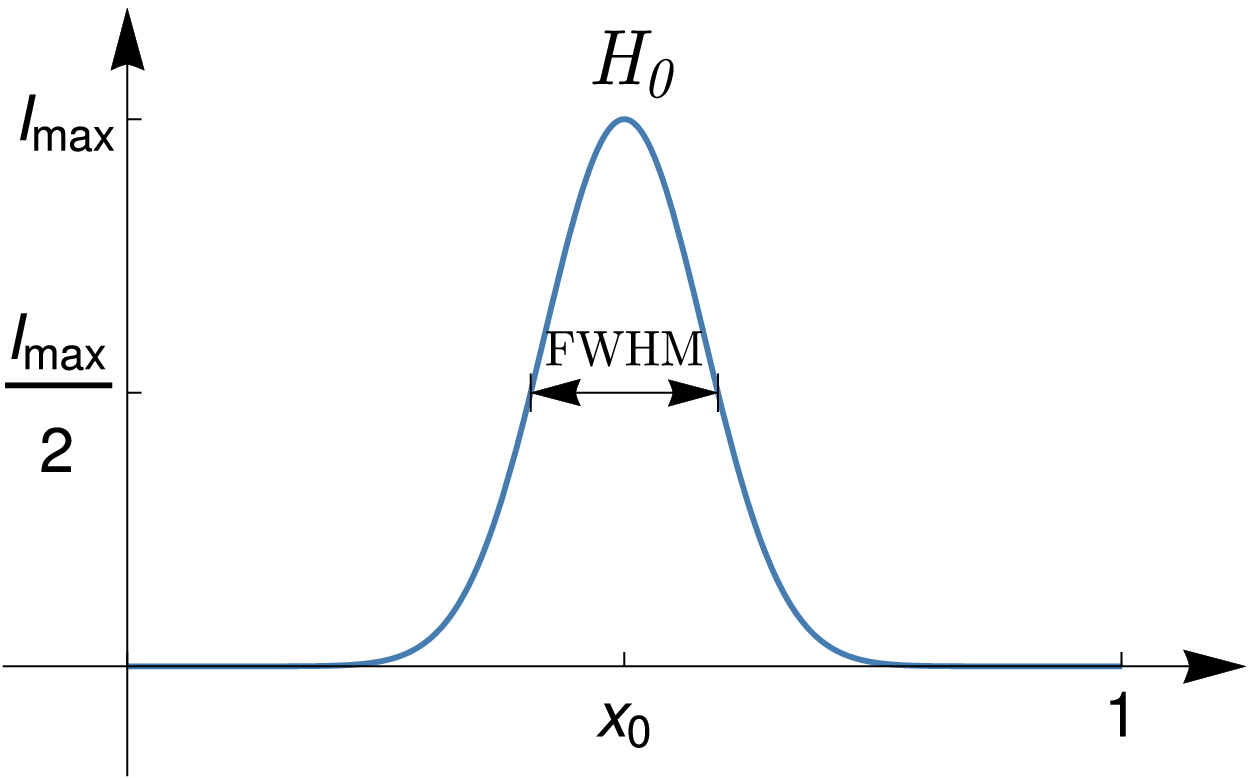}
		\end{subfigure}
		\begin{subfigure}{.49\textwidth}
			\centering
			\includegraphics[width=.95\linewidth]{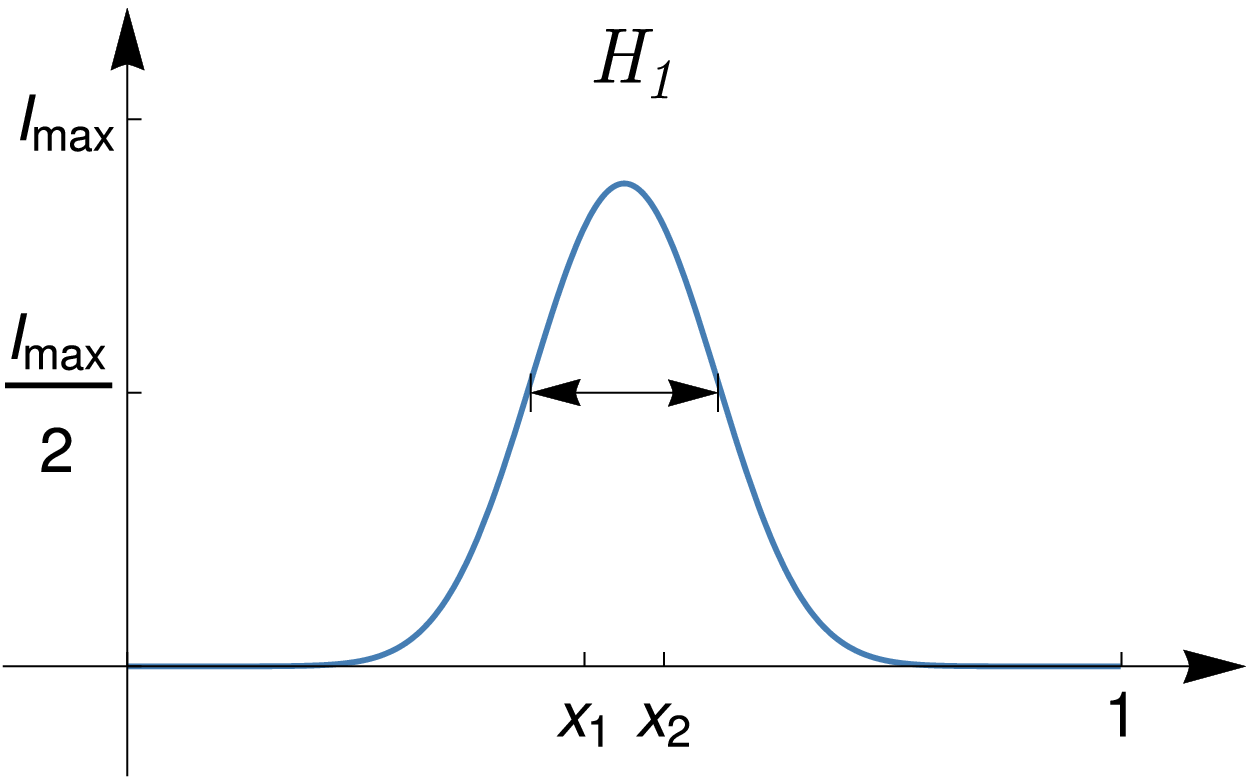}
		\end{subfigure}\vspace{2mm}
		\begin{subfigure}{.49\textwidth}
			\centering
			\includegraphics[width=.95\linewidth]{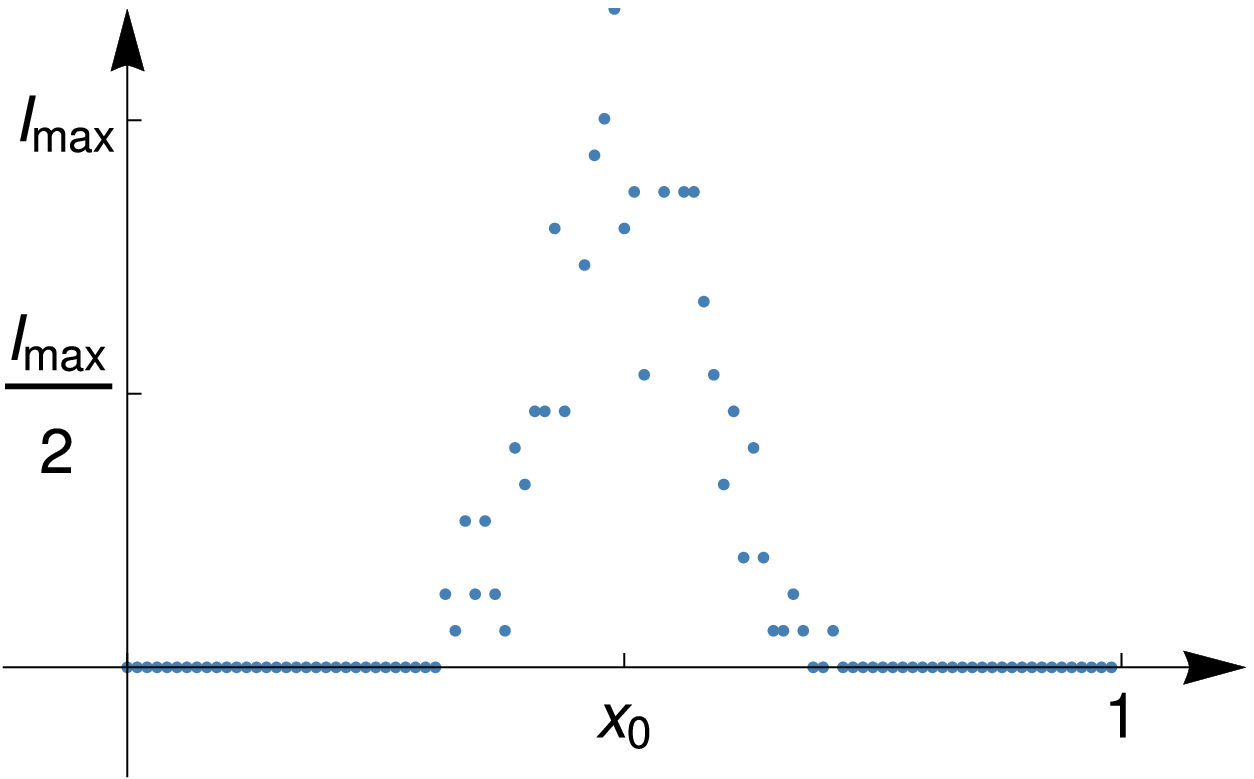}
		\end{subfigure}
		\begin{subfigure}{.49\textwidth}
			\centering
			\includegraphics[width=.95\linewidth]{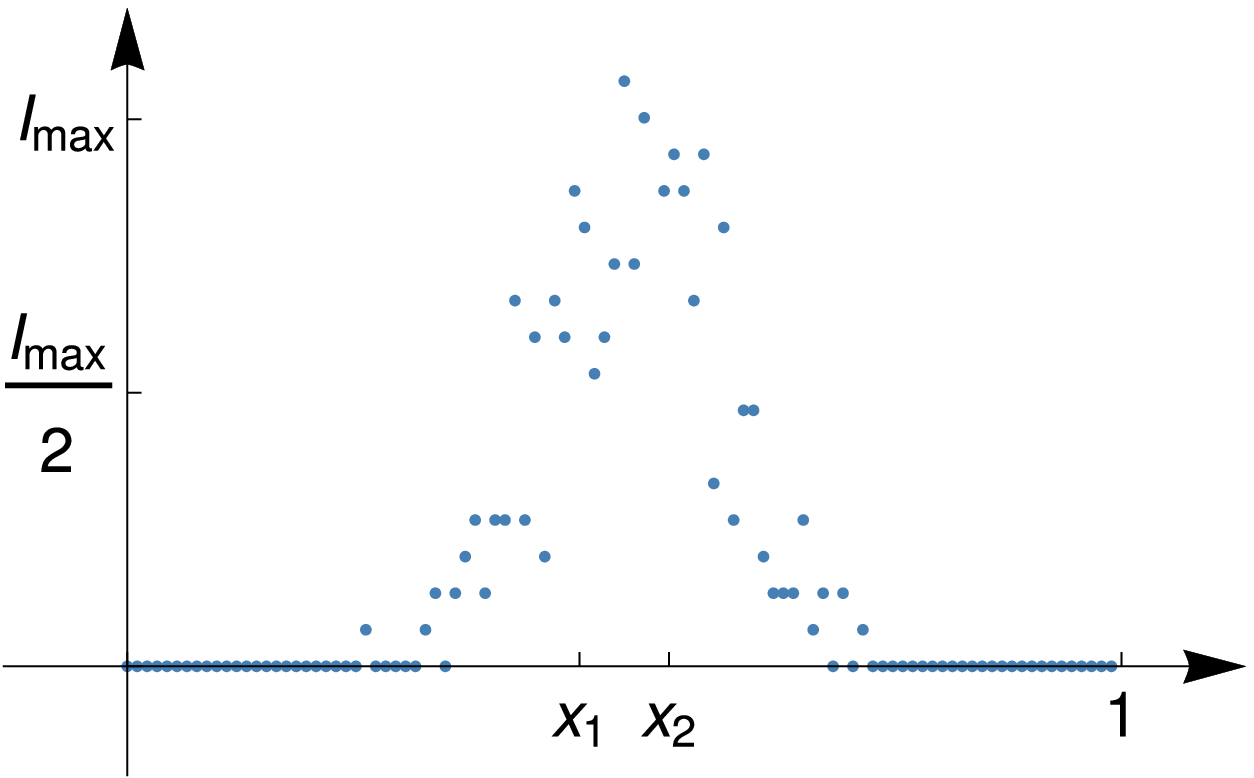}
		\end{subfigure}\vspace{2mm}
		\begin{subfigure}{.49\textwidth}
			\centering
			\includegraphics[width=.95\linewidth]{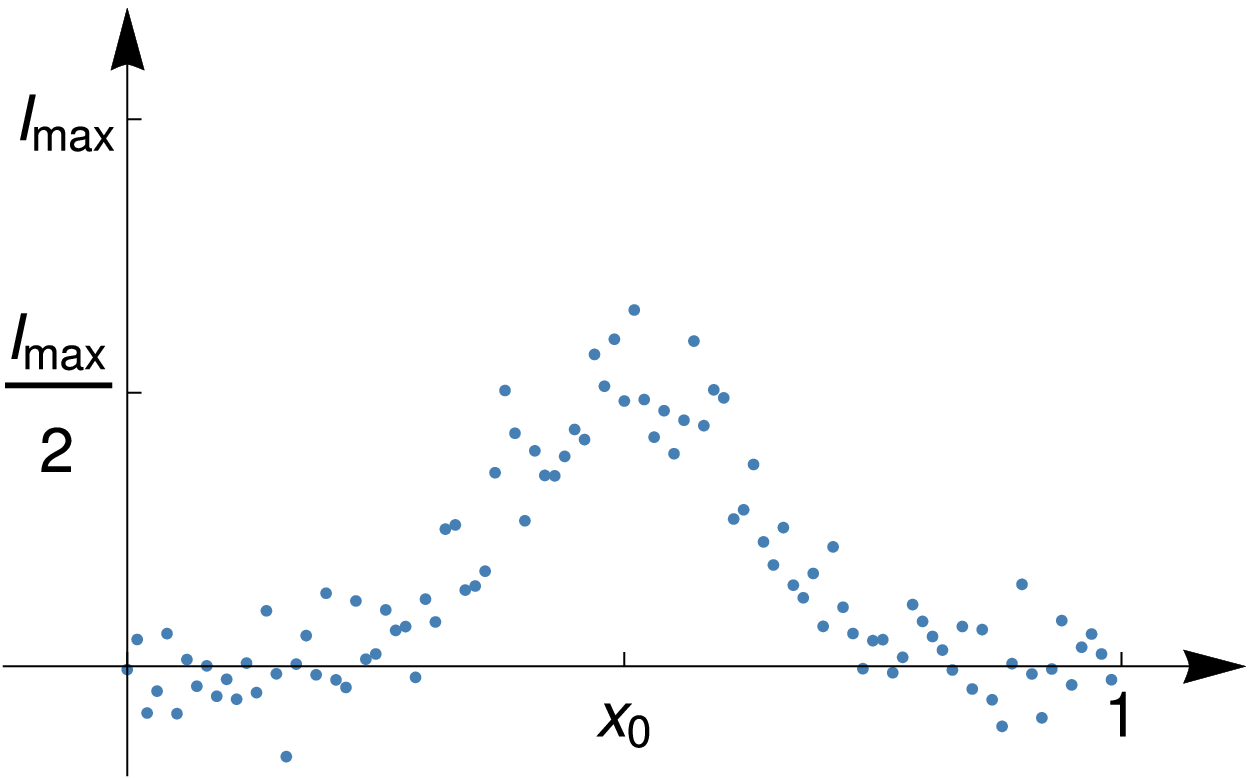}
		\end{subfigure}
		\begin{subfigure}{.49\textwidth}
			\centering
			\includegraphics[width=.95\linewidth]{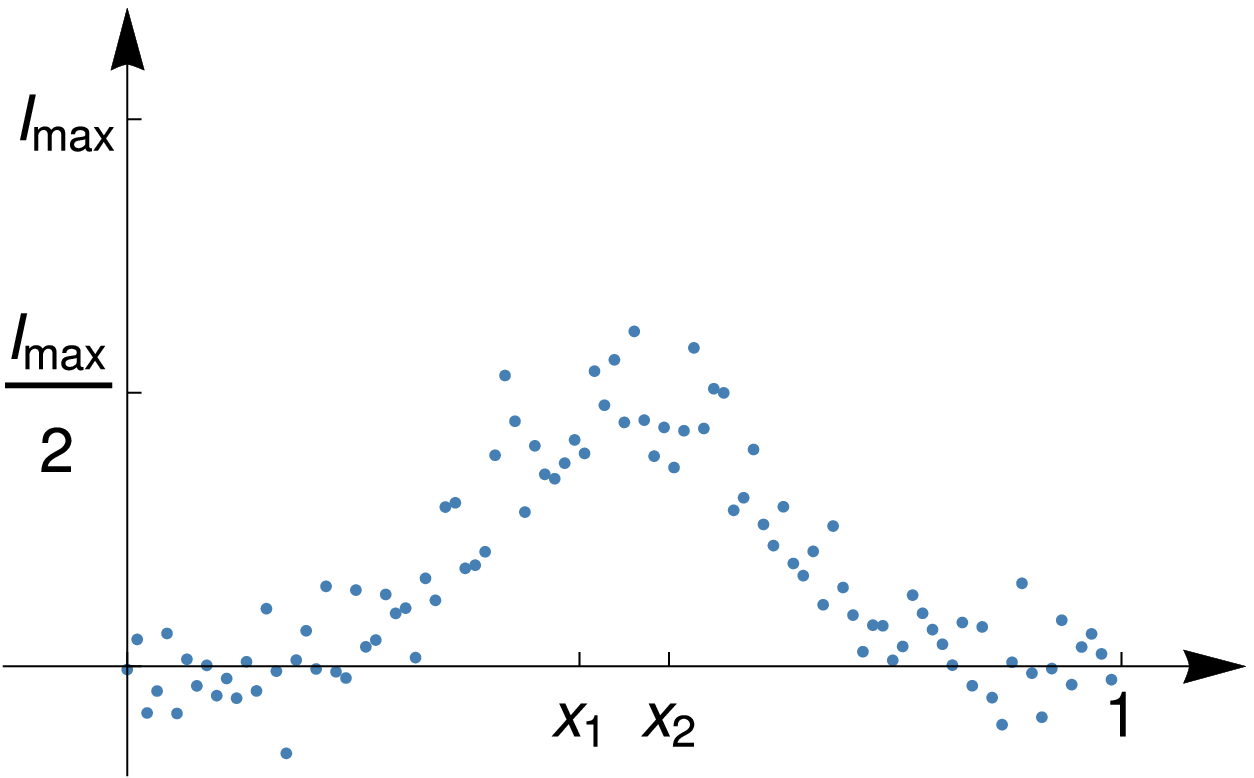}
		\end{subfigure}\vspace{2mm}
		\begin{subfigure}{.49\textwidth}
			\centering
			\includegraphics[width=.95\linewidth]{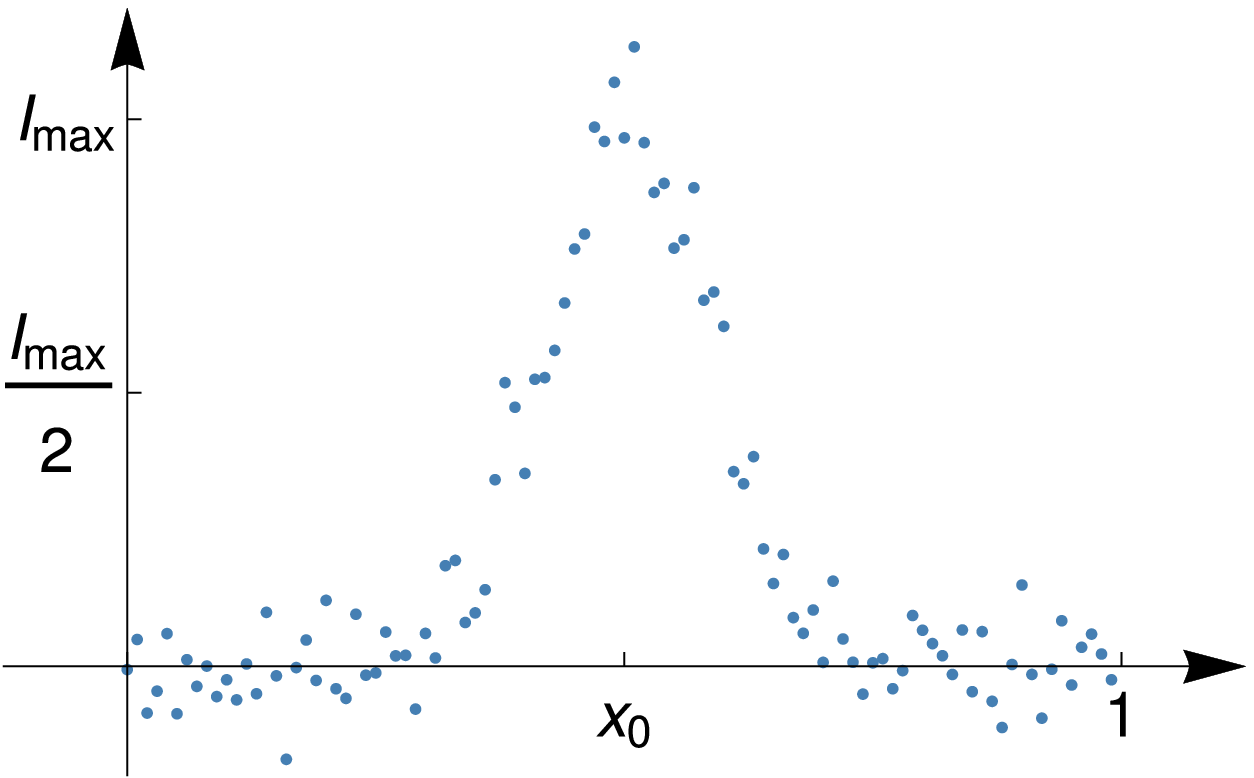}
		\end{subfigure}
		\begin{subfigure}{.49\textwidth}
			\centering
			\includegraphics[width=.95\linewidth]{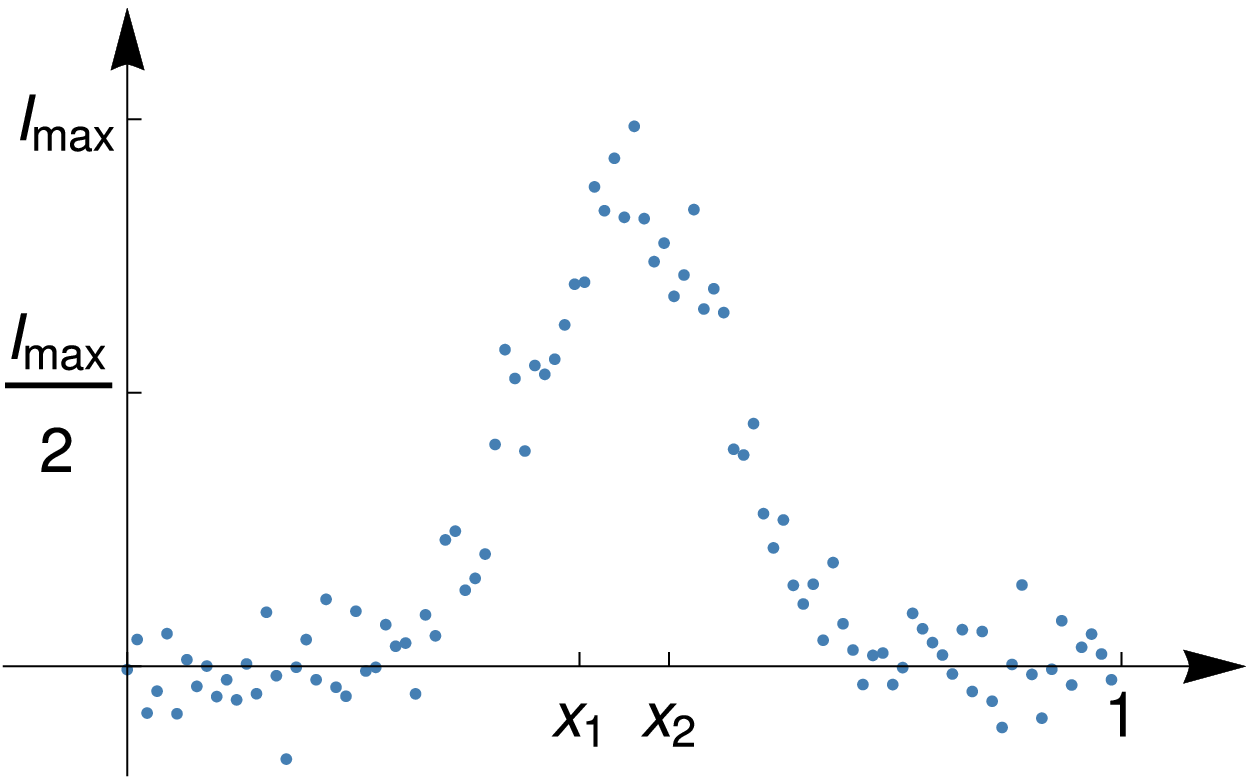}
		\end{subfigure}
		\caption{Resolution as a statistical testing problem in one dimension. First row: On the left hand side the hypothesis with the psf centered at $x_0$, on the right hand side the alternative with two psfs centered at $x_1$ and $x_2$, distance $d < \fwhm$ apart. Second, third and fourth rows: The corresponding observational data generated according to the Poisson, VSG and HG models, respectively.}
		\label{fig:testProblem}
	\end{figure}
	
	\begin{remark}\label{rmk:higherDimensions}
		Note that in practice, the hypothesis testing problem \eqref{eq:H0vsH1} occurs in multiple dimensions (depending on the observational setup). However, if $x_0, x_1, x_2 \in \left[0,1\right]^m$, $m \ge 1$, the statistically most difficult situation, independently of the (spatial) dimension $m$, is if all three points fall on a line, as otherwise the distributions under $H_0$ and $H_1$ would not have the same center of mass. Therefore, the whole problem can essentially be reduced to the one-dimensional problem of testing on this line.
	\end{remark}
	
	A \textit{(randomized) statistical test} for the hypothesis testing problem \eqref{eq:H0vsH1} is a measurable map $\Phi_n:\reals^n\to [0,1]$, $n\in\nat$, where $\Phi_n(Y) = p$ for $(y_1,\ldots, y_n) = Y$ means that we reject the null hypothesis with probability $p$. Each statistical test can make a \textit{type I error} when the hypothesis is falsely rejected with probability $\E  {H_0}{\Phi_n(Y)}$, and a \textit{type II error} when the hypothesis is falsely accepted with probability $1-\E  {H_1}{\Phi_n(Y)}$.
	
	As for the locations, the moment $x_0,x_1$ and $x_2$ are fixed, $H_0$ vs. $H_1$ in \eqref{eq:H0vsH1} constitutes a simple hypothesis vs. a simple alternative testing problem. Thus, according to the Neyman-Pearson lemma \citep{Lehmann2005} for a fixed $n$ and a fixed significance level $\alpha$, the \textit{likelihood ratio test} (LRT) for $H_0$ vs. $H_1$ is uniformly most powerful, i.e. no other statistical test can perform better. For our model \eqref{eq:model}, the LRT $\Phi_n \colon \reals^n \to [0, 1]$ takes the form
	\begin{equation}\label{eq:LRT}
	\Phi_n (Y) = \begin{cases}
	1		\quad \text{if}\;\, T_n(Y) > q^*_{\alpha,n},\\
	\gamma	\quad \text{if}\;\, T_n(Y)=q^*_{\alpha,n},     \\
	0		\quad \text{if}\;\, T_n(Y) < q^*_{\alpha,n},
	\end{cases}
	\end{equation}
	with the log likelihood ratio statistic $T_n(Y)$ given in terms of the probability mass functions or densities $f_{t\theta}$ of $F_{t\theta}$ by
	\begin{equation}\label{eq:LRTstat}
	T_n\left(Y\right) 
	=
	\log\left(\frac{\prod_{i=1}^{n}f_{tp_{1i}}}{\prod_{i=1}^{n}f_{tp_{0i}}}\right)
	=
	\sum_{i=1}^{n}\log\left(\frac{f_{tp_{1i}}}{f_{tp_{0i}}}\right),
	\end{equation}
	which then has to be determined according to the models \textbf{(P)}--\textbf{(HG)} from \Cref{sec:statistics}.
	Here and in what follows we abbreviate the detection probabilities in the $i$th bin by
	\begin{equation} \label{eq:p0iAbbrevation}
	p_{0i} := \int_{(i-1)/n}^{i/n}h(x-x_0)\, \mathrm{d}x
	\end{equation}
	under the hypothesis $H_0$ and
	\begin{equation} \label{eq:p1iAbbrevation}
	p_{1i} := \frac{1}{2}  \int_{(i-1)/n}^{i/n}h(x-x_1)\,\mathrm{d}x + \frac{1}{2}  \int_{(i-1)/n}^{i/n}h(x-x_2)\, \mathrm{d}x
	\end{equation}
	under the alternative $H_1$.
	
	Given a significance level $\alpha \in \left(0,1\right)$, the threshold $q^*_{\alpha,n}$ and the constant $\gamma$ in \eqref{eq:LRT} have to be chosen such that $\E  {H_0}{\Phi_N(Y)} = \Prob{H_0}{T_n(Y) > q^*_{\alpha,n}} + \gamma \Prob{H_0}{T_n(Y) = q^*_{\alpha,n}} = \alpha$, as this ensures $\alpha$ to be the level (i.e. the probability of the type I error) of the test.
	
	\subsection{Statistical resolution}\label{sec:statisticalResolution}
	In the following we adopt a minimax testing point of view. To this end, we begin by determining which choice of $x_1$ and $x_2$ in \eqref{eq:H0vsH1} is the most difficult to detect.
	
	\begin{theorem}\label{thm:symmetricallyPlacedHardest}
		Consider the testing problem \eqref{eq:H0vsH1} for $x_0 = 1/2$ fixed. Assume that the psf $h$ is even. \new{Fix} $0 < \alpha < 1/2$ and consider the asymptotic regime that $t, n\to\infty$ and $d\to0$. Then for each of the three models defined in \Cref{model:poisson} the uniformly most powerful test $\Psi^*$ (and hence the LRT) for \eqref{eq:H0vsH1} with asymptotic level $\alpha$ (i.e. $\E{H_0}{\Psi^*(Y)} \to \alpha$) has the asymptotically smallest power $\E{H_1}{\Psi^*(Y)}$ when 
		\[
		x_0 = \frac{x_1 + x_2}{2},
		\]
		i.e. when $x_1$ and $x_2$ are placed symmetrically around $x_0$.
	\end{theorem}
	
	
	With the above preparations in mind, we now propose the following definition for the resolution of a microscope:
	\begin{definition}[Statistical resolution of a microscope]\label{def:statResolution}
		Let $Y=\left(Y_i\right)_{i \in \left\{1,...,n\right\}}$ be as in \eqref{eq:model} and let $\psf$ be the point spread function of the microscope under investigation. Choose one of the three models Poisson, VSG or HG. Let $0<\alpha,\beta<1/2$, $x_0 \in \left[0,1\right]$, $t\in\nat$ and $n\in\nat$ be fixed. 
		We define the microscope's \textit{statistical resolution at point $x_0$, discretization $n$, exposure time $t$, type I error $\alpha$ and type II error $\beta$ under the prescribed model} as the unique value $d\in(0,1)$ such that the uniformly most powerful test (and hence the LRT \eqref{eq:LRT}) $\Psi^*$ for \eqref{eq:H0vsH1} with $x_1$ and $x_2$ chosen such that $d = \left|x_1 - x_2\right|$ and $x_0 = \frac{1}{2}(x_1 + x_2)$ has exactly level $\alpha$ and power $1-\beta$, i.e. \new{it} satisfies
		\[
		\E{H_0}{\Psi^*(Y)} = \alpha \qquad\text{and}\qquad \E{H_1}{\Psi^*(Y)} = 1-\beta.
		\]		
	\end{definition}
	
	In other words, if the distance $d$ between the two sources $x_1$ and $x_2$ in \eqref{eq:H0vsH1} satisfies $\left|x_1 - x_2\right| = d$, the statistical resolution is determined by the best possible test with detection power $1-\beta$ while the error of incorrectly assigning two sources (when only one is present) is controlled by $\alpha$. It is immediately clear that a larger value of $d$ will result in larger power, and a smaller value of $d$ will result in smaller power, i.e. the power as a function of $d$ is monotonically increasing and furthermore continuous. Thus, the statistical resolution is well-defined. Moreover, for $x_1$ and $x_2$ with $\left|x_1 - x_2\right| \leq d$ no level $\alpha$ test is able to distinguish $H_0$ and $H_1$ with power $\geq 1-\beta$. Note that, doing so, the sum of errors is bounded by $\alpha+\beta$, which is why we restrict ourselves to the case $\alpha,\beta \in \left(0,\frac12\right)$. Consequently, if $\alpha = 0$ or $\beta = 0$, the resolution is infinite---no method can achieve finite resolution if one of the errors is zero. In the case $\alpha = \beta = \frac12$ the test $\Psi \sim \Bin\left(1,\frac12\right)$, hence the resolution is $0$ which corresponds to the information of a coin flip to decide between $H_0$ and $H_1$.

	The aim of this paper is to study the \emph{asymptotic} behavior (as $n,t \to \infty$ and $d \to 0$) of the statistical resolution $d = |x_2 - x_1|$ in the three models from \Cref{model:poisson} and to relate our results to the classical Abbe and Rayleigh criteria. Furthermore, we will show that the (asymptotic) behavior of $d$ serves as a good proxy in finite sample situations whenever $n$ and $t$ are sufficiently large and $d$ is sufficiently small, which might be useful for designing experiments. This is investigated in simulations presented in \Cref{sec:simulations}. 
	
	\section{Main theorem and discussion}
	\subsection{Assumptions}
	To derive the precise asymptotic behavior of the statistical resolution $d$ of a given (super-resolution) microscope, we have to pose smoothness assumptions on its psf $h$ depending on the employed model.
	
	In the HG model we require the following.
	\begin{assumption}[HG model]\label{ass:h1}
		Suppose that the psf $h$ is even and non-constant. Furthermore let $h\ge0$ and $h(\cdot - x_i) \in C^2\left[0,1\right]$ for all $i\in\{0,1,2\}$.
	\end{assumption}
	The requirement that $h \geq 0$ is natural in view of $h$ being an intensity. The differentiability condition is rather mild and clearly satisfied for the Airy pattern in \eqref{eq:airy} and its most common approximation by a Gaussian.
	
	In case of the VSG and the Poisson models, we need a stronger condition:
	\begin{assumption}[VSG and P models]\label{ass:h2}
		Suppose that the psf $h$ is even and non-constant. Furthermore let $h>0$ and $h(\cdot - x_i) \in C^4\left[0,1\right]$ for all $i\in\{0,1,2\}$.
	\end{assumption}
	
	Note that due to compactness of $\left[0,1\right]$, \Cref{ass:h2} implies that $h \geq c > 0$ for some constant $c$.
	\begin{remark}\label{rmk:backNoise}
		We emphasize that the Airy pattern in \eqref{eq:airy} does not satisfy $h > 0$. However, in accordance with many models considered in the literature it is pertinent to include so-called background contributions, i.e. photons arising from other sources than the psf \new{(cf. \Cref{rmk:generalModels})}. Examples of such modeling include \citep{Acuna1997} and \citep{Diezmann2017}, which in the notation of \eqref{eq:model} would correspond to $Y_i \sim F_{t\int_{B_i} g(x)\,\mathrm{d}x + \gamma/n}$ with a positive constant $\gamma$ and $g$ given by \eqref{eq:H0vsH1}. If we were to incorporate this background noise into the psf $h$ and hence due to \eqref{eq:convolution} into the image $g$, we would obtain \eqref{eq:model} with $\tilde g = g + \gamma > 0$. From this point of view, the assumption $h > 0$ corresponds to the natural requirement that photons can be detected everywhere. 
		We also note that a Gaussian psf on $[0,1]$ \eqref{eq:gaussKernel}, which is the most commonly used approximation to the Airy pattern (see e.g. \citep{Diezmann2017} or \Cref{fig:airyFwhmAbbeRayleigh}B), clearly satisfies \Cref{ass:h2}.
	\end{remark}

	\subsection{Main theorem}
	For two sequences $\left(a_n\right)_{n\in\nat}$ and  $\left(b_n\right)_{n \in \nat}$ we write $a_n \asymp b_n$, $a_n \ll b_n$, $a_n \gg b_n$ and $a_n \sim b_n$ if $\lim_{n \to \infty} a_n / b_n  = 1$, $\lim_{n\to\infty} a_n/b_n = 0$, $\lim_{n\to\infty} b_n/a_n = 0$ and $\lim_{n\to\infty} a_n/b_n = c$ for some constant $c>0$, respectively. Note that, due to asymptotic considerations, we may restrict to non-randomized tests in what follows, i.e. to set $\gamma = 0$ in \eqref{eq:LRT}. Recall that we consider asymptotics as $d\to0$ and $n,t\to\infty$. We are now ready to state our main result on the asymptotic behavior of $d$. 
	
	\begin{theorem}\label{thm:main}
		Assume model \eqref{eq:model} and consider the testing problem \eqref{eq:H0vsH1} with $x_0, x_1, x_2 \in \left(0,1\right)$ such that $x_0 = (x_1 + x_2)/2$ \new{and $d = \left|x_1 - x_2\right| \to 0$, $n,t \to \infty$}. \new{Let $0<\alpha,\beta<1/2$ be fixed type I and II errors, respectively.} For \new{fixed} $0 < \nu < 1$ denote by $q_{\nu}$ the $\nu$ quantile of the standard normal distribution $\mathcal N \left(0,1\right)$.
		
		\begin{description}
			
			\item[(a) Poisson model \label{thm:mainPoi}]\hfill \\
			Let the distribution in \eqref{eq:model} be given by $F_{t\theta} = \Poi\left(t\theta\right)$ and the psf $h$ satisfy \Cref{ass:h2}. Then the statistical resolution $d$ of the corresponding microscope is 
			\begin{equation}\label{eq:resolutionPoi}
			d \asymp 2\sqrt{2}\, \sqrt{q_{1-\beta}-q_{\alpha}} \left(\int_0^1 \frac{h''\left(x-x_0\right)^2}{h\left(x-x_0\right)} \diff x\right)^{-1/4} t^{-1/4}.
			\end{equation}

			\item[(b) Variance stabilized Gaussian model \label{thm:mainVSG}]\hfill \\
			Let the distribution in \eqref{eq:model} be given by $F_{t\theta} = \mathcal N \left(2\sqrt{t\theta}, 1\right)$ and the psf $h$ satisfy \Cref{ass:h2}. Then the statistical resolution $d$ of the corresponding microscope also satisfies \eqref{eq:resolutionPoi}.

			\item[(c) Homogeneous Gaussian model \label{thm:mainGauss}]\hfill \\
			Let the distribution in \eqref{eq:model} be given by $F_{t\theta} = \mathcal N \left(t\theta, 1\right)$, $n=o\left(t^2\right)$ and the psf $h$ satisfy \Cref{ass:h1}. Then the statistical resolution $d$ of the corresponding microscope is 
			\begin{equation}\label{eq:resolutionGauss}
			d \asymp 2\sqrt{2}\,\sqrt{q_{1-\beta}-q_{\alpha}} \left(\int_0^1 \psf''\left(x-x_0\right)^2 \diff x\right)^{-1/4} t^{-1/2}\, n^{1/4}.
			\end{equation}
		\end{description}
		
	\end{theorem}

	\begin{remark}
		The assumption $n=o(t^2)$ for the HG model is necessary to get $d\searrow0$ asymptotically as $t,n\to\infty$. This assumption is not restrictive for modern microscopy---in most modern experiments there is at least one photon per pixel \citep{Diezmann2017} already from the background, i.e. $t\ge n$ seems natural. 
	\end{remark}
	
	\begin{remark}\label{rmk:mainThmWithBackNoise}
		\new{In case of constant background noise $\gamma > 0$ (\Cref{rmk:generalModels,rmk:backNoise}) and thinning with factor $0<\eta\le1$ (\Cref{rmk:generalModels}), we get
			\begin{equation}\label{eq:resolutionPoiThinned}
			d \asymp 2\sqrt{2}\, \sqrt{q_{1-\beta}-q_{\alpha}} \left(\eta\int_0^1 \frac{h''\left(x-x_0\right)^2}{h\left(x-x_0\right)+\gamma} \diff x\right)^{-1/4} t^{-1/4}	
			\end{equation}
			for \eqref{eq:resolutionPoi} and
			\begin{equation}\label{eq:resolutionHGthinned}
			d \asymp 2\sqrt{2}\,\sqrt{q_{1-\beta}-q_{\alpha}} \left(\eta^2\int_0^1 \psf''\left(x-x_0\right)^2 \diff x\right)^{-1/4} t^{-1/2}\, n^{1/4}
			\end{equation}}
		\new{for \eqref{eq:resolutionGauss}.}
		
		\new{From \eqref{eq:resolutionHGthinned} we see that possible background noise does not play any role in the HG model, thereby showing that the HG model is too simple for describing resolution accurately. The dependence on $\gamma$ in other models is as expected: the larger the background noise $\gamma$, the larger (i.e. poorer) the resolution $d$ at a scaling rate of $\gamma^{-1/4}$.}
		
		\new{As for for the thinning $\eta$, we see that whenever $\eta < 1$, this effectively reduces the illumination time precisely by the same factor, which agrees with intuition of the thinning factor as the probability of a photon detection in the sensing device.}
	\end{remark}
	
	\subsection{Strategy of the proof}
	
	Let us briefly comment on the techniques employed in the proof of \Cref{thm:main} \new{presented in \Cref{sec:mainResults}}. In both Gaussian models, the level and power of the LRT can be computed explicitly. The formulas \eqref{eq:resolutionPoi} and \eqref{eq:resolutionGauss} are then derived by straightforward approximations of integrals by sums as $t,n\to\infty$ and $d\to0$. In the Poisson model, the analysis is more difficult, as the LRT statistic consists of $n$ weighted Poisson random variables of varying intensity which might tend to any value in $\left[0,\infty\right]$ depending on the asymptotic relation between $t$ and $n$. We prove a CLT for the LRT statistic in case of $t \ll n^{2-\delta}$ for some constant $\delta >0$. If $t \gg \sqrt{n}\log^8 n$, we can exploit recent results from \citep{Ray2018} stating that the Poisson model is asymptotically equivalent in the Le Cam sense to the VSG model and hence \eqref{eq:resolutionPoi} holds true. Hence, both regimes together cover the whole parameter space. Note that in the overlapping regime there is no contradiction, since in both regimes we get the same asymptotic statistical resolution.
	
	\subsection{Physical implications}\label{sec:physicalImplications}
	Since in most microscopy experiments type I and type II errors are of equal importance, for the rest of this section we set the type I and II errors to be equal $\beta = \alpha$.
	To understand the experimental implications of \Cref{thm:main}, recall that for many (super-resolution) microscopes the psf can be well approximated by a Gaussian kernel
	\begin{equation}\label{eq:gaussKernel}
	h\left(x-x_0\right) = \frac{1}{\sqrt{2 \pi \sigma^2}} \exp\left(-\frac{1}{2 \sigma^2} \left( x- x_0\right)^2\right)
	\end{equation}
	centered at $x_0$ with variance $\sigma^2>0$, see \Cref{fig:airyFwhmAbbeRayleigh}B for an illustration. In this case, 
	\begin{equation*}
	\fwhm = 2\sqrt{2 \log 2}\, \sigma \approx 2.355\, \sigma
	\end{equation*}
	and setting $x_0 = 1/2$, we get
	\begin{align*}
	\int_0^1 \psf''\left(x-x_0\right)^2 \diff x
	&= 
	\frac{6 \sqrt{\pi } \sigma^3 \erf\left(\frac{1}{2\sigma }\right)
		+
		e^{-\frac{1}{4\sigma ^2}} \left(2\sigma ^2-1\right)}{16 \pi  \sigma ^8}
	=
	\frac{3}{8} \pi^{-1/2}\erf\left(\frac{1}{2\sigma}\right) \sigma^{-5} + o\left(\sigma^{-5}\right)\\
	&=
	\frac{3}{8} \pi^{-1/2} \sigma^{-5} + o\left(\sigma^{-5}\right)
	,\\
	\int_0^1 \frac{h''(x-x_0)^2}{h(x-x_0)}\, \mathrm{d} x 
	&=
	\frac{2\erf\left(\frac{1}{2\sqrt{2}\sigma}\right)}{\sigma^4}
	-
	\frac{e^{-\frac{1}{8 \sigma^2}} \left(4\sigma^2+1\right)}{4\sqrt{2\pi}
		\sigma^7} 
	=
	2\erf\left(\frac{1}{2\sqrt{2}\sigma}\right)\sigma^{-4} + o\left(\sigma^{-4}\right)\\
	&=
	2\sigma^{-4} + o\left(\sigma^{-4}\right),
	\end{align*}
	as $\sigma \searrow 0$ with the error function
	\[
	\erf(x) 
	=
	\frac{1}{\sqrt{\pi}}\int_{-x}^{x} e^{-t^2}\, \mathrm{d}t 
	= 2\Phi \left(\sqrt{2}x\right) - 1 . 
	\]
	
	Thus, according to \eqref{eq:resolutionGauss} we obtain in the \textbf{homogeneous Gaussian model} the following asymptotic behavior for the statistical resolution:
	\begin{align}\label{eq:detectionBoundaryHomGauss}
	d &\asymp  \frac{8\pi^{1/8}}{6^{1/4}}
	\sqrt{q_{1-\alpha}}\, 
	\frac{n^{1/4}}{\sqrt{t}}\sigma^{5/4}   
	= \frac{2^{7/8}\pi^{1/8}}{3^{1/4}(\log 2)^{5/8}}
	\sqrt{q_{1-\alpha}}\frac{n^{1/4}}{\sqrt{t}}\fwhm^{5/4}. 
	\end{align}
	Note that this is \emph{not} in agreement with the previously discussed FWHM resolution criterion \eqref{eq:fwhmLimit}, which postulates a linear dependency of $d$ on the FWHM, see also \citep{Egner2020} or \citep{Dekker1997}. From this point of view it becomes evident that the homogeneous Gaussian model is statistically too simple to capture the actual difficulty of the practical experiment. 
	
	In contrast, in the \textbf{variance stabilized Gaussian}, and \textbf{Poisson} models we compute
	\begin{equation}\label{eq:detectionBoundaryPoisson}
	d \asymp 2^{7/4}\sqrt{q_{1-\alpha}}\, t^{-1/4} \sigma 
	= \frac{2^{1/4}}{\sqrt{\log 2}} \sqrt{q_{1-\alpha}}\, t^{-1/4} \fwhm,
	\end{equation}
	which shows in fact a linear dependency of $d$ on the FWHM in good agreement with the criteria discussed in \Cref{sec:diffraction}. We summarize these results in \Cref{tab:resolution}. To interpret the results, let us look at the FWHM values in $[0.1,0.5]$. This interval is well-justified since in practice, e.g. for STED microscopes, the resolution is around $50\,\text{nm}$ \citep{Hell2007,Egner2020} and for measuring a single molecule, the field of view would naturally be restricted to a region of around $100 - 500 \,\text{nm}$. For such FWHM values the ratio between the resolutions of \eqref{eq:detectionBoundaryHomGauss} and \eqref{eq:detectionBoundaryPoisson} lies in the interval $[0.795 n^{1/4} t^{-1/4}, 1.19 n^{1/4} t^{-1/4}]$  with it being equal if $\fwhm\approx 0.250t/n$. Therefore, if $t=n$, then the difference between the homogeneous Gaussian and other models' resolution is $\approx\pm 20\%$. The difference is larger if the discretization $n$ is greater than the illumination time $t$ and vice versa. Moreover, if $n\ge 2.57 t$, then the resolution in the homogeneous Gaussian model is always larger than in the other models and hence too pessimistic for short illumination times. It is always smaller if $n \le 0.498 t$ and thus is too optimistic for long illumination times. 
	
	Even though we have argued before that $t\geq n$ is a natural assumption due to the background contributions, the case $n \geq t$ is especially interesting in super-resolution microscopy if the background is neglected. In two-dimensional experiments, it is common to scan with bin-sizes of $10 \times 10\, \text{nm}$, which for a single molecule requires around $10 \times 10$ bins. For modern dyes, the number of expected photons from one marker can be around $500$ in a standard confocal experiment, but in super-resolutions setups, this number can be considerably smaller due to the smaller region of excitation, e.g. around $50-100$. Hence, in our one-dimensional explanatory setup, values of around $10$ for $n$ and $7-10$ for $t$ are realistic when considering super-resolution setups without background.
	
	\bgroup
	\def\arraystretch{1.3}
	\begin{table}[h]
		\begin{center}
			\caption{
				Limiting asymptotic statistical resolution as given by \Cref{thm:main} for the Gaussian psf \eqref{eq:gaussKernel}.\label{tab:resolution} For ease of comparison, here we have set $n=t$ in the homogeneous Gaussian model.}
			\begin{tabular}{|l|l l l|}\hline
				\def\arraystretch{1.25}
				\diagbox[width=13em]{\textbf{Model}}{\textbf{Error} $\alpha=\beta$}&
				$0.01$   & $0.05$  & $0.1$ \\ \hline
				Homogeneous Gaussian & $3.08\, t^{-1/4}\fwhm^{5/4}$ & $2.59\, t^{-1/4}\fwhm^{5/4}$ & $2.29\, t^{-1/4} \fwhm^{5/4}$        \\ 
				VSG / Poisson        & $2.18\, t^{-1/4} \fwhm$ & $1.83\, t^{-1/4} \fwhm$ & $1.62\, t^{-1/4} \fwhm$ \\ \hline
			\end{tabular}
		\end{center}
	\end{table}
	\egroup
	
	
	Once the value of $t$ has been fixed, the asymptotic statistical resolution \eqref{eq:detectionBoundaryPoisson} allows to compare our results to the classical resolution limits of Abbe \eqref{eq:abbeLimit} and Rayleigh \eqref{eq:rayleighLimit}. Recall that the FWHM of the Airy pattern is $0.51\lambda/\na$, and hence both criteria can be read as $c \cdot \fwhm$ with a constant $c >0$. Consequently, we can compute the corresponding value of $\alpha$ such that the right-hand side in \eqref{eq:detectionBoundaryPoisson} equals $c \cdot \fwhm$. The results are shown in \Cref{tab:comparison}. We find that e.g. for $t = 10$ the Abbe criterion allows for a type I error of roughly $6.81\%$, whereas the Rayleigh criterion allows only $1.33\%$. We expect higher number of photons necessary in actual experiments, since we have completely disregarded the background noise by choosing the psf \eqref{eq:gaussKernel}.
	
	
	\bgroup
	\def\arraystretch{1.2}
	\begin{table}[h]
		\begin{center}
			\caption{The type I and II errors ($\alpha = \beta$) such that Abbe or Rayleigh criterion is fulfilled for the VSG and Poisson models for different values of the expected number of photons $t$ in 1D. Here we have assumed a Gaussian psf \eqref{eq:gaussKernel}, so the formula \eqref{eq:detectionBoundaryPoisson} can be simply inverted to calculate $\alpha$.
			}\label{tab:comparison}
			\begin{tabular}{|l|r r r r r|}\hline
				\def\arraystretch{1.25}
				\diagbox[width=13em]{\textbf{Error $\alpha=\beta$}}{\textbf{$\Exp{N} =  t$}}&
				$10$   & $20$  & $30$ & $40$ & $50$ \\ \hline
				Abbe criterion & $6.81\%$ & $1.76\%$ & $0.494\%$ & $0.144\%$ & $0.0432\%$ \\ 
				Rayleigh criterion & $1.33\%$ & $0.0857\%$ & $0.00614\%$ & $4.61\cdot 10^{-4}\%$ & $3.56\cdot 10^{-5}\%$ \\\hline
			\end{tabular}
		\end{center}
	\end{table}
	\egroup
	
	We can also use \eqref{eq:detectionBoundaryPoisson} to analyze the actual improvement by STED over a classical confocal microscope in our statistical context. To this end, recall that the FWHM is decreased by a factor determined by the maximal intensity within the depletion spot, cf. \eqref{eq:STEDresolution}. However, increasing the maximal intensity within the depletion spot automatically reduces the number of emitted and hence observable photons, which leads to an increased statistical error. Even though in practical examples \eqref{eq:STEDresolution} is still a good approximation of the actual resolution \citep{Hell2007}, we can make this more precise using \eqref{eq:detectionBoundaryPoisson} and explain the well-known observation that, unlike Abbe or Rayleigh criteria would suggest, the resolution improvement is not proportional to the FWHM decrease. In experiments, the parameter $\xi$ in \eqref{eq:STEDresolution} is typically chosen such that $\fwhm_{\mathrm{conf}} \approx 6 \fwhm_{\mathrm{STED}}$. As the expected number of photons is determined by the total amount of light emitted by the dyes, $t_{\mathrm{STED}}$ will be significantly smaller than $t_{\mathrm{conf}}$. To estimate $t_{\mathrm{STED}}$, in 1D we can use the first order approximation
	\[
	t_{\mathrm{STED}} \approx \frac{1}{\fwhm\text{ improvement}} t_{\mathrm{conf}} = \frac{1}{6} t_{\mathrm{conf}}.
	\]
	The rationale behind it is that when the psf is thinned by a factor equal to the $\fwhm$ improvement, the same holds for the total number of photons since it is proportional to the integral over the psf. Using \eqref{eq:detectionBoundaryPoisson} this yields
	\[
	d_{\mathrm{STED}} = \frac{1}{6^{\frac34}} d_{\mathrm{conf}} \approx \frac{d_{\mathrm{conf}}}{4},
	\]
	i.e. even though the FWHM is decreased by a factor of $6$, the resolution is only decreased by a factor of around $4$. This agrees quite well with experimental observations, see e.g. \citep{Egner2020}.
	
	\new{To some extent, this seems contradictory to the common interpretation derived from \eqref{eq:fwhmLimit} that the resolution depends linearly on the FWHM. However, this experimentally well-supported interpretation is only true if the other experimental parameters such as discretization, number of photon counts, etc. are fixed, and hence does not explicitly provide the dependency on noise, as our criteria clearly do. In practice, it is in fact well known that the resolution will also depend on other parameters of the sensing system such as the noise.} On the other hand, due to the development of more stable dyes, the number of observable photons has increased during the last decades along with the development of super-resolution microscopes. Thus, the decrease of the FWHM was accompanied by an increase of the signal-to-noise ratio, such that the rule of thumb ``resolution $\sim$ FWHM'' can still be considered valid. Our results give a mathematically rigorous and explicit formula involving both effects, and at the same time explain the experimental observations quite well.
	
	Note that the above argumentation can be readily extended to the two- or three-dimensional setting, as then the corresponding improvement can be computed for each spatial dimension separately.
	
\subsection{Related work}\label{sec:relatedWork}

Investigation of resolution in a statistical setting is not new. The HG model (and variations) was considered in \citep{Harris1964, Milanfar2002, Shahram2004, Shahram2005, Shahram2006} and the Poisson model (and variations) in \citep{Helstrom1964, Helstrom1973, Acuna1997}. However, with the exception of \citep{Acuna1997}, most of these works lack mathematical rigor, whereas \citep{Acuna1997} instead of defining resolution statistically suggest a redefinition in terms of the power function \newSupp{(see also \eqref{eq:powerAcuna})} and do not work out the dependency on the FWHM, see below for more details. 

Already in the 1960s, resolution has been investigated from a decision theoretic point of view in signal processing theory. Early references include Harris \citep{Harris1964} for the homogeneous Gaussian model and Helstrom \citep{Helstrom1964, Helstrom1965} for the Poisson model. In \citep{Helstrom1964, Helstrom1965} Helstrom considered signals consisting of different wavelengths varying in space, noting that using Reiffen and Sherman's paper \citep{Reiffen1963} on optimum demodulation for time-varying Poisson processes one could consider a signal varying in both space and time. For ease of understanding, we assumed that our psf intensity does not vary with time and is monochromatic, see \eqref{eq:airy}.
Harris \citep{Harris1964} only calculated the probability of a correct decision (power) without any consideration of the level. Helstrom \citep{Helstrom1964} assumed a CLT and basically obtained type I error and power expressions in the CLT regime \newSupp{(\Cref{sec:CLTanalysis})} for our Poisson model in his Equation (15). To see this, we have to set $ g_0 = q_{\alpha, t,n,d}^* := q_{1-\alpha}\sqrt{\V{H_0}{\statPoi}}+\E{H_0}{\statPoi}$, \newSupp{see \eqref{eq:q_defintion_Poisson}}, as the threshold in Helstrom's theory (which is not specified there), $M_{0}(x) = p_{0i}$, $M_{1}(x) = p_{1i}$, where $M_{\cdot}(x)$ is the effective photon count rate density at $x\in[-1/2, 1/2]^2$, and change integrals in his work to sums.

In \citep{Helstrom1973} Helstrom went even further than in \citep{Helstrom1964} and considered \eqref{eq:H0vsH1} in the context of quantum information theory, following the statistical paradigm originally set out by Middleton \citep{Middleton1953}. Among other things, Helstrom found out that $P_e$, the average of type I and type II errors, converges to $1/2\exp(-t)$ with increasing distance $d$. Here $t$ is interpreted as the average number of photons. As expected, the bound tends to zero in the classical regime as $t\to\infty$. Reassuringly, the form of his combined error probability $P_e$ becomes the same as ours with increasing $t$. However, Helstrom's results cannot be transferred to our case due to the quantum information theoretic setting, and his proofs are not mathematically rigorous. Notably, he found that $P_e$ is very close to its asymptotic minimum $1/2\exp(-t)$  whenever $d$ approximately equals twice the Rayleigh criterion, which led him to define the resolution as twice the Rayleigh limit. Much of the current research on resolution in quantum information theory revolves around trying to design different measurement techniques \citep{Tsang2016, Tsang2016a, Nair2016, Lu2018} which would allow to experimentally come as close as possible to the theoretical limits calculated by Helstrom \citep{Helstrom1973}. Some of these measurement techniques have been already confirmed by proof of principle experiments, see e.g. \citep{Tham2017}, others even applied to biological imaging \citep{Tenne2019}.
We emphasize that our theory is designed to describe everyday microscopy experiments with rather many photons so that Helstrom's limit $1/2\exp(-t)$ can be safely disregarded. Even though the mathematical treatment of quantum optics experiments 
is beyond the scope of this paper, we think that it is a fruitful research direction also for statisticians (see e.g. \citep{Yamagata2013}, where the authors have defined a \textit{quantum} likelihood ratio).

\new{We also mention contributions from the field of modern signal processing and engineering, namely the works by Milanfar and collaborators \citep{Milanfar2002, Shahram2004, Shahram2006}, see also \citep{Shahram2005} for an overview. These authors also investigate resolution in terms of statistical measurement errors, and they derive a dependency of the resolution on the inverse fourth root of the so-called \textit{measurement} signal-to-noise-ratio, see also \cite{Smith2005}. Note that this has some similarity with the dependency on $t$ in \eqref{eq:resolutionPoi}. However, even though resolution is treated as a statistical testing problem, in all these papers a homogeneous Gaussian model (which is challenged by our analysis) is assumed and the estimation error (quantified by a Cramer--Rao type lower bound in \cite{Smith2005}) rather than the detection error, which we believe provides a more accurate description of resolution in a statistical context, is used. See also Terebizh \citep{Terebizh1995} for a non-Bayesian view on this.  
	\cite{Ferreira2020} has recently introduced support stability which means---roughly speaking---that the true number of support points has to be recovered exactly (in our context one support point in the hypothesis and two support points in the alternative). In the context of a deterministic noise model with band-limited PSF they determine conditions when for a LASSO type estimator such support stability is valid.}

Closest to our paper appears \new{to be} the work \citep{Acuna1997} by Acu{\~n}a and Horowitz on telescope resolution. There, the testing problem $H_0\colon d = 0$ vs. $H_1\colon d>0$ in a 2D model on a line is considered. This corresponds to our Poisson model, but with explicit constant background noise (see \Cref{rmk:backNoise} on how to incorporate such noise into our model and \Cref{rmk:mainThmWithBackNoise} for corresponding results). Their main quantity of interest is $p_{1i}$ \eqref{eq:p1iAbbrevation} considered as a function of $d$. Under assumptions on $p_{1i}$'s roughly corresponding to our assumptions on the psf $h$, they analyzed the likelihood ratio test in the regime where $t\to\infty$, but kept the number of measurements (discretization) $n$ fixed and finite. Clearly, a finite value of $n$ will at some point restrict the resolution to be of the order $1/n$, as no information finer than the bin-size can be obtained. Moreover, the mathematical treatment of this regime is substantially simpler, as the LRT statistic is given by a finite sum of independent weighted Poisson random variables, whose intensity tends to $\infty$, and hence one obtains a CLT trivially. Acu{\~n}a and Horowitz \citep{Acuna1997} also note that there is a different regime with finite fixed $t$ and $n\to\infty$, but do not treat this. All of our results except for asymptotic equivalence also hold in this regime: See \newSupp{\Cref{rmk:finiteNhomogeneous,rmk:finiteNstabilized}}, and note that the relation between $t$ and $n$ necessary for \newSupp{\Cref{thm:CLTforPoi}} is trivially satisfied for constant $t$. The authors define resolution as the (asymptotic) power function of the likelihood ratio test rather than as a single number, which in some sense, is close to our \Cref{def:statResolution}. However, we believe that it is not intuitive for practitioners to define the resolution as a probability, since they are used to thinking of resolution as a distance.
The main result of \citep{Acuna1997} is the calculation of this power function in the regime $t\to\infty$, $n=const$, which we can reproduce asymptotically for large $n$ and $t$ from our more general results (up to dimension) if we keep a sum instead of the integral in \eqref{eq:resolutionPoi}, see \newSupp{\Cref{rmk:finiteNstabilized}}. Note furthermore that the power expression of \citep{Acuna1997} is only valid if $d = const \times t^{-1/4}$ in accordance with our result \eqref{eq:resolutionPoi}. We stress that our results give an explicit dependency on the FWHM.

Finally we mention, that the term `super-resolution' is used in mathematical and statistical communities also in a different context, see \citep{Donoho1992, Morgenshtern2016, Candes2013, Candes2014, Fernandez-Granda2015}. There super-resolution addresses the ways to localize signals with (un)known amplitudes by observing their (noisy) Fourier samples, i.e. samples in the frequency domain. The domain is always assumed to have some cut-off frequency $f_c$ corresponding to the inverse Abbe limit in our context. In contrast, in this paper we assume that the locations of our signals are always \textit{known}, i.e. we will follow the experimentalists' terminology.

\section{Simulations}\label{sec:simulations}
\subsection{Simulation setup}
To investigate the finite sample validity of our asymptotic theory, we have performed simulations exploring the (asymptotic) resolution's $d$ dependence on the illumination time $t$, $\fwhm$ and discretization $n$, see \eqref{eq:detectionBoundaryHomGauss} and \eqref{eq:detectionBoundaryPoisson}.

In all simulations we chose the level $\alpha = 0.1$ and determined when the type II error is in the range $\beta \in [0.95\alpha, 1.05\alpha)$. For simplicity, we only describe the simulation in \Cref{fig:simulationsSmallAndBig} (a) of $d$ vs. $\fwhm$ in detail, others were conducted similarly. Throughout the simulation we set discretization $n=20$ and $d=\fwhm$ as the starting distance between the peaks in the alternative. Then for $10,000$ times we generated $n$ independent random variables following the corresponding model \eqref{eq:model} under the alternative and calculated the type II error. We then used the bisection method to advance $d$ until the type II error became between $0.95\alpha$ and $1.05\alpha$. 
We performed the above procedure for the $\fwhm$ range $0.15,0.16,\ldots, 0.25$.

\subsection{Simulation results}

The slopes obtained by log-log plots support our theory well already for small $t$ and $n$, see left column of \Cref{fig:simulationsSmallAndBig} and \Cref{tab:simulationResultsExperimental}. 
We stress that in the Gaussian models we only have to consider $d\searrow0$, provided that we change the integrals in \Cref{thm:main} to sums (see \newSupp{\Cref{rmk:finiteNhomogeneous,rmk:finiteNstabilized}}).
Therefore, it is expected that for given $t$ and $d$ the simulations are in general closer to the theoretical results in \Cref{thm:main} for the VSG and HG models.
This is confirmed by the simulations, where in general the HG simulated values are much closer to the theoretical ones. 
As a rule of thumb, if $t\ge 500$ and $n \ge 500$, the asymptotic formulas can be used as good approximations, see the right column of \Cref{fig:simulationsSmallAndBig}. In general, increasing the intensity $t$ seems to make asymptotic formulas closer to the simulations than increasing the discretization $n$. This is displayed in \Cref{fig:simulationsIntermediate} which looks at the $(t,n)$ plane in more detail: The asymptotic formulas get much closer to the simulations when transitioning from (a) with $(50,50)$ to (c) with $(100,50)$, than from (a) to (b) with $(50,100)$.

\begin{figure}[!htb]
	\begin{subfigure}{.49\textwidth}
		\includegraphics[width=\linewidth]{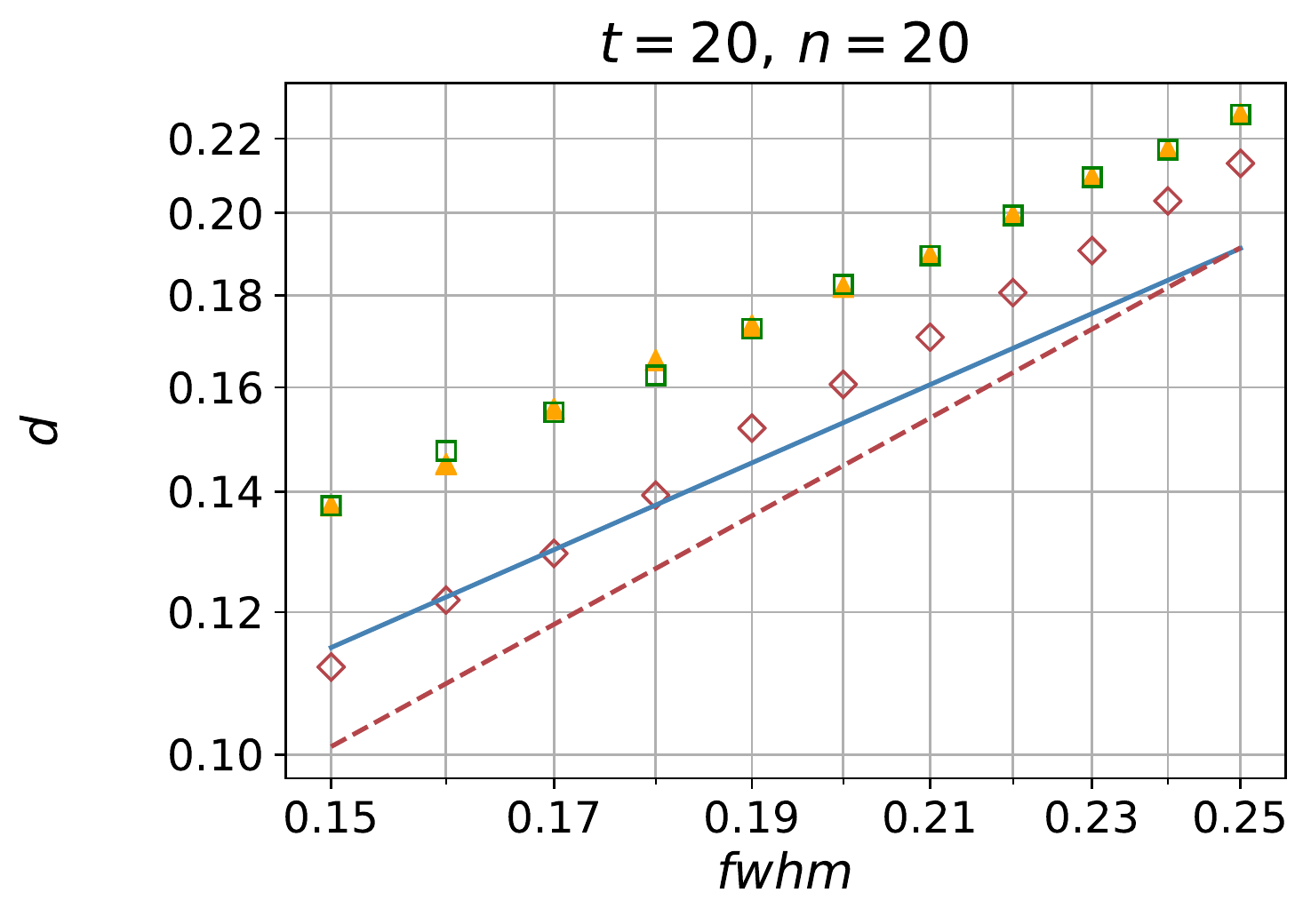}
		\caption{
		}
	\end{subfigure}
	\hspace{1mm}
	\begin{subfigure}{.49\textwidth}
		\includegraphics[width=\linewidth]{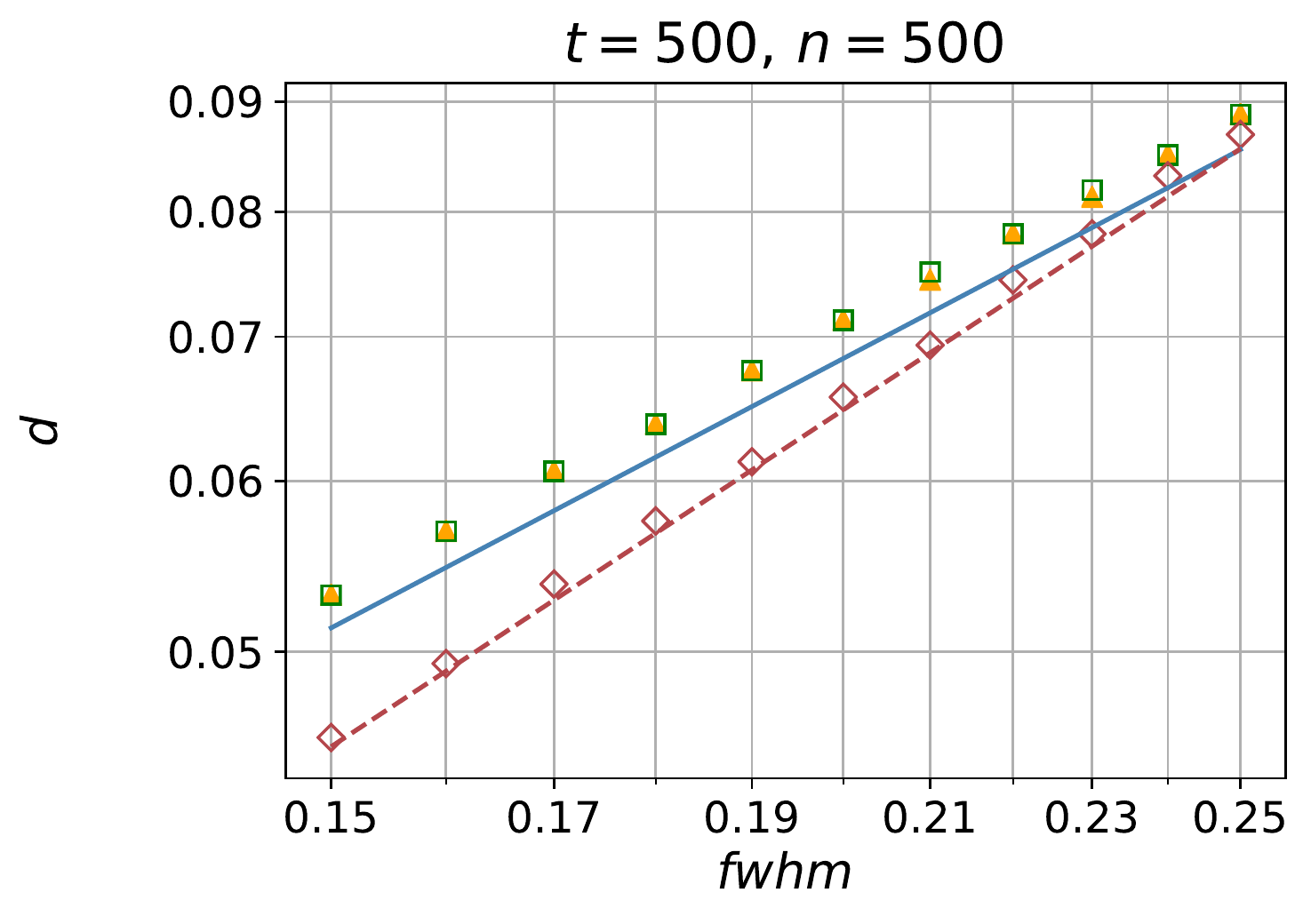}
		\caption{
		}
	\end{subfigure}\\
	\begin{subfigure}{.49\textwidth}
		\includegraphics[width=\linewidth]{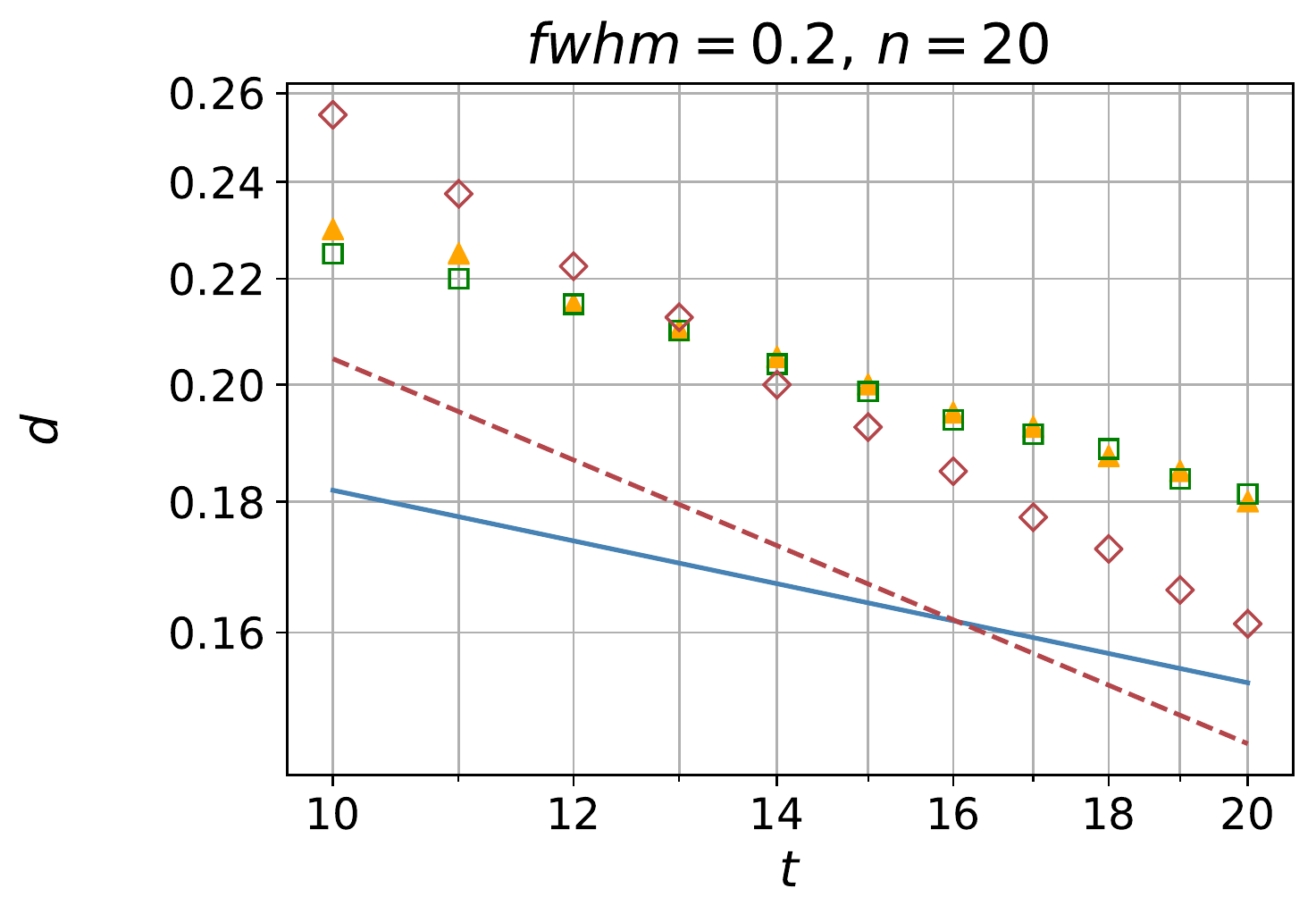}
		\caption{
		}
	\end{subfigure}
	\hspace{1mm}
	\begin{subfigure}{.49\textwidth}
		\includegraphics[width=\linewidth]{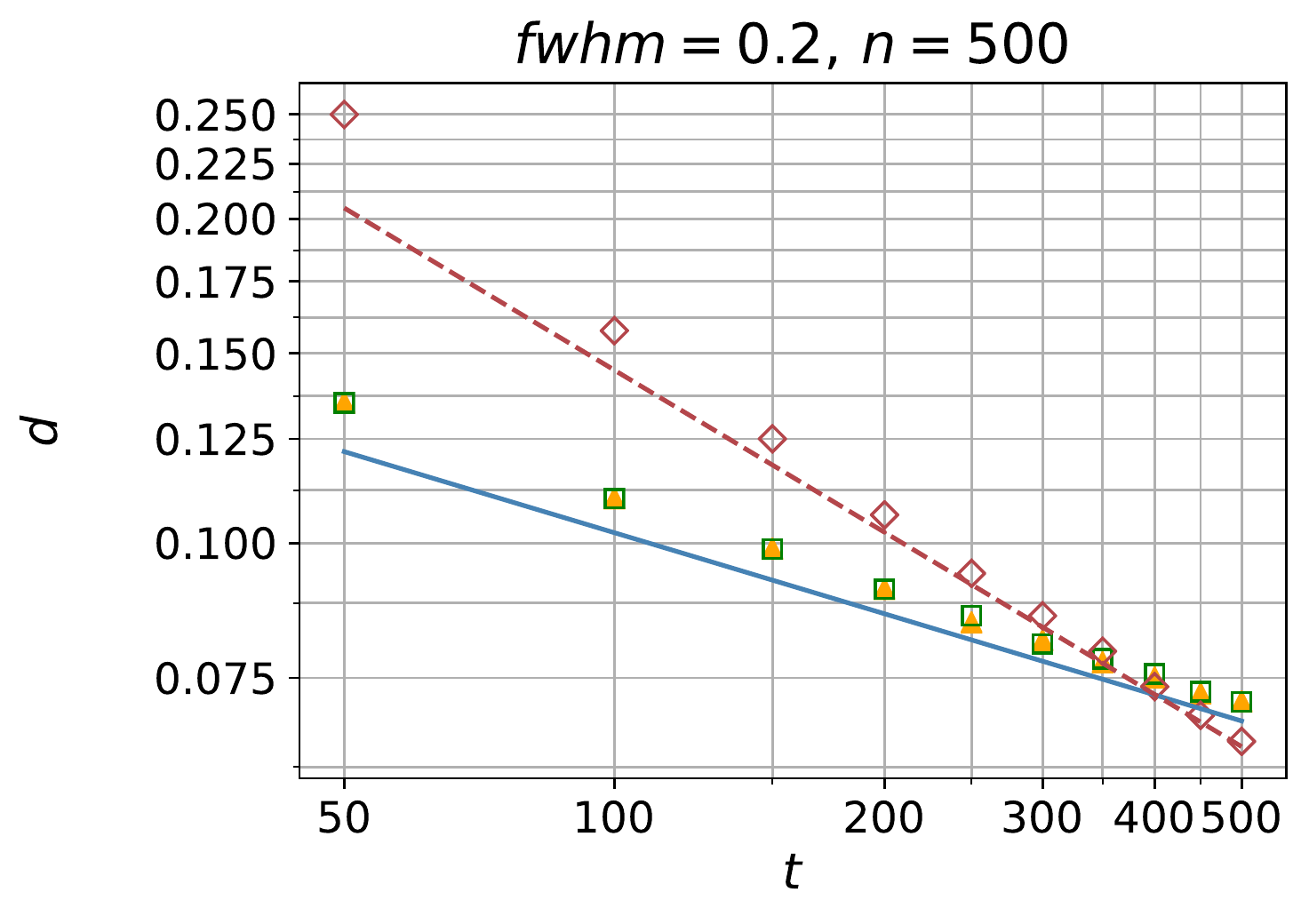}
		\caption{
		}
	\end{subfigure}
	\begin{subfigure}{.49\textwidth}
		\includegraphics[width=\linewidth]{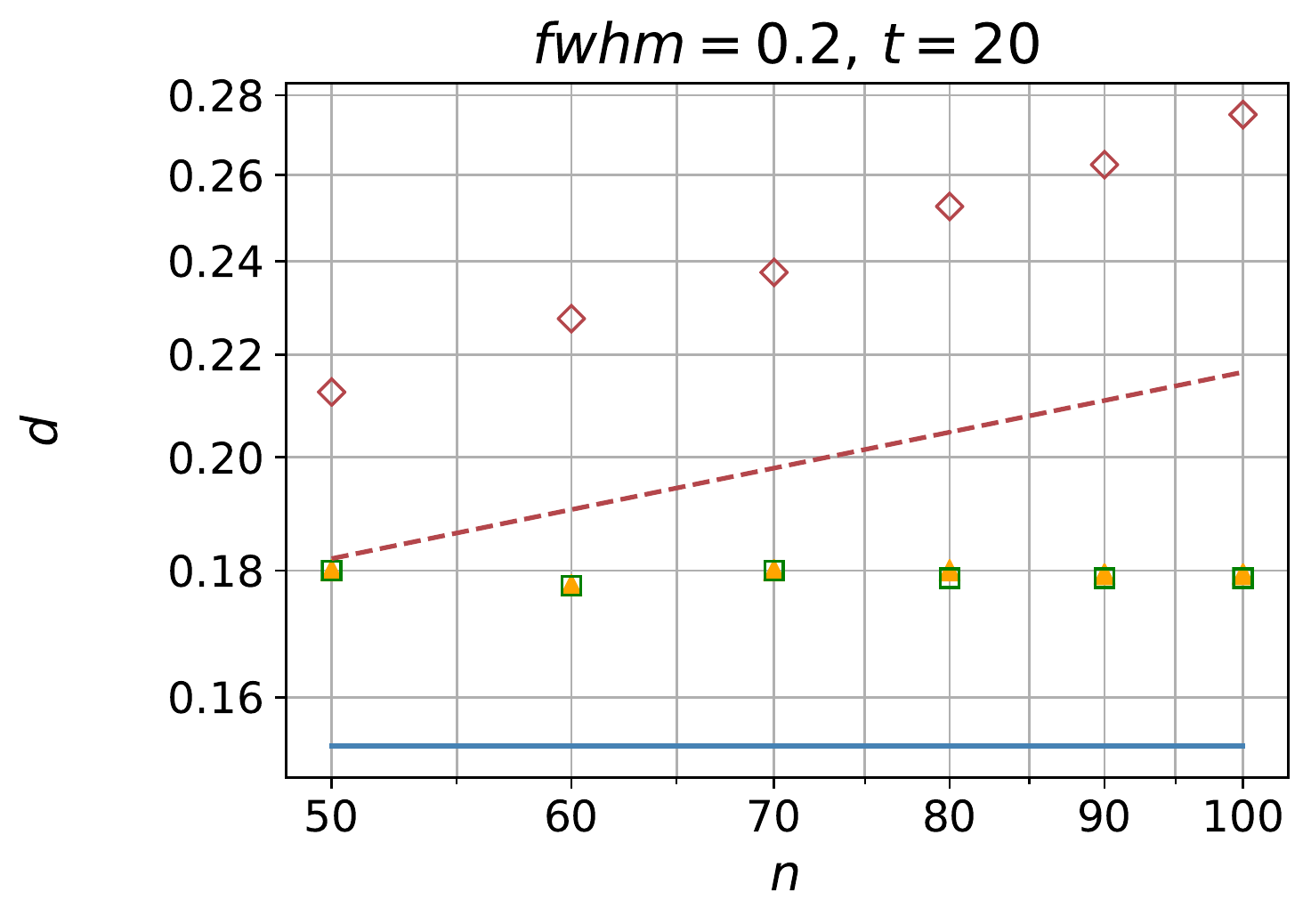}
		\caption{
		}
	\end{subfigure}
	\hspace{1mm}
	\begin{subfigure}{.49\textwidth}
		\includegraphics[width=\linewidth]{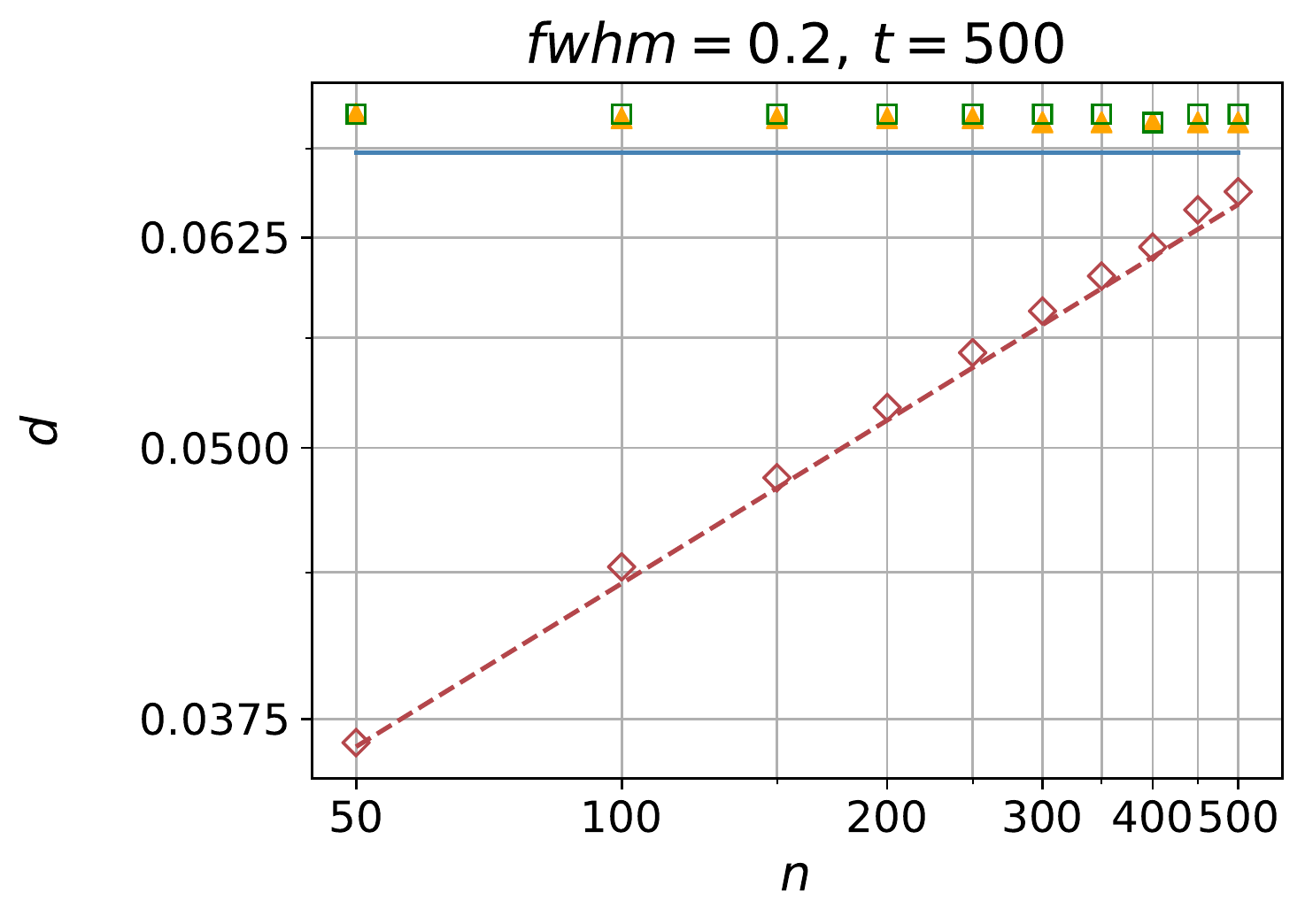}
		\caption{
		}
	\end{subfigure}
	\begin{subfigure}{.99\textwidth}
		\centering
		\includegraphics[width=.43\linewidth]{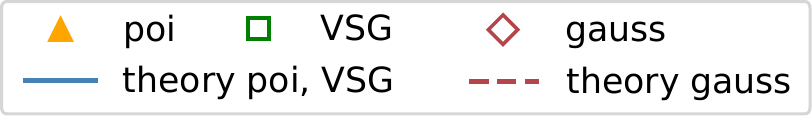}
	\end{subfigure}
	\caption{
		Simulations investigating finite sample validity of the asymptotic relations $d = 2.29\, t^{-1/2}n^{1/4} \fwhm^{5/4}$ \eqref{eq:detectionBoundaryHomGauss} for the homogeneous Gaussian model, and $d = 1.62\, t^{-1/4} \fwhm$ \eqref{eq:detectionBoundaryPoisson} for the VSG and Poisson models, see \Cref{model:poisson}. Here we have set $\alpha = 0.1$. For short illumination times $t$ and small discretizations $n$ only the slopes of theoretical formulas are close to the slopes obtained from simulations (left column, see also \Cref{tab:simulationResultsExperimental}). As $t$ and $n$ increase, the theoretical formulas become accurate approximations also in terms of absolute error (right column).
		\label{fig:simulationsSmallAndBig}
	}
\end{figure}

\bgroup
\def\arraystretch{1.3}
\begin{table}[h]
	\begin{center}
		\caption{Limiting asymptotic statistical resolution as given by \Cref{thm:main} for the Gaussian psf \eqref{eq:gaussKernel}
			\label{tab:simulationResultsExperimental} for small values of $t$ and $n$. The entries in $d(\fwhm)$ correspond to \Cref{fig:simulationsSmallAndBig} (a), in $d(t)$ to \Cref{fig:simulationsSmallAndBig} (c) and in $d(n)$ to \Cref{fig:simulationsSmallAndBig} (e).
		}
		\begin{tabular}{|l|l l l|}\hline
			\def\arraystretch{1.3}
			\multirow{2}{*}{Model} 
			& $d(\fwhm)_{emp}$ & $d(t)_{emp}$ & $d(n)_{emp}$\\   
			& $d(\fwhm)_{th} $ & $d(t)_{th} $ & $d(n)_{th}$\\ \hline
			\multirow{2}{*}{HG} 
			& $1.23\fwhm^{1.26}$ & $1.17\, t^{-0.665}$ & $ 0.0502\, n^{0.368} $\\   
			& $1.08\fwhm^{5/4}$ & $ 0.647\,t^{-1/2}  $ & $0.0685\, n^{1/4}$\\ \cline{1-4}
			%
			%
			Poisson & $0.879\fwhm^{0.979}$ & $ 0.519\,t^{-0.352}$ & $0.177 n^{0.00274}$ \\
			%
			VSG & $0.873\fwhm^{0.975}$ & $0.495\,t^{-0.336}$ & $0.183 n^{-0.00464}$ \\ \cline{2-4}  
			& $0.765\fwhm$ & $0.323\, t^{-1/4}$ & $0.153$\\ \hline
		\end{tabular}
	\end{center}
\end{table}
\egroup

\begin{figure}[!htb]
	\begin{subfigure}{.49\textwidth}
		\includegraphics[width=\linewidth]{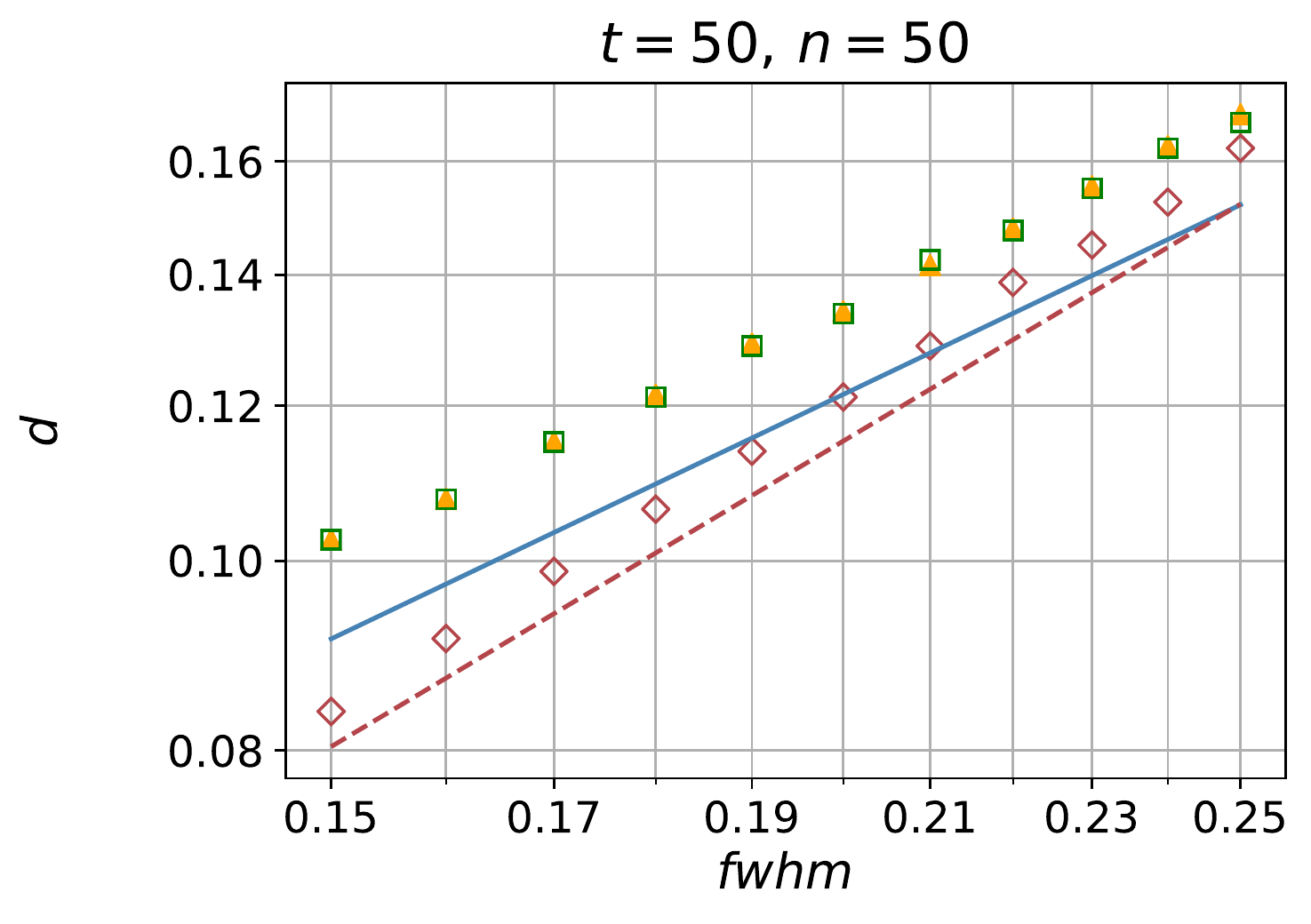}
		\caption{
		}
	\end{subfigure}
	\hspace{1mm}
	\begin{subfigure}{.49\textwidth}
		\includegraphics[width=\linewidth]{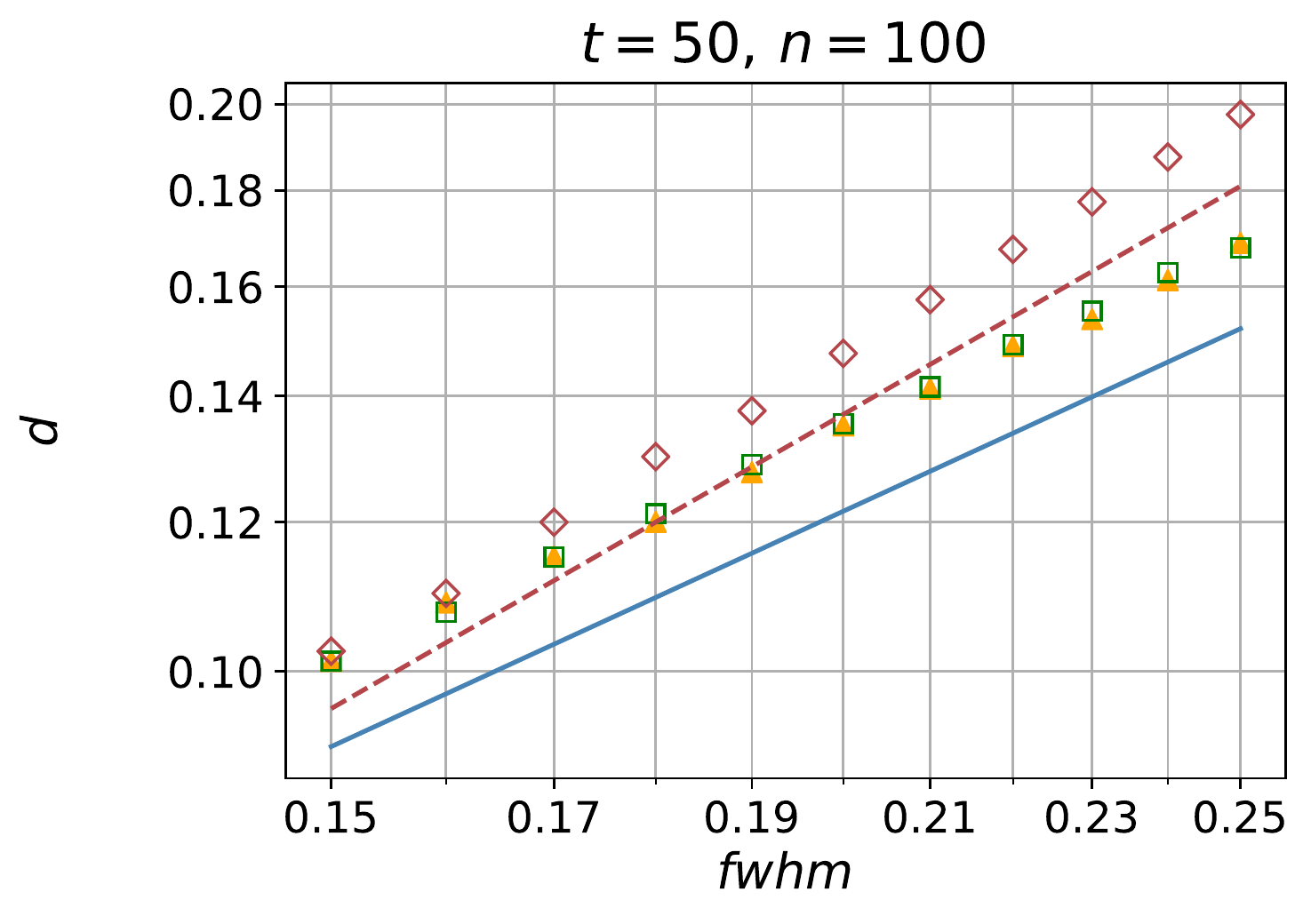}
		\caption{
		}
	\end{subfigure}\\
	\begin{subfigure}{.49\textwidth}
		\includegraphics[width=\linewidth]{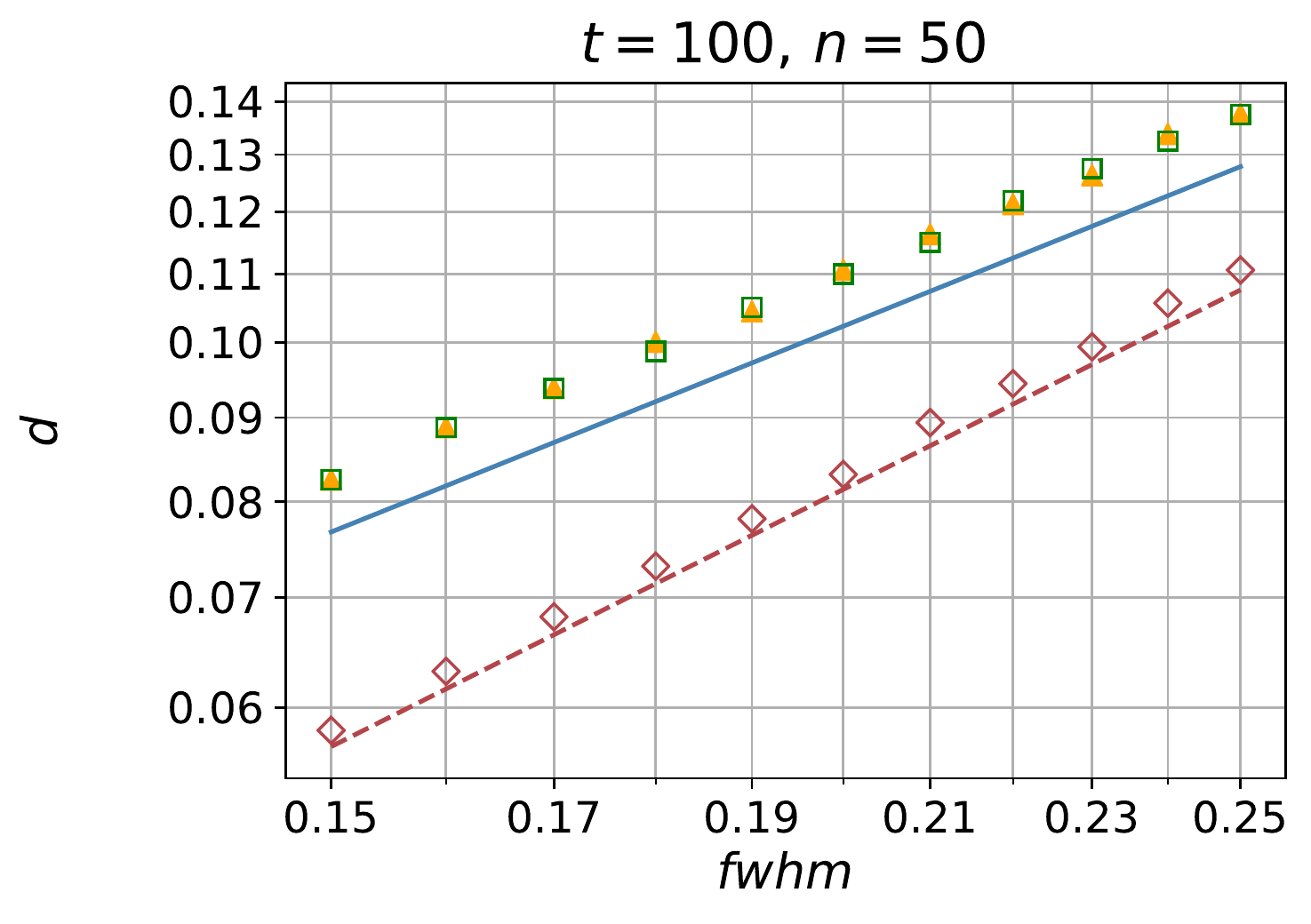}
		\caption{
		}
	\end{subfigure}
	\hspace{1mm}
	\begin{subfigure}{.49\textwidth}
		\includegraphics[width=\linewidth]{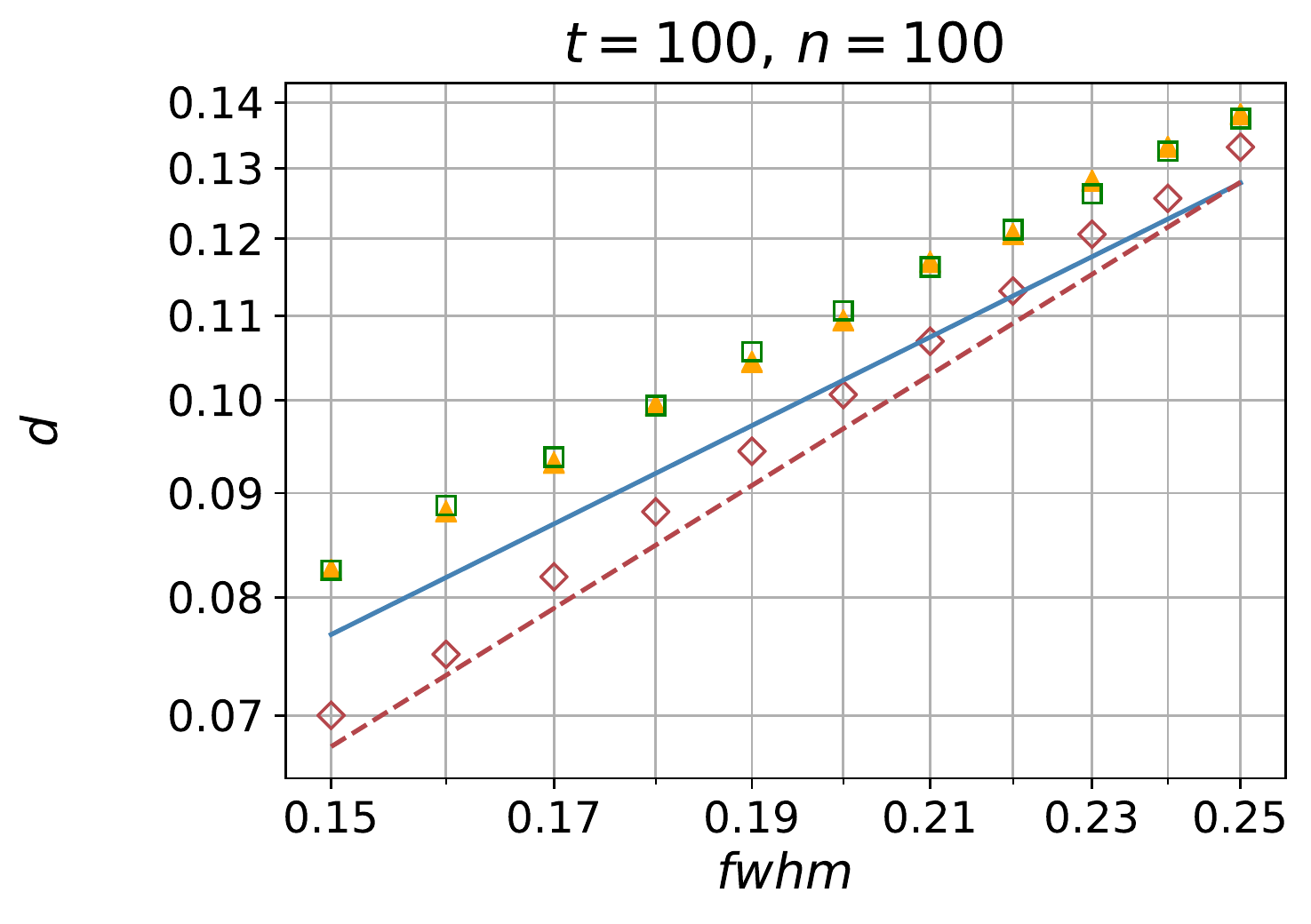}
		\caption{
		}
	\end{subfigure}
	\begin{subfigure}{.99\textwidth}
		\centering
		\includegraphics[width=.43\linewidth]{simulationLegend_withoutBinomial.pdf}
	\end{subfigure}
	\caption{
		Simulations investigating the finite sample validity of the asymptotic relations $d = 2.29\, t^{-1/2}n^{1/4} \fwhm^{5/4}$ \eqref{eq:detectionBoundaryHomGauss} for the homogeneous Gaussian model, and $d = 1.62\, t^{-1/4} \fwhm$ \eqref{eq:detectionBoundaryPoisson} for the VSG and Poisson models, see \Cref{model:poisson}. Here we have set $\alpha = 0.1$ and explored the intermediate parameter regime $t=50,100$ and $n=50,100$. 
		\label{fig:simulationsIntermediate}
	}
\end{figure}

	\newpage
	\appendix
	\section{Appendix: Proof of the main Theorem}\label{sec:mainResults} 
	
	In this section we will prove \Cref{thm:main}. Proofs of other results that may be skipped in the first reading are provided in \Cref{sec:auxProofs}.
	
	We have separated the proof into three parts---one for each of the three models defined in \Cref{sec:statisticalModel}. We start with the homogeneous and variance stabilized Gaussian models because the proof for the Poisson model relies on them.
	
	Before we start, let us introduce some notation. For functions $f(\cdot - x_0)\in L^1[0,1]$ let
	\begin{equation}\label{eq:integralOperator}
	\int_i f := \int_{(i-1)/n}^{i/n} f(x - x_0)\,\mathrm{d}x
	\qquad\text{and}\qquad
	\int_0^1 f := \int_{0}^{1} f(x - x_0)\,\mathrm{d}x.
	\end{equation}
	Mostly, we will use it for the psf $h$ and its derivatives. Note that we can rewrite \eqref{eq:p0iAbbrevation} as $p_{0i} = \int_i h$ and provided that \Cref{ass:h2} holds, we have
	\begin{equation}\label{eq:int_iGreater0}
	\int_i h \ge \min_{x\in[0,1]}\frac{h(x-x_0)}{n} \ge \frac{c}{n}
	\end{equation}
	for some constant $c>0$.
	
	\subsection{Homogeneous Gaussian model}\label{sec:gaussian}
	\begin{proof}[Proof of \Cref{thm:main} for the HG model]		\label{proof:Gauss}
		
		As $F_{t\theta} = \Gauss(t\cdot \theta, 1)$, the LRT statistic in \eqref{eq:LRTstat} becomes
		\[
		T_{t,n,d}\left(Y\right) = \log\left(\frac{\varphi\left(Y~|~H_1\right)}{\varphi\left(Y~|~H_0\right)}\right)
		=
		\frac{1}{2}\sum_{i=1}^n \left(t^2p_{0i}^2 - t^2p_{1i}^2 + 2 Y_i t\left(p_{1i} - p_{0i}\right)\right)
		\]
		with $\varphi$ the density of a standard normal variate. The corresponding likelihood ratio test \eqref{eq:LRT} is given by
		\begin{equation}\label{eq:gaussStatistic}
		\Phi_{t,n,d}\left(Y\right) := \begin{cases} 1 & \text{if}\quad T_{t,n,d} \left(Y\right) > q_{\alpha,t,n,d}^*, \\[0.1cm] 0 & \text{otherwise,} \end{cases}
		\end{equation}
		where $q_{\alpha,t,n,d}^* = \sqrt{2\mu_{t,n,d}} q_{1-\alpha} - \mu_{t,n,d}$ with $q_{1-\alpha}$ the $1-\alpha$ quantile of $\mathcal N \left(0,1\right)$ and
		\begin{equation}\label{eq:mu_nDefinition}
		\mu_{t,n,d} = \frac{t^2}{2}\sum_{i=1}^n \left(p_{1i} - p_{0i}\right)^2.
		\end{equation}
		
		For ease of readability, in the following we will only write the dependence on $n$, and omit indices $t$ and $d$.
		
		It holds that under $H_0: T_{n}\left(Y\right) \sim \mathcal N \left(-\mu_{n}, 2\mu_{n}\right)$
		and under $H_1: T_{n}\left(Y\right) \sim \mathcal N \left(\mu_{n}, 2\mu_{n}\right)$. We calculate
		\begin{align}
		\Prob{H_0}{\mathbf{reject}} &= \Prob{H_0}{T_{n}\left(Y\right) > q_{\alpha,n}^*}\nonumber\\
		&= \Prob{}{-\mu_{n} + \sqrt{2\mu_{n}} W > q_{\alpha,n}^*}
		= 1-\Prob{}{W \leq \frac{q_{\alpha,n}^* + \mu_{n}}{\sqrt{2\mu_{n}}}} = \alpha, \label{eq:gaussType1error}
		\end{align}
		where $W \sim \mathcal N \left(0,1\right)$. Thus, the test is indeed a level $\alpha$ test.
		
		We want the type II error to be equal to $\beta$. Thus, we require
		\begin{align}
		\beta = \Prob{H_1}{\mathbf{accept}} &= \Prob{H_1}{T_{n}\left(Y\right) \leq q_{\alpha,n}^*}   \nonumber  \\[0.1cm]
		&= \Prob{}{\mu_{n}  + \sqrt{2\mu_{n}}W \leq q_{\alpha,n}^*}
		= \Prob{}{W \leq q_{1-\alpha} -\sqrt{2\mu_{n}}},\label{eq:gaussType2error}
		\end{align}
		where again $W \sim \mathcal N \left(0,1\right)$. This implies that 
		\begin{equation}\label{eq:mu_nEqualquantiles}
		\mu_n = (q_{1-\alpha} - q_{\beta})^2/2.
		\end{equation}
		
		By definition of $\mu_n$ we have (recall \eqref{eq:integralOperator})
		\[
		\mu_{n} = \frac{t^2}{2}\sum_{i=1}^n \left(p_{1i} - p_{0i}\right)^2 
		=
		\frac{t^2}{2}\sum_{i=1}^n \left(\int_i \Delta\right)^2,
		\] 
		where
		\begin{equation*}
		\Delta(x-x_0):= \frac{1}{2} h(x-x_1)+ \frac{1}{2} h(x-x_2)-h(x-x_0)
		\end{equation*}
		is the difference between the psfs under $H_1$ and $H_0$.
		Since $h\in C^2[0,1]$,
		\begin{equation}\label{eq:Taylor series h}
		h(x-x_j)=\sum_{k=0}^2\frac{h^{(k)}(x-x_0)}{k!}(x_0-x_j)^k+o\left((x_0-x_j)^2\right).
		\end{equation}
		Hence, for $x_0 = (x_1+x_2)/2$
		\begin{align}\label{eq:Delta}
		\Delta(x-x_0)
		&=\frac{1}{2}\sum_{k=0}^2\frac{h^{(k)}(x-x_0)}{k!}(x_0-x_1)^k
		+
		\frac{1}{2}\sum_{k=0}^{2}\frac{h^{(k)}(x-x_0)}{k!}(x_0-x_2)^k-h(x-x_0)\nonumber\\
		&\hspace{8cm}+o\left((x_0-x_1)^2+(x_0-x_2)^2\right)\nonumber\\
		&=\frac{1}{4}\, h''(x-x_0)\left((x_0-x_1)^2 
		+
		(x_0-x_2)^2\right)
		+ h'(x-x_0)\big(x_0- (x_1 + x_2)/2\big)\nonumber\\
		&\hspace{8cm}
		+ o\left((x_0-x_1)^2+(x_0-x_2)^2\right) \\
		&=\frac{1}{8}h''\big(x-x_0\big)\,d^2
		+ o\left(d^2\right).\label{eq:Delta2}
		\end{align}
		Thus, we get
		\begin{align}
		\mu_{n} &= \frac{t^2}{2}\sum_{i=1}^n \left(p_{1i} - p_{0i}\right)^2 
		=
		\frac{t^2}{2}\sum_{i=1}^n \left(\int_i \Delta\right)^2
		= 
		\frac{t^2}{2}\sum_{i=1}^n\left( \frac{d^2}{8}\int_i h'' +o\left(\frac{d^2}{n}\right)\right)^2 \nonumber\\
		&= 
		\frac{t^2}{2}\sum_{i=1}^n \left(\frac{d^4}{64}\left(\int_i h''\right)^2  +  o\left(\frac{d^4}{n^2}\right)\right) 
		= 
		\frac{t^2}{2}\sum_{i=1}^n \frac{d^4}{64}\left(\int_i h''\right)^2  +  o\left(\frac{t^2 d^4}{n}\right) \label{eq:mu_nfiniteGauss}\\
		&=
		\frac{t^2 d^4}{128n} \int_{0}^{1}\left(h''\right)^2  +  o\left(\frac{t^2 d^4}{n}\right), \label{eq:mu_{n}}
		\end{align}
		applying \Cref{lem:summingRatio}. Rearranging \eqref{eq:mu_{n}} for $d$ and using \eqref{eq:mu_nEqualquantiles} we get the desired relation \eqref{eq:resolutionGauss}. Hence, $d$ as given in \eqref{eq:resolutionGauss} is the asymptotic statistical resolution.
	\end{proof}
	
	\begin{remark}\label{rmk:finiteNhomogeneous}
		In the derivation of \eqref{eq:mu_nfiniteGauss} only $d\searrow0$ (and hence $t\to\infty$) is required. Thus, for a finite $n$ we get 
		\[
		d \asymp 2\sqrt{2}\,\sqrt{ q_{1-\beta} - q_{\alpha} } \left(\sum_{i=1}^n \left(\int_i h''\right)^2 \right)^{-1/4} t^{-1/2}.
		\]
	\end{remark}

	\subsection{Variance stabilized Gaussian model}\label{sec:stabilizedGaussian}
	
	\begin{proof}[Proof of \Cref{thm:main} for the VSG model]\label{proof:VSGmain}
		Let $F_{2\sqrt{t\theta}} = \Gauss(2\sqrt{t\theta},1)$, i.e.
		\begin{equation*}
		Y_i \overset{\text{indep.}}\sim \Gauss\left(2\left(t\int_{(i-1)/n}^{i/n} g(x-x_0)\,\mathrm{d}x\right)^{1/2},1\right).
		\end{equation*}
		Then the log-likelihood function is
		\[
		T_{t,n,d}\left(Y\right) = \log\left(\frac{\varphi\left(Y~|~H_1\right)}{\varphi\left(Y~|~H_0\right)}\right)
		=
		\sum_{i=1}^n \left[2t\left(p_{0i} - p_{1i}\right) + 2 Y_i \sqrt{t}\left(\sqrt{p_{1i}} - \sqrt{p_{0i}}\right)\right]
		\]
		with $p_{\cdot i}$ defined in \Cref{eq:p0iAbbrevation,eq:p1iAbbrevation}.
		We define the corresponding likelihood ratio test 
		as in \eqref{eq:gaussStatistic}, but this time we set $q_{\alpha,t,n,d}^* = \sqrt{2\nu_{t,n,d}} q_{1-\alpha} - \nu_{t,n,d}$ with
		\[
		\nu_{t,n,d} = 2t\sum_{i=1}^{n}\left(\sqrt{p_{1i}}-\sqrt{p_{0i}}\right)^2.
		\]

		The proof is similar to the proof of the homogeneous Gaussian model in \Cref{proof:Gauss}. We again skip the indices $t$ and $d$ in what follows.
		
		We have under $H_0: T_{n}\left(Y\right) \sim \mathcal N \left(-\nu_{n}, 2\nu_{n}\right)$ and under $H_1: T_{n}\left(Y\right) \sim \mathcal N \left(\nu_{n}, 2\nu_{n}\right)$.
		We calculate
		\begin{align}
		\Prob{H_0}{\mathbf{reject}} &= \Prob{H_0}{T_{n}\left(Y\right) > q_{\alpha,n}^*}\nonumber\\
		&= \Prob{}{-\nu_{n} + \sqrt{2\nu_{n}} W > q_{\alpha,n}^*}
		= 1-\Prob{}{W \leq \frac{q_{\alpha,n}^* + \nu_{n}}{\sqrt{2\nu_{n}}}} = \alpha,\label{eq:vsgType1error}
		\end{align}
		where as previously $W \sim \mathcal N \left(0,1\right)$. Thus, the test is indeed a level $\alpha$ test.
		
		We want the type II error to be equal to $\beta$. Thus, we require
		\begin{align}
		\beta = 
		\Prob{H_1}{\mathbf{accept}} 
		&= \Prob{H_1}{T_{n}\left(Y\right) \leq q_{\alpha,n}^*}   \nonumber  \\[0.1cm]
		&= \Prob{}{ \nu_{n} + \sqrt{2\nu_{n}} W\leq\sqrt{2\nu_{n}} q_{1-\alpha} - \nu_{n}}  
		= \Prob{}{W \leq q_{1-\alpha} -\sqrt{2\nu_{n}}}.\label{eq:vsgType2error}
		\end{align}
		This implies that 
		\begin{equation}\label{eq:nu_nEqualquantiles}
		\nu_n = (q_{1-\alpha} - q_{\beta})^2/2 = (q_{1-\beta} - q_{\alpha})^2/2,
		\end{equation}
		since $q_{1-\gamma} = -q_{\gamma}$ for quantiles of $\mathcal{N}(0,1)$.
		On the other hand, by definition of $\nu_n$ we have
		\begin{equation*}
		\nu_{n} = 2t\sum_{i=1}^n\left(\sqrt{p_{1i}}-\sqrt{p_{0i}}\right)^2.
		\end{equation*}
		Using the Taylor series expansion \eqref{eq:Taylor series h} as $d\to0$ we get
		\begin{align}
		\left(\sqrt{p_{1i}}-\sqrt{p_{0i}}\right)^2 
		&= \left(\sqrt{\int_i h + \frac{d^2}{8}\int_i h''\, 
			+ o\left(\frac{d^2}{n}\right)}
		-\sqrt{\int_i h}\right)^2\nonumber\\
		&= \left(\sqrt{\int_i h}
		\sqrt{
			1 
			+\frac{d^2}{8}\frac{\int_i h''}{\int_i h}
			+ o\left(d^2\right)}
		-\sqrt{\int_i h}\right)^2  
		= \left(\frac{d^2}{16}\frac{\int_i h''}{\sqrt{\int_i h}}
		+ o\left(\frac{d^2}{\sqrt{n}}\right)\right)^2\nonumber\\
		&= \frac{d^4}{256}\frac{(\int_i h'')^2}{\int_i h}
		+ o\left(\frac{d^4}{n}\right), \label{eq:VSGtaylor}
		\end{align}
		where the $\left(\int_i h\right)^{-1}$ terms are well-defined by \eqref{eq:int_iGreater0}.
		Thus,
		\begin{align}\label{eq:nu_n}
		\nu_{n}
		&=
		2t \sum_{i=1}^n
		\left( \frac{d^4}{256}\frac{(\int_i h'')^2}{\int_i h}
		+ o\left(\frac{d^4}{n}\right)\right)
		= \frac{td^4}{128}\int_0^1\frac{(h'')^2}{h}
		+ o\left(td^4\right)
		\end{align}
		as $n\to\infty$ by \Cref{lem:summingRatio} with $f(x) = h''(x-x_0)$ and $g(x)=h(x-x_0)$.
		Rearranging the last equation for $d$ together with \eqref{eq:nu_nEqualquantiles} gives \eqref{eq:resolutionPoi}, as required.
	\end{proof}
	
	\begin{remark}\label{rmk:finiteNstabilized}
		Just as in the homogeneous Gaussian model (\Cref{rmk:finiteNhomogeneous}), we can keep $n$ finite in the above proof, provided that $d\searrow0$ and $t\to\infty$. Thus, for finite $n$ it holds
		\begin{equation*}
		d \asymp 2\sqrt{2}\,\sqrt{ q_{1-\beta} - q_{\alpha} } \left(\sum_{i=1}^n 
		\frac{\left(\int_i h''\right)^2}{\int_i h}
		\right)^{-1/4} t^{-1/4}.
		\end{equation*}
		From this equation by solving for $1-\beta$ and using \Cref{rmk:mainThmWithBackNoise} to incorporate constant background noise $\gamma > 0$, we can get an expression for the asymptotic power function ($t\to\infty$, $n=const.$)
		\begin{equation*}
		power(t,d) \asymp \Phi\left(q_{\alpha} + \sqrt{\sum_{i=1}^{n} \frac{(\int_i h'')^2}{\int_i (h + \gamma) }}\frac{d^2 \sqrt{t}}{8}\right).
		\end{equation*}
		This regime coincides with the one investigated in \citep{Acuna1997}, where Acu\~na and Horowitz looked at a 2D Poisson model on a line with constant background noise for telescopes. Their power function can be written as
		\begin{equation}\label{eq:powerAcuna}
		power(t, d) \asymp \Phi\left( q_{\alpha} + \sigma_0\frac{d^2\sqrt{t}}{8} \right)
		\end{equation}
		where $t$ is now interpreted as the telescope exposure time and $\sigma_0$ a constant which in 1D can be written as
		\[
		\sqrt{
			\kappa\sum_{i=1}^{n} 
			\left( \int_i (h + \gamma)\right)^{-1}
			\left(\frac{\partial^2}{\partial \left(\frac{d}{2}\right)^2}\Bigg\vert_{d=0}\left(\frac{1}{2}\int_{(i-1)/n}^{i/n} h\left(x-x_0 - \frac{d}{2}\right) + h\left(x-x_0 + \frac{d}{2}\right)\mathrm{d}x\right) \right)^2
		}
		\]
		with constant $\kappa >0$ describing the total intensity of the star in question.
		Reassuringly, this expression for $\sigma_0$ coincides (up to $\kappa$) for large $n$ with our factor $\sqrt{\sum_{i=1}^{n} \frac{(\int_i h'')^2}{\int_i (h + \gamma) }}$ by the mean value theorem, i.e. we are able to reproduce the main result of \cite{Acuna1997} from our more general ones, see also \Cref{sec:relatedWork}. 
	\end{remark}

	\subsection{Poisson model}\label{sec:poisson}
	
	The proof for the Poisson model is split into two parts. If $ t \gg \sqrt{n} \log^8 n $, then the Le Cam asymptotic equivalence between the Poisson and the VSG models holds and thus the proof follows from the VSG model. If $ t \ll n^{2-\delta}$ for some $\delta >0$, then a CLT holds and we can prove \Cref{thm:main} (a) directly.
	
	\subsubsection{Analysis in the asymptotic equivalence regime}\label{sub:PoiAsymptoticEquivalence}
	
	We briefly recall the theory of asymptotic equivalence developed by Le Cam \citep{LeCam1986}, \citep{LeCam2000}. We mostly follow the presentation of \citep{Grama2002}. In our context we consider a \textit{statistical experiment}---a set
	\[
	\mathcal{E} = (X, \mathcal{X}, \{ P_{\theta} : \theta \in \Theta\} ),
	\]
	where $(X, \mathcal{X})$ is a measurable space with the parameter set $\Theta\subset\reals$, a possibly unbounded interval, and $P_{\theta}$ is an absolutely continuous probability measure with respect to some dominating $\sigma$-finite measure $\mu$. Consider a second, possibly easier to tackle, experiment $\mathcal{G} = (Y, \mathcal{Y}, \{ Q_{\theta}:\theta\in\Theta\})$ over the same parameter space $\Theta$.
	Let further $(D, \mathcal{D})$ be a measurable space of possible decisions. Then the set of Markov kernels $\kappa:(X,\mathcal{X}) \to (D,\mathcal{D})$ is the set of randomized decision procedures for the experiment $\mathcal{E}$. We denote it by $\Pi(\mathcal{E})$. We let $\mathcal{L}(D, \mathcal{D})$ to be the set of all loss functions $L:\Theta\times D \to [0, \infty)$ such that $0 \le L(\theta, z) \le 1$ for all $\theta\in\Theta$ and $z\in D$. Given a decision procedure $\kappa \in \Pi(\mathcal{E})$, the true value $\theta\in\Theta$ and a loss function $L \in \mathcal{L}(D, \mathcal{D})$, the risk is
	\[
	R(\mathcal{E}, \kappa, L, \theta) = \int_X \int_D L(\theta, z) \kappa(x, \mathrm{d}z) P_{\theta} (\mathrm{d} x).
	\]
	We define the deficiency as
	\[
	\delta (\mathcal{E}, \mathcal{G}) 
	:= 
	\sup
	\sup_{L\in\mathcal{L}(D, \mathcal{D})}
	\inf_{\kappa_1 \in \Pi (\mathcal{E})}
	\sup_{\kappa_2 \in \Pi (\mathcal{G})}
	\sup_{\theta\in\Theta}
	|R(\mathcal{E}, \kappa_1, L, \theta)
	- R(\mathcal{G}, \kappa_2, L, \theta) |
	\]
	with the first supremum ranging over all possible decision spaces $(D, \mathcal{D})$. Since deficiency is asymmetric, we define the Le Cam (pseudo) distance as 
	\[
	\Delta(\mathcal{E}, \mathcal{G}) 
	:=
	\max \{ \delta(\mathcal{E}, \mathcal{G} ),
	\delta(\mathcal{G}, \mathcal{E} ) \}.
	\]
	
	\begin{definition}
		Two sequences of statistical experiments $\mathcal{E}^n$ and $\mathcal{G}^n$, $n\in\nat$, are \textit{asymptotically equivalent} if
		\[
		\Delta( \mathcal{E}^n, \mathcal{G}^n ) \to 0. 	
		\]
	\end{definition}
	We can summarize the implications of the above definition for our analysis in the following proposition.
	\begin{proposition}
		Let $\mathcal{E}^n_1$ and $\mathcal{E}^n_2$, $n\in\nat$, be two sequences of statistical experiments that are asymptotically equivalent, and let $\Psi^n_1$ and $\Psi^n_2$ be the corresponding optimal tests. Then we have
		\[
		\E{H_0}{\Psi^n_1} \to \alpha \quad\text{and}\quad \E{H_1}{\Psi^n_1} \to 1-\beta
		\quad\iff\quad
		\E{H_0}{\Psi^n_2} \to \alpha \quad\text{and}\quad \E{H_1}{\Psi^n_2} \to 1-\beta,
		\]
		i.e. the type I error of $\Psi^n_1$ converges to $\alpha$ and the type II error to $\beta$ if and only if the type I error of $\Psi^n_2$ converges to $\alpha$ and type II error to $\beta$.
		Thus, an asymptotic resolution sequence in the sense of \Cref{def:statResolution} for the first sequence of experiments will also be an asymptotic resolution for the second sequence.
	\end{proposition}
	
	The above proposition allows us to transfer the VSG result to the Poisson model in the asymptotic equivalence regime:
	
	\begin{corollary}\label{thm:gaussEquivalentToPoisson} 
		Let $0<\alpha,\beta<1/2$ be type I and II errors, respectively. Assume that $\sqrt{n} \log^{8} n=o\left(t\right)$ and \Cref{ass:h2} are valid. Then \Cref{thm:main} (a) holds.
	\end{corollary}
	
	\begin{proof}
		Our \hyperref[model:VSG]{VSG model} can be viewed as a Gaussian model
		\[
		Y_i \overset{\text{indep.}}\sim \Gauss\left(2\sqrt{f_n(i/n)},1\right)
		\]
		with 
		\begin{equation}\label{eq:f_nDefinition}
		f_n(x)=t\int_{x-1/n}^{x}g(y)\,\mathrm{d}y
		\end{equation} for $x\in[1/n, 1]$.
		According to Example 4.2 of \citep{Grama2002}, a sequence of $n$ Poisson observations
		\begin{equation*}
		X_i \overset{\text{indep.}}\sim \Poi\left(f(i/n)\right)
		\end{equation*}
		is asymptotically equivalent to the above Gaussian model with some \textit{fixed function} $f:[0,1]\to\reals$ provided that $f$ is bounded $c_1\le f(x) \le c_2$ by some absolute constants $c_1,c_2>0$ and it is H\"older with exponent $\beta > 1/2$. This result was extended in Theorem 4 of \citep{Ray2018} to include functions $f$ which are not bounded away from zero: functions $f=f_n$ that may depend on $n\in\nat$, satisfy
		\begin{equation}\label{eq:lowerBoundSchmidt}
		\inf_{x\in[0,1]} f_n(x) \gg n^{-\beta/(\beta+1)}\log^8 n
		\end{equation}
		and $f_n$ are H\"older with $ 1/2<\beta\le 1$. Thus, we only need to extend our $f_n$'s to functions on $[0,1]$ and prove that they satisfy the assumptions of Theorem 4 of \citep{Ray2018} to complete the proof.
		
		As a first step, extend the image function $g$ to a function on $C^2[-1/n, 1]$ such that
		\begin{equation}\label{eq:gExtension}
		n^{-1/2} \log^{8+\delta} n \le t \int_{-1/n}^{0} g(y)\,\mathrm{d}y \le t,
		\end{equation}
		for some $\delta > 0$. Then we can extend $f_n$'s in \eqref{eq:f_nDefinition} to $f_n:[0,1]\to\reals$. 
		
		We have that $f_n\le t$ since $\int_0^1 g = 1$ and $g > 0$, and $f_n\in C^3[0,1]$ since $g\in C^2[-1/n,1]$. Hence, $f_n$ is H\"older with $\beta = 1$.
		Due to the psf $h$ being fixed, our testing problem \eqref{eq:H0vsH1} and \Cref{ass:h2}, for all $x\in [1/n,1]$ it holds that
		\[
		\int_{x-1/n}^{x} g(y)\,\mathrm{d}y \ge \frac{\min_{x\in[0,1]} g(x)}{n}.
		\]
		Due to the continuity of $g$ and compactness of $[0,1]$, $\min_{x\in[0,1]} g(x) \ge c$ for some constant $c>0$. Thus,
		\[
		\inf_{x\in[0,1]} f_n = \inf_{x\in[0,1]} t \int_{x-1/n}^x g\left(y\right) \,\mathrm d y \geq \min\left\{c \frac{t}{n}, n^{-1/2} \log^{8+\delta} n \right\} \gg n^{-\frac12} \log^8\left(n\right),
		\]
		by \eqref{eq:gExtension} and our assumption $\sqrt{n} \log^{8} n = o(t)$, thereby showing that \eqref{eq:lowerBoundSchmidt} holds. Therefore, in this case the Poisson model is equivalent to the VSG model for which \Cref{thm:main} (a) holds by the above proof.
	\end{proof}
	
	\subsubsection{Analysis in the central limit theorem regime}\label{sec:CLTanalysis}
	
	To finish the proof of \Cref{thm:main} under the Poisson model, in this section we will prove a CLT in a different parameter regime than in the regime treated previously based on asymptotic equivalence. The regimes of present and previous sections cover the whole parameter domain, thereby completing the proof.

	Here we have (recall \eqref{eq:model}) $F_{t\theta} = \Poi(t\theta)$, or more explicitly
	\begin{equation}\label{eq:poissonRandomVars}
	Y_i \overset{\text{indep.}}\sim \Poi(\lambda_i )
	\quad
	\text{with}
	\quad
	\lambda_i = t \int_{(i-1)/n}^{i/n} g(x)\,\mathrm{d}x.
	\end{equation}
	Note that $\lambda_{1i} = t p_{1i}$ and $\lambda_{0i} = t p_{0i}$. 
	The likelihood ratio statistic for \eqref{eq:H0vsH1} under the model \eqref{eq:poissonRandomVars} is
	\begin{align*}
	\statPoi(Y) 
	&=
	\log\left(\prod_{i=1}^{n} e^{-(\lambda_{1i}-\lambda_{0i})}\left(\frac{\lambda_{1i}}{\lambda_{0i}}\right)^{Y_i}\right)
	= \sum_{i=1}^{n}Y_i\log\left(\frac{\lambda_{1i}}{\lambda_{0i}}\right).
	\end{align*}

	\begin{theorem}[CLT for Poisson LR]\label{thm:CLTforPoi}
		Assume a psf $h$ satisfies \Cref{ass:h2} and that $n = t^{1/2+\delta}$ for some $\delta > 0$.
		Then a CLT holds for $\statPoi(Y)$ under the hypothesis \eqref{eq:H0} and the alternative \eqref{eq:H1} as $t,n\to\infty$ and $d\to 0$, i.e.
		\[
		\frac{\statPoi-\Exp{\statPoi}}{\sqrt{\Var{\statPoi}}} 
		\stackrel{\mathcal{D}}{\to}
		\mathcal{N}(0,1).
		\]
	\end{theorem}
	
	
	\begin{proof}[Proof of \Cref{thm:CLTforPoi}]
		We apply the Lindeberg-Feller CLT for triangular arrays (see \citep{Billingsley1986}). For ease of readability, we again skip indices $t$ and $d$ in what follows.
		
		Let
		\[
		X_{ni} = a_i Y_i,
		\]
		so that
		\[
		T_{n}(Y) = \sum_{i=1}^{n}X_{ni},
		\]
		where as before
		\begin{equation*}
		Y_i \overset{\text{indep.}}\sim \Poi(\lambda_i ) 
		\quad\text{ and }\quad
		a_i := \log\left(\frac{\lambda_{1i}}{\lambda_{0i}}\right).
		\end{equation*}
		Note that $\lambda_{\cdot i}$ and $a_i$ depend on $n$ as well. We also set $\mu_{ni} = \Exp{ X_{ni}}$, $\sigma^2_{ni} = \Var{X_{ni}}$ and $\tau_n^2 = \sum_{i=1}^{n}\sigma_{ni}^2$. We need to show (see e.g. \citep{Billingsley1986}) that $\sigma_{ni}^2<\infty$ and that for all $\veps > 0$ we have
		\[
		L_n(\veps) = \frac{1}{\tau_n^2}\sum_{i=1}^{n}\int (x-\mu_{ni})^2\, \indicator_{\{|x-\mu_{ni}| > \veps \tau_n\}} \mathrm{d} \Prob{X_{ni}}{x} \to 0 \quad\text{as } n\to\infty.
		\]
		We use the Taylor approximation $\log(1+y_i) = \sum_{k=0}^2(-1)^k/(k+1)y_i^{k+1} + o\left(y_i^3\right)$ together with \eqref{eq:Delta2} to get
		\begin{equation}\label{eq:definitionOfy_iPoi}
		y_i 
		= \frac{\lambda_{1i}}{\lambda_{0i}} - 1
		= \frac{\lambda_{1i}-\lambda_{0i}}{\lambda_{0i}} 
		= \frac{\int_i\Delta}{\int_i h}
		= \frac{1}{8}\frac{\int_i h''}{\int_i h}d^2
		+ \frac{1}{384}\frac{\int_i h''''}{\int_i h}d^4
		+ o\left(d^4\right).
		\end{equation}
		The $\left(\int_i h\right)^{-1}$ terms are well-defined by \eqref{eq:int_iGreater0}.
		
		Under the hypothesis $H_0$ it holds
		\begin{align*}
		\mu_{ni} 
		&= \E{H_0}{X_{ni}} 
		= a_i\E{H_0}{Y_i}
		= a_i \lambda_{0i}
		= a_i t\int_i h
		= \log(1+y_i) t \int_{i}h\\
		&=\frac{t d^2}{8}\int_i h'' 
		+ t d^4\left(-\frac{1}{128}\frac{(\int_i h'')^2}{\int_i h} 
		+ \frac{1}{384}\int_i h''''\right) 
		+ O\left(\frac{td^6}{n}\right),\\	
		\nu_n 
		&= \sum_{i=1}^{n}\mu_{ni} 
		= \E{H_0}{T_n(Y)}
		= t\sum_{i=1}^{n}a_i \int_i h
		= \frac{t d^2}{8}\int_0^1 h'' 
		+ t d^4\left(-\frac{1}{128}\ratioSum_n
		+ \frac{1}{384}\int_0^1 h''''\right) 
		+ O\left(td^6\right),\\
		\sigma^2_{ni} 
		&= \V{H_0}{X_{ni}} 
		= a_i^2\,\V{H_0}{Y_i}
		= a_i^2  \lambda_{0i}
		=  t\log(1+y_i)^2\int_i h
		= \frac{td^4}{64}\frac{(\int_i h'')^2}{\int_i h} 
		+ O\left(\frac{td^6}{n}\right),\\
		\tau_n^2 
		&= \sum_{i=1}^{n}\sigma_{ni}^2
		= \V{H_0}{T_n(Y)}
		= t\sum_{i=1}^{n} a_i^2\int_i h 
		= t\sum_{i=1}^{n} \log(1+y_i)^2\int_i h 
		= \frac{t d^4}{64} \ratioSum_n + O(td^6)
		\end{align*}
		with 
		\begin{equation*}
		\ratioSum_n := \sum_{i=1}^{n}\frac{\left(\int_{(i-1)/n}^{i/n}h''(x-x_0)\,\mathrm{d}x   \right)^2}{\int_{(i-1)/n}^{i/n} h(x-x_0)\,\mathrm{d}x}.
		\end{equation*}
		Clearly, it holds that $\sigma_{ni}^2 <\infty$. 
		Applying \Cref{lem:summingRatio} with $f(x)=h''(x-x_0)$ and $g(x)=h(x-x_0)$ we see that
		\[
		\ratioSum_n
		=
		\int_{0}^{1}\frac{h''(x-x_0)^2}{h(x-x_0)}\,\mathrm{d}x
		+
		o(1) < \infty
		\]
		and hence
		\[
		\tau_n^2 = \frac{td^4}{64}\int_{0}^{1}\frac{h''(x-x_0)^2}{h(x-x_0)}\,\mathrm{d}x
		+
		O(td^6). 
		\]
		We consider
		\begin{align}\label{eq:Ln}
		L_n(\veps) 
		&= \frac{1}{\tau_n^2}\sum_{i=1}^{n}\sum_{\substack{k\in a_i\nato\\|k-\mu_{ni}|>\veps\tau_n}} (k-\mu_{ni})^2\,\Prob{H_0}{X_{ni}=k}\nonumber\\
		&= \frac{1}{\tau_n^2}\sum_{i=1}^{n}\sum_{\substack{l\in\nato\\
				|a_i l-\mu_{ni}|>\veps\tau_n}} (a_i l-\mu_{ni})^2\,\Prob{H_0}{Y_i=l}.
		\end{align}
		Note that if $a_i = 0$, then $ \left|a_i l - \mu_{ni}\right| = 0$. If $a_i \ne 0$, then the condition $ \left|a_i l - \mu_{ni}\right| > \veps\tau_n $ on $l$ is equivalent to $l\in\domain$, where $\domain$ is the set consisting of all $l\in\nato$ satisfying
		\begin{equation*}
		\begin{dcases}
		l 
		>
		\veps\sqrt{t}\frac{\sqrt{\sum_{i=1}^{n}a_i^2\int_i h}}{|a_i|} + t\int_i h\\
		l 
		<
		-\veps\sqrt{t}\frac{\sqrt{\sum_{i=1}^{n}a_i^2\int_i h}}{|a_i|} + t\int_i h.
		\end{dcases}
		\end{equation*}
		It holds that
		\[
		a_i = \log(1+y_i) 
		= y_i + O(y_i^2) 
		= \frac{\int_i h''}{\int_i h}\frac{d^2}{8} + o\left(d^2\right),
		\]
		\begin{equation}\label{eq:boundOna_i}
		\sum_{i=1}^{n}a_i^2\int_i h
		= \sum_{i=1}^{n}\left( y_i^2 + O(y_i^3) \right) \int_i h
		= \sum_{i=1}^{n} 
		\frac{(\int_i h'')^2}{\int_i h}\frac{d^4}{64}
		+ o\left(d^4\right)
		\end{equation}
		and thus
		\[
		R_i:=\frac{\sqrt{\sum_{i=1}^{n}a_i^2\int_i h}}{|a_i|} = O(1).
		\]
		Hence, the domain $\domain$ is a subset of those indices $l\in\nato$ such that
		\begin{equation}\label{eq:poissonProofDomain2}
		\begin{cases}
		\veps\sqrt{t}R + \frac{t}{n}q < l < \infty \\
		0 \le l < -\veps\sqrt{t}R + \frac{t}{n}\bar{q},
		\end{cases}
		\end{equation}
		where $R=\min_{i\in\{1,\ldots,n\}} R_i$, $q = \min_{x\in[0,1]} h(x-x_0)$ and $\bar{q} = \max_{x\in[0,1]} h(x-x_0)<\infty$, since $h(\cdot - x_0)\in C^4 [0,1]$.
		For $n=t^{1/2+\delta}$ with $\delta>0$ arbitrary, there are no $l$'s satisfying the second inequality of \eqref{eq:poissonProofDomain2} for sufficiently large $t$.
		Hence, setting $l_0 = \lceil \sqrt{t}(\veps R + t^{-\delta}q) \rceil$ it holds that
		\begin{align*}
		L_n(\veps) 
		&\le
		\frac{1}{\tau_n^2}
		\mathlarger{\sum}_{i=1}^{n}
		\mathlarger{\sum}_{l=l_0}^{\infty}
		a_i^2 \left(l-t \int_i h\right)^2
		e^{-\lambda_{i0}} \frac{\lambda_{i0}^l}{l!}=
		\frac{t^2}{\tau_n^2}
		\mathlarger{\sum}_{i=1}^{n} a_i^2
		\mathlarger{\sum}_{l=l_0}^{\infty}
		\left(\frac{1}{t}- \frac{\int_i h}{l}\right)^2
		e^{-\lambda_{i0}} \frac{\lambda_{i0}^ll}{(l-1)!}\nonumber
		.
		\end{align*}
		Moreover,
		
		\[
		\left(\frac{1}{t}- \frac{\int_i h}{l}\right)^2 
		\le 
		\frac{1}{t^2} - 2\frac{q}{t l n} + \frac{\bar{q}^2}{l^2 n^2}
		=  o(1).
		\]
		Note that \eqref{eq:boundOna_i} also implies that
		\[
		\sum_{i=1}^{n}a_i^2 = O(nd^4).
		\]
		Thus, we have that
		\begin{align*}
		L_n(\veps) 
		\le
		c' \frac{t^2}{\tau_n^2}
		\sum_{i=1}^{n} a_i^2
		\sum_{l=l_0}^{\infty}
		e^{-\lambda_{i0}} \frac{\lambda_{i0}^ll}{(l-1)!}
		\le
		c n t 
		\sum_{l=l_0}^{\infty}
		\frac{l}{(l-1)!}\left(\frac{t}{n}\bar{q}\right)^l
		=
		c t^{3/2+\delta}
		\sum_{l=l_0}^{\infty}
		\frac{l}{(l-1)!}\left(t^{1/2-\delta}\bar{q}\right)^l,
		\end{align*}
		for some constants $c, c' > 0$.	Consider
		\begin{align*}
		\sum_{l=a}^{\infty}\frac{(t^{1/2-\delta}\bar{q})^l}{l!}
		&=
		\frac{\left(t^{1/2-\delta}\bar{q}\right)^a}{a!}
		\left(1+\sum_{l=a+1}^{\infty}
		\frac{(t^{1/2-\delta}\bar{q})^l}{l!}
		\frac{a!}{(t^{1/2-\delta}\bar{q})^a}
		\right)\\
		&=
		\frac{\left(t^{1/2-\delta}\bar{q}\right)^a}{a!}
		\left(
		1+
		\frac{t^{1/2-\delta}\bar{q}}{a+1} + \frac{( t^{1/2-\delta}\bar{q})^2}{(a+1)(a+2)}+
		\ldots
		\right).
		\end{align*}
		Setting $a=\lceil\sqrt{t}(\veps+t^{-\delta})\rceil$, second and further terms in the brackets are of order $(\veps^{-1}t^{-\delta})^k$ and so we get
		\[
		1+\sum_{k=1}^{\infty}(\veps^{-1} t^{-\delta})^{k} = \frac{1}{1-\veps^{-1} t^{-\delta}} = O(1).
		\]
		Using Stirling's approximation $\log m! = m \log m - m + O(\log m)$
		we have that	
		\begin{align*}
		\frac{\left(t^{1/2-\delta}\right)^a}{a!} 
		&=
		\frac{\exp\left(a\log(t^{1/2-\delta})\right)}{\exp(a\log a - a + O(\log a))}
		=
		\exp\left(a\left(\log t^{1/2-\delta} - \log a + 1\right) + O(\log a)\right)\\
		&=
		\exp
		\left(
		\sqrt{t}(\veps+t^{-\delta})
		\left(
		-\log(t^{\delta}\veps+1)
		+1
		+o(1)
		\right)
		\right)
		= O\left((t^{\delta}\veps+1)^{-\sqrt{t}(\veps+t^{-\delta})}\right)
		.
		\end{align*}
		In our case the terms are of the form
		\[
		t^{1-2\delta}\bar{q}^2\sum_{k=a}^{\infty}\frac{(t^{1/2-\delta}\bar{q})^k}{k!}\frac{k+2}{k+1},
		\]
		with $a=\left\lceil \sqrt{t} ( \veps R + t^{-\delta}q )  \right\rceil-2
		\sim\lceil\sqrt{t}(\veps+t^{-\delta})\rceil$ and $(k+2)/(k+1)\le 2$. Thus, the above considerations apply and all together we get
		\[
		L_n(\veps)
		\le
		O
		\left(
		t^{5/2-\delta}
		(t^{\delta}\veps+1)^{-\sqrt{t}(\veps+t^{-\delta})})
		\right)\to 0 \quad \text{as}\quad t,n\to\infty,\; d\to0.
		\]
		
		Under the hypothesis $H_1$ we have
		\begin{align*}
		\mu_{ni} 
		&= \E{H_1}{X_{ni}} = a_i \lambda_{1i} 
		= a_i(1+y_i)\lambda_{0i}
		= \frac{t d^2}{8}\int_i h'' 
		+ t d^4\left(\frac{1}{128}\frac{(\int_i h'')^2}{\int_i h}
		+ \frac{1}{384}\int_i h''''\right) 
		+ O\left(\frac{td^6}{n}\right),\\	
		\nu_n 
		&= \sum_{i=1}^{n}\mu_{ni} 
		= \E{H_1}{T_n(Y)}
		= \sum_{i=1}^{n} a_i(1+y_i)\lambda_{0i}
		= \frac{t d^2}{8}\int_0^1 h'' 
		+ t d^4\left(\frac{1}{128}\ratioSum_n 
		+ \frac{1}{384}\int_0^1 h''''\right) 
		+ O\left(td^6\right),\\
		\sigma^2_{ni} 
		&= \V{H_1}{X_{ni}} 
		= a_i^2 \lambda_{1i}
		= a_i^2(1+y_i)\lambda_{0i}
		= \frac{td^4}{64}\frac{(\int_i h'')^2}{\int_i h} + O\left(\frac{td^6}{n}\right),\\
		\tau_n^2 
		&= \sum_{i=1}^{n} a_i^2(1+y_i)\lambda_{0i}
		=\sum_{i=1}^{n}\sigma_{ni}^2 
		= \frac{td^4}{64} \ratioSum_n 
		+ O\left(td^6\right)
		\end{align*}
		and hence similar considerations prove Lindeberg's condition in this case.
	\end{proof}	
	
	\begin{remark}\label{rmk:CLTequivPoi}
		Due to 
		$\sigma^2_{ni}/\tau_n^2\to 0$
		Lindeberg's condition is necessary for the CLT to hold.
	\end{remark}
	
	Now we can analyze the Poisson LRT
	\begin{equation*}
	\Phi_{t,n,d}(Y):=
	\begin{cases}
	1\quad\text{if}\;\; \statPoi(Y) > q_{\alpha,t,n,d}^*,\\
	0\quad\text{otherwise},
	\end{cases}
	\end{equation*}
	in the CLT regime above.
	Here
	\begin{equation}\label{eq:q_defintion_Poisson}
	q_{\alpha, t,n,d}^* := q_{1-\alpha}\sqrt{\V{H_0}{\statPoi}}+\E{H_0}{\statPoi}.
	\end{equation}

	\begin{proof}[Proof of \Cref{thm:main} (Poisson model in the CLT regime)]
		Again we skip the indices of $t$ and $d$.
		
		We want to find such $q_{\alpha, n}^*$ that
		\begin{equation}\label{eq:alphaPoi}
		\Prob{H_0}{\mathbf{reject}} 
		= \Prob{H_0}{T_{n}\left(Y\right) > q_{\alpha,n}^*} = \alpha
		\end{equation}
		and
		\begin{equation}\label{eq:betaPoi}
		\Prob{H_1}{\mathbf{accept}} 
		= \Prob{H_1}{T_{n}\left(Y\right) \le q_{\alpha,n}^*} = \beta
		\end{equation}
		hold. By the CLT \ref{thm:CLTforPoi}, \Cref{eq:alphaPoi} holds asymptotically, i.e. for sufficiently large $t,n$ and sufficiently small $d$, \eqref{eq:alphaPoi} holds exactly with some $\tilde{q}_{\alpha, n}^* = q_{\alpha,n}^* +o(1)$. Similarly, by the CLT \ref{thm:CLTforPoi} under $H_1$ we get \Cref{eq:betaPoi} with $q_{\alpha, n}^* := \sqrt{\V  {H_1}{T_n}}q_{\beta} + \E  {H_1}{T_n}$.
		For the quantile to be well-defined, we need to figure out when
		\[
		\sqrt{\V  {H_1}{T_n}}q_{\beta} + \E  {H_1}{T_n} = \sqrt{\V  {H_0}{T_n}}q_{1-\alpha} + \E  {H_0}{T_n} + o(1).
		\]
		Using previous calculations it holds that
		\begin{equation*}
		\E{H_1}{T_n} - \E{H_0}{T_n} 
		=
		\sum_{i=1}^n y_i\log(1+y_i)\lambda_{0i} 
		= 
		\sum_{i=1}^n\lambda_{0i}\left(y_i^2   + O(y_i^3)\right)
		\end{equation*}
		and
		\begin{equation}\label{eq:diffOfVarPoi}
		\sqrt{\V{H_0}{T_n}} = 
		\sqrt{\sum_{i=1}^n \lambda_{0i}\left(y_i^2 + O(y_i^3)\right)}
		= 	\sqrt{\V{H_1}{T_n} }
		.
		\end{equation}
		Thus, the quantile is well-defined if
		\begin{align}
		q_{1-\alpha}\sqrt{\V{H_0}{T_n}} - q_{\beta}\sqrt{\V{H_1}{T_n}}
		&=
		\E{H_1}{T_n} - \E{H_0}{T_n} + o(1)
		\Longleftrightarrow	\label{eq:expDiffGreaterThanSd}\\
		\sqrt{\sum_{i=1}^n \lambda_{0i}\left(y_i^2 + O(y_i^3)\right)}
		&=
		q_{1-\alpha} - q_{\beta}
		\Longleftrightarrow\nonumber\\
		\frac{\sqrt{t}d^2}{8}
		\sqrt{\int_0^1\frac{h''(x-x_0)^2}{h(x-x_0)}\,\mathrm{d}x}
		+o(d^4)
		&=
		q_{1-\alpha} - q_{\beta}
		= q_{1-\beta} - q_{\alpha}.\nonumber
		\end{align}
		Solving for $d$, we get the desired resolution relation \eqref{eq:resolutionPoi}.
	\end{proof}

	\section{Auxiliary proofs}\label{sec:auxProofs}
	
	\subsection{Proof that symmetrically placed signals is the hardest case asymptotically}\label{sec:symmetricallyPlacedHardest}
	
	\begin{proof}[Proof of \Cref{thm:symmetricallyPlacedHardest}]
		We first prove the statement for the \textbf{homogeneous Gaussian model}. Let
		\[
		\lambda :=  \left|x_0 - \frac{x_1 + x_2}{2}\right|.
		\]
		In general, using \eqref{eq:Delta}, $\mu_n$ as defined in \eqref{eq:mu_nDefinition} for large $n$ is equal to
		\begin{align*}
		\mu_n 
		&=
		\frac{t^2}{2n}
		\Bigg(
		\frac{\lambda^4}{4}\int_{0}^{1}(h'')^2
		+\frac{d^2\lambda^2}{8}\int_{0}^{1}(h'')^2
		+\frac{d^4}{64} \int_{0}^{1}(h'')^2
		+\lambda^2 \int_{0}^{1}(h')^2
		+\left( \frac{d^2}{4} + \lambda^2\right)\lambda \int_{0}^{1}h'' h'\\ 
		&+ O\left(\text{higher order terms} \right)
		\Bigg).
		\end{align*}
		Since the psf $h$ is even, it holds that $\int_{0}^{1}h'(x-0.5)\, h''(x-0.5)\,\mathrm{d}x= 0$. Considering $\mu_n$ as a function of $\lambda$, we find its minimum at $\lambda = 0 + O(d^2)$. Since
		\begin{align*}
		\text{under }&H_0: T_n\left(Y\right) \sim \mathcal N \left(-\mu_n, 2\mu_n\right), \\
		\text{under }&H_1: T_n\left(Y\right) \sim \mathcal N \left(\mu_n, 2\mu_n\right),
		\end{align*}
		we see that the case $x_0 = \frac{1}{2}(x_1 + x_2)$ is indeed the hardest to distinguish.
		
		The proof for the \textbf{variance stabilized Gaussian model} follows the same lines and is therefore omitted.
		
		For the \textbf{Poisson} model we have two cases to consider. Whenever $t \gg \sqrt{n}\log^8 n$ we can employ asymptotic equivalence and hence the result follows from the variance stabilized Gaussian model. If $t = n^{2-\delta}$ for some $\delta > 0$, we can prove a CLT also in the asymmetric case. The proof of the CLT is the same as previously (see the proof of \Cref{thm:CLTforPoi}) just for the ratio
		\[
		R:=\min_{j\in\{1,\ldots,n\},\; |a_j|\ne 0}\frac{\sqrt{\sum_{i=1}^{n}a_i^2\int_i h}}{|a_j|}
		\]
		we now use more terms in the series expansion (if $a_j=0$, then the corresponding summand in \eqref{eq:Ln} is zero). We have
		\begin{align*}
		y_i &= \frac{\lambda_{1i}-\lambda_{0i}}{\lambda_{0i}} = \frac{\int_i\Delta}{\int_i h}
		=
		\frac{\int_i h'}{\int_i h}\lambda
		+ \frac{1}{2}\frac{\int_i h''}{\int_i h}
		\left(\frac{d^2}{4} + \lambda^2 \right)
		+ o\left(\lambda^2\right)
		+ o\left(d^2    \right),\\
		a_i &= \log(1+y_i) 
		= y_i + O\left(y_i^2\right),\\
		\sqrt{\sum_{i=1}^{n}a_i^2\int_i h}
		&= \sqrt{\sum_{i=1}^{n}\log(1 + y_i)^2\int_i h}
		= \sqrt{\sum_{i=1}^n \left(y_i^2 + O(y_i^3)\right)\int_i h}\\
		&= \sqrt{\sum_{i=1}^{n} 
			\left(\frac{(\int_i h')^2}{\int_i h}\lambda^2 
			+ \frac{\int_i h'\int_i h''}{\int_i h}\lambda\, \frac{d^2}{4} 
			+ \frac{(\int_i h'')^2}{\int_i h}\frac{d^4}{64}\right)
			+ O\left(\lambda^3  \right) 
			+ o\left(        d^4\right) 
			+ o\left(\lambda d^2\right)},
		\end{align*}
		and thus $R= O(1)$ as before.
		The rest of the proof is the same as in \Cref{thm:CLTforPoi} and is therefore omitted.
		
		The calculation of the asymptotic resolution in the Poisson model essentially boils down to \Cref{eq:expDiffGreaterThanSd} stated here once more for convenience
		\begin{align}\label{eq:poiExpVarbB2}
		q_{1-\alpha}\sqrt{\V{H_0}{T_n}} - q_{\beta}\sqrt{\V{H_1}{T_n}}
		&=
		\E{H_1}{T_n} - \E{H_0}{T_n} + o(1)
		\Longleftrightarrow\\
		\sqrt{\sum_{i=1}^n \lambda_{0i}\left(y_i^2 + O(y_i^3)\right)}
		&=
		q_{1-\alpha} - q_{\beta}
		= q_{1-\beta} - q_{\alpha}.\nonumber
		\end{align}
		Hence, using the above calculations and $\lambda_{0i} = tp_{0i}$, \Cref{eq:poiExpVarbB2} is equivalent to
		\begin{align*}
		t\Bigg(\lambda^2\int_0^1\frac{(h')^2}{h} 
		+ \frac{d^4}{64}\int_0^1\frac{(h'')^2}{h}
		+ \lambda\Bigg(\frac{d^2}{4} &+ \lambda^2\Bigg)\int_0^1\frac{h'h''}{h}
		+ o\left(\lambda^3   \right)
		+ o\left(d^2\lambda\right)
		+ o\left(d^4       \right)\Bigg)\\ 
		&\hspace{25mm}=
		(q_{1-\beta}-q_{\alpha})^2.
		\end{align*}
		Since the psf $h$ is even, $h'$ is odd and $h''$ is even. Hence,
		\[
		\int_0^1\frac{h'(x-0.5)\, h''(x-0.5)}{h(x-0.5)}\,\mathrm{d}x = 0
		\]
		and thus the left hand side considered as a function of $\lambda$ attains its minimum at $\lambda = 0$. This implies that for given values of $\alpha$, $t$ and $d$, the power $1-\beta$ is the smallest when $\lambda = 0$, i.e. $x_0 = \frac{1}{2}(x_1 + x_2)$, is the most difficult alternative.
	\end{proof}
	
	\subsection{An integral approximation}\label{sec:proofs}
	
	\begin{lemma}\label{lem:summingRatio}
		Let $f:[0,1]\to\reals$ and $g:[0,1]\to\reals_{>0}$ be two absolutely continuous functions.
		Then
		\begin{equation*}
		\sum_{i=1}^{n}\frac{\left(\int_{(i-1)/n}^{i/n} f(x)\,\mathrm{d}x\right)^2}{\int_{(i-1)/n}^{i/n} g(x)\,\mathrm{d}x} 
		\xrightarrow{n\to \infty}
		\int_{0}^{1} \frac{f(x)^2}{g(x)}\,\mathrm{d}x <\infty.
		\end{equation*}
	\end{lemma}
	
	\begin{proof}
		Note that $f(x)^2/g(x)$ is absolutely continuous, and thus Riemann integrable.
		Using the mean value theorem we get
		\begin{align*}
		\sum_{i=1}^{n}\frac{\left(\int_{(i-1)/n}^{i/n}f(x)\,\mathrm{d}x\right)^2}{\int_{(i-1)/n}^{i/n}g(x)\,\mathrm{d}x}
		&=
		\frac{1}{n}\sum_{i=1}^{n}\frac{f(\xi'_i)^2}{g(\xi_i)} + o(1)
		=
		\frac{1}{n}\sum_{i=1}^{n}\frac{\left(f(\xi'_i)-f(\xi_i)+f(\xi_i)\right)^2}{g(\xi_i)} + o(1)\\
		&=
		\frac{1}{n}\sum_{i=1}^{n}\frac{(f(\xi'_i)-f(\xi_i))^2+2(f(\xi_i')-f(\xi_i))f(\xi_i)+f(\xi_i)^2}{g(\xi_i)} + o(1)
		\end{align*}
		with $\xi_i,\xi'_i\in[(i-1)/n,i/n]$. Now by continuity of $f$ it holds for all $1\le i \le n$ that
		\[
		|f(\xi_i)-f(\xi_i')|
		\le
		\max_{x\in \left[\frac{i-1}{n},\frac{i}{n}\right]}f(x)
		- \min_{x\in \left[\frac{i-1}{n},\frac{i}{n}\right]}f(x)\to 0,
		\]
		as $n\to\infty$. Thus, by Riemann integrability
		\[
		\sum_{i=1}^{n}\frac{\left(\int_{(i-1)/n}^{i/n}f(x)\,\mathrm{d}x\right)^2}{\int_{(i-1)/n}^{i/n}g(x)\,\mathrm{d}x}
		=
		\frac{1}{n}\sum_{i=1}^{n}\frac{f(\xi_i)^2}{g(\xi_i)} + o(1)\to \int_{0}^{1}\frac{f(x)^2}{g(x)}\,\mathrm{d}x,
		\]
		as $n\to\infty$.
	\end{proof}
	
	\subsection{Generalized testing problem with different weights}\label{sec:generalCase}
	
	Consider the generalized testing problem with the hypothesis
	\begin{subequations}\label{eq:H0vsH1general}
		\begin{equation*}
		H_0 : g(x) = \psf(x-x_0)
		\end{equation*}
		against the alternative
		\begin{equation*}
		H_1 : g(x) = q\,\psf(x-x_1) + (1-q)\,\psf(x-x_2),
		\end{equation*}
	\end{subequations}
	with $q\in (0,1)$ and $x_0 = q x_1 + (1-q) x_2$ fixed. The case considered in the main part of this paper corresponds to $q=1/2$.
	
	Assume w.l.o.g. that $x_2 \ge x_1$ and let $d=x_2 - x_1$. Note that $x_0-x_1 = (1-q)d$ and $x_0-x_2 = -qd$. Hence,   \Cref{eq:Delta} becomes
	\begin{align}
	\Delta(x-x_0) 
	&:= q h(x-x_1)+ (1-q) h(x-x_2)-h(x-x_0)\nonumber\\
	&= q\sum_{j=0}^2\frac{h^{(j)}(x-x_0)}{j!}(x_0-x_1)^j
	+
	(1-q)\sum_{j=0}^2\frac{h^{(j)}(x-x_0)}{j!}(x_0-x_2)^j-h(x-x_0)\nonumber\\
	&\hspace{8cm}+o\left((x_0-x_1)^2+(x_0-x_2)^2\right)\nonumber\\
	&= \frac{q(1-q)d^2}{2}\, h''(x-x_0) + o\left(d^2\right).\label{eq:DeltaGeneral}
	\end{align}
	
	\textbf{Homogeneous Gaussian model}\\
	\Cref{eq:mu_nDefinition} becomes
	\[
	\mu_n 
	= \frac{t^2}{2}\sum_{i=1}^n \left(p_{1i} - p_{0i}\right)^2
	= \frac{t^2}{2}\sum_{i=1}^n \left( \int_i\Delta \right)^2
	= \frac{t^2 d^4 q^2 (1-q)^2}{8n} \int_{0}^{1}\left(h''\right)^2  +  o\left(\frac{t^2 d^4}{n}\right).
	\]
	Just like previously, to have the type I error $=\alpha$ and type II error $=\beta$,we have to set $\mu_n = (q_{1-\alpha} - q_{\beta})^2/2$, see \eqref{eq:gaussType1error} and \eqref{eq:gaussType2error}. Hence, in this case the asymptotic resolution \eqref{eq:resolutionGauss} becomes
	\begin{equation*}
	d \asymp \frac{\sqrt{2}}{\sqrt{q(1-q)}}\sqrt{q_{1-\beta}-q_{\alpha}} \left(\int_0^1 \psf''\left(x-x_0\right)^2 \diff x\right)^{-1/4} t^{-1/2}\, n^{1/4}.
	\end{equation*}
	Thus, the alternative where the psfs have the same weights, i.e. $q=1/2$, is the \textit{easiest}, since this is the maximum of $q(1-q)$. Also note that as $q\to0$ or $q\to 1$, the resolution $d\to\infty$, as expected.
	\vspace{0.5mm}
	
	\textbf{Variance stabilized Gaussian model\\}
	Using \eqref{eq:DeltaGeneral} \Cref{eq:VSGtaylor} becomes
	\[
	\left(\sqrt{p_{1i}} - \sqrt{p_{0i}}\right)^2 
	=
	\int_i h \left(\sqrt{1 + \frac{\int_i \Delta}{\int_i h}} - 1 \right)^2
	=
	\frac{d^4q^2(1-q)^2}{16} \frac{\left( \int_i h'' \right)}{\int_i h} + o\left(\frac{d^4}{n}\right).
	\]
	Thus, the generalized equivalent of \eqref{eq:nu_n} is
	\[
	\nu_n =
	2t\sum_{i=1}^{n}\left(\sqrt{p_{1i}} - \sqrt{p_{0i}}\right)^2 =
	\frac{td^4q^2(1-q)^2}{8} \int_0^1\frac{ (h'')^2}{h} + o\left(td^4\right).
	\]
	As before, by \eqref{eq:vsgType1error} and \eqref{eq:vsgType2error} we have to set $\nu_n = (q_{1-\alpha} - q_{\beta})^2/2$ to have the type I error $=\alpha$ and type II error $=\beta$. Therefore, in this case the asymptotic resolution is
	\begin{equation}\label{eq:resolutionPoiGeneral}
	d \asymp \frac{\sqrt{2}}{\sqrt{q(1-q)}} \sqrt{q_{1-\beta} - q_{\alpha}} \left(\int_0^1 \frac{h''\left(x-x_0\right)^2}{h\left(x-x_0\right)} \diff x\right)^{-1/4} t^{-1/4}.
	\end{equation}
	
	\textbf{Poisson model\\}
	First of all, note that the proof in the asymptotic equivalence regime holds by the general VSG model proof above. As for the CLT regime, \Cref{eq:definitionOfy_iPoi} becomes
	\begin{align*}
	y_i 
	= \frac{\lambda_{1i}}{\lambda_{0i}} - 1
	= \frac{\lambda_{1i}-\lambda_{0i}}{\lambda_{0i}} 
	= \frac{\int_i\Delta}{\int_i h}
	&= \frac{q(1-q)}{2}\frac{\int_i h''}{\int_i h}d^2
	+ \frac{q(1-q)(1-2q)}{6}\frac{\int_i h'''}{\int_i h}d^3\nonumber\\
	&+ \frac{q(1-q)((1-q)^2-q(1-2q))}{24}\frac{\int_i h''''}{\int_i h}d^4
	+ o\left(d^4\right) 
	\end{align*}
	and the following terms $\E{H_0}{\statPoi}, \V{H_0}{\statPoi}, \E{H_1}{\statPoi}$ and $\V{H_1}{\statPoi}$ change accordingly. We skip these expressions due to their length and because they are not particularly insightful. However, it is clear that the CLTs under $H_0$ and $H_1$ still hold, just like in the symmetric alternative $q=1/2$ case. 
	
	The crux of the asymptotic resolution determination is \Cref{eq:expDiffGreaterThanSd} which in the general case reads
	\[
	\frac{q(1-q)\sqrt{t}d^2}{2}
	\sqrt{\int_0^1\frac{h''(x-x_0)^2}{h(x-x_0)}\,\mathrm{d}x}
	+o(d^4)
	= q_{1-\beta} - q_{\alpha}.
	\]
	Therefore, the asymptotic resolution is the same as in the general VSG model \eqref{eq:resolutionPoiGeneral}.

	\begin{remark}\label{rmk:symmetricHardestGeneralized}
		Note that the case $x_0 = q x_1 + (1-q) x_2$ (center of intensity) is the hardest to distinguish in the general testing problem \eqref{eq:H0vsH1general} \textit{for even psfs}; the proof easily follows from Section \ref{sec:symmetricallyPlacedHardest} by setting $\lambda =  x_0 - (q x_1 + (1-q) x_2)$.
	\end{remark}

	\section*{Acknowledgements}
	We gratefully acknowledge the support of the DFG, CRC 755 ``Nanoscale Photonic Imaging'', subproject A7, Cluster of Excellence 2067: Multiscale Bioimaging: From molecular medicine to networks of excitable cells (MBExC) and RTG 2088 ``Discovering structure in complex data: Statistics meets Optimization and Inverse Problems''. We are grateful to Alexander Egner and Jan Keller-Findeisen for helpful comments and discussions. Furthermore we thank two anonymous referees and an associate editor for constructive reports, which led to an improved presentation of the results.
	
	\bibliographystyle{apalike} %
	\bibliography{bibliography}     
	\vspace{1cm}
	
\end{document}